\documentclass[12pt]{article}
\usepackage{amsmath,amsfonts,amssymb}
\usepackage{alltt}

\usepackage{amsmath,amsfonts,ifthen,fullpage,enumerate}  
\usepackage{mathrsfs}
\usepackage{hyperref}
\usepackage{mathtools}
\usepackage{tikz-cd}
\usetikzlibrary{arrows.meta}
\tikzset{>={Latex[width=2mm,length=2mm]}}

\usepackage{colortbl}

\usepackage{enumerate}

\usepackage{amsmath,amsfonts,amssymb}
\usepackage{float}

\usepackage{makecell}  

\usepackage{graphicx}        

\usepackage{tikz-cd}

\usepackage{tcolorbox}
\def\greyboxvar#1#2
{\begin{tcolorbox} [boxrule=-5pt, boxsep=-0.1cm, colback=white!95!black, 
	width=#2] #1\end{tcolorbox}}

\def\greybox#1
{
	\begingroup
\setlength{\tabcolsep}{2pt} 
\renewcommand{\arraystretch}{1.3} 
	\begin{tabular}{c}
\cellcolor{gray!25}#1\end{tabular}
\endgroup
}

\usepackage[all,pdf]{xy}
\usepackage{lmodern,amssymb}   

\def\spam{\mathop{\rm span}\nolimits}
\def\gcd{\mathop{\rm gcd}\nolimits}
\def\lcm{\mathop{\rm lcm}\nolimits}

\newcommand\divides{\ensuremath{\bigm|}}

\def\rank{\mathop{\rm rank}\nolimits}

\def\diag{\mathop{\rm diag}\nolimits}

\def\pmat#1{\begin{pmatrix}#1\end{pmatrix}}
\def\mat#1{\begin{matrix}#1\end{matrix}}
\def\question#1{{\bf Question: }#1}
\def\question#1{}
\def\sperp{\text{\tiny$\perp$}} 

\def\cR{{\cal R}}

\def\CC{\mathbb{C}}

\def\HH{\mathbb{H}}

\def\ZZ{\mathbb{Z}}

\def\Cd{\C^d}

\def\Hd{\HH^d}

\def\C{\mathbb{C}}

\newtheorem{theorem}{Theorem}[section]
\newtheorem{corollary}{Corollary}[section]
\newtheorem{lemma}{Lemma}[section]
\newtheorem{example}{Example}[section]
\newtheorem{remark}{Remark}[section]
\newtheorem{proposition}{Proposition}[section]
\newtheorem{definition}{Definition}[section]

\newenvironment{proof}{{\noindent \it
Proof.}}{\hfill$\Box$\medskip}
\newif\ifdraft\def\draft{\drafttrue\hoffset=.8truecm\showlabeltrue
\def\comment##1{{\bf comment: ##1}}
\headline={\sevenrm \hfill \ifx\filenamed\undefined\jobname\else\filenamed\fi%
(.tex) (as of \ifx\updated\undefined???\else\updated\fi)
 \TeX'ed at {\hour\time\divide\hour by 60{}%
\minutes\hour\multiply\minutes by 60{}%
\advance\time by -\minutes
\the\hour:\ifnum\time<10{}0\fi\the\time\  on \today\hfill}}
}

\def\inpro#1{\langle#1\rangle}
\def\ip#1{\langle\kern-.28em\langle#1\rangle\kern-.28em\rangle_\nu}

\def\cD{{\cal D}}

\def\openR{{{\rm I}\kern-.16em {\rm R}}}

\let\ga\alpha

\let\gb\beta

\let\gl\lambda

\let\gs\sigma

\let\go\omega
\let\gO\Omega
\let\ga\alpha

\let\gb\beta

\let\gs\sigma
\def\inpro#1{\langle#1\rangle}

\def\Stab{\mathop{\rm Stab}\nolimits}

\def\ker{\mathop{\rm ker}\nolimits}

\def\Iff{\hskip1em\Longleftrightarrow\hskip1em}
\def\Implies{\hskip1em\Longrightarrow\hskip1em}

\def\formeq{\the\sectionno.\the\equationno}  
\def\elabel#1/#2/#3/{\global\advance\equationno by 1 %
\ifx#1\empty\else\emember#1%
\ifshowlabel\marginal{\string#1}\fi\fi%
\ifmmode\eqno{#3(\formeq#2)}\else#3\formeq#2\fi} 

\def\makeblanksquare#1#2{
\dimen0=#1pt\advance\dimen0 by -#2pt
      \vrule height#1pt width#2pt depth0pt\kern-#2pt
      \vrule height#1pt width#1pt depth-\dimen0 \kern-#1pt
      \vrule height#2pt width#1pt depth0pt \kern-#2pt
      \vrule height#1pt width#2pt depth0pt
}

\title{\bf 
The lattice of normal reflection subgroups of an irreducible reflection group
}

\author{Shayne Waldron\\ 
 \\
Department of Mathematics \\ University of Auckland\\
Private
Bag 92019, Auckland, New Zealand\\
e--mail: waldron@math.auckland.ac.nz}

\begin{document}

\maketitle 

\begin{abstract}

The 
reflection subgroups of a reflection group have
a natural lattice structure given by the reflections that they contain.
By considering the 
	conjugation action 
	orbits of the reflection subgroups for a given root line,
we are able to give an essentially combinatorial way to calculate the lattice 
of all the normal reflection subgroups of a given (finite irreducible) reflection 
group, and natural generators for them.
Moreover, we observe 
	that every complex reflection group is a normal subgroup of 
the unique maximal reflection 
group which shares its collineation group. 
Hence, we are able to present the 
Shephard-Todd classification of the complex reflection groups
as a collection of maximal reflection groups, together with appropriate 
	(collineation preserving) normal reflection subgroups.

We investigate the quotients by the normal reflection subgroups, which are known to be
reflection groups.
We also consider the action of the collineation group on some appropriate small systems of lines,
and how these results extend 
		to quaternionic reflection groups.
Some novel techniques are introduced, including 
	the notion of a ``hidden reflection'', a combinatorial-geometric 
description of the reflection subgroups and the size of their
	conjugacy class,
	and the role played by 
the abelianisation of the reflection group.

\end{abstract}

\bigskip
\vfill

\noindent {\bf Key Words:}
complex reflection groups,
quaternionic reflection groups,
collineation groups,
normal reflection subgroups,
roots of reflections,
lattices of subgroups,
parabolic subgroups,
reflections, hidden reflections,
maximal reflection groups.
SICs (equiangular lines),
Shephard-Todd classification,
abelianisation of a group,
dominos

\bigskip
\noindent {\bf AMS (MOS) Subject Classifications:}
primary
20C25, \ifdraft Projective representations and multipliers \else\fi
20F55, \ifdraft Reflection and Coxeter groups (group-theoretic aspects) \else\fi
20G20, \ifdraft Linear algebraic groups over the reals, the complexes, the quaternions \else\fi
51F15, \ifdraft Reflection groups, reflection geometries [See also 20H10, 20H15] {For Coxeter groups, see 20F55} \else\fi
\quad
secondary
13A50, \ifdraft Actions of groups on commutative rings; invariant theory [See also 14L24] \else\fi
15B33, \ifdraft (Matrices over special rings (quaternions, finite fields, etc.) \else\fi
51M15. \ifdraft Geometric constructions in real or complex geometry \else\fi

\vskip .5 truecm
\hrule
\newpage

\section{Introduction}

A reflection group $G$ is a group which is generated by reflections
(see \cite{C76}, \cite{C80}, \cite{BMR98} \cite{LT09}, \cite{B10}, \cite{P23}).
Therefore, its reflection subgroups 
(subgroups which are reflection groups)
naturally 
form a lattice with the order relation given by inclusion of the reflections
in the subgroups, and maximal element $G$ and minimal element $1$ 
(generated by the empty set of reflections). 
Since the intersection of two reflection subgroups may not be a 
reflection group, the meet in this lattice 
may or may not be the intersection of the subgroups,
i.e., it may not be a sublattice of the lattice of all subgroups of $G$
(though the inclusions and the join are the same).

We are interested in the (finite irreducible) complex and quaternionic 
reflection groups, which were classified by Shephard and Todd \cite{ST54}
and Cohen \cite{C80}. These groups have all been implemented in 
computer algebra systems such as Magma, and so their reflection 
subgroups can be calculated \cite{T12}, \cite{DPR13}, \cite{RS26}. 
Two classes of reflection subgroups have
been studied in particular:
\begin{itemize}
\item The parabolic subgroups, i.e., those which fix pointwise a subset of the
vector space that they act on. 
It was proved that these are reflection groups by Steinberg \cite{S64} 
in the complex case, and by \cite{BST23} in the quaternionic case.
The nontrivial parabolic subgroups of $G$ are not normal 
and are of lower rank (Proposition \ref{Parabolicnotnormal}).
\item The normal reflection subgroups. In \cite{BBR02} 
(see 
		\cite{AW23}), it
was shown that the quotient of a complex reflection group by a normal 
reflection subgroup gives a reflection group. 
If $G$ is a primitive complex reflection group, then its
nontrivial normal reflection subgroups are irreducible and 
		hence of full rank (Theorem \ref{Galatticethm}).
\end{itemize}

In this paper, 
we study the lattice of normal reflection subgroups of a finite complex or 
quaternionic irreducible reflection group. Key properties of a reflection 
group $G$ that we consider are 
\begin{itemize}
\item The collineation group, i.e., the action of $G$ on lines.
\item The reflection subgroups of $G$ for a given root, i.e., 
	the maximal rank one parabolic subgroups, up to conjugacy in $G$.
\end{itemize}
This information is a little more than knowing the order of the reflection group and its centre, 
and how many reflections of each order it contains -- properties
which are usually tabulated 
(explicitly or implicitly) 
in lists of reflection groups.
From this information, we are led to our main insights:
\begin{itemize}
\item The lattice of normal reflection subgroups is completely 
determined by the orbits of the reflection subgroups for the roots. 
This provides a natural indexing for the subgroups,
which can be calculated, together with generators, in 
an essentially combinatorial way 
(Lemma \ref{normalrefslatticelemma}, Theorem \ref{Galatticethm}).
\item Every reflection group is a normal subgroup of a unique maximal
	reflection group which shares its collineation group.
Consequently, the classification of reflection groups can be 
described as the maximal reflection groups, together with their
collineation preserving normal reflection subgroups
(Lemma \ref{normalsgmaxreflectgp}, Theorem \ref{complexrefgpclass}).
\end{itemize}
Much of the paper is devoted to calculating the maximal reflection groups and
their collineation preserving normal reflection subgroups. This 
splits into the complex case (easier, with more normal reflection subgroups)
and the quaternionic case (a little more involved, with fewer normal reflection
subgroups). Within these cases there are the primitive groups (finitely many of them)
and the imprimitive groups (infinite families).

We will assume some basic familiarity with reflection groups, and in particular
the Shephard-Todd classification of the complex reflection groups into
three infinite families of imprimitive groups $G(m,p,d)$ 
and the primitive groups $G_4,G_5,\ldots,G_{37}$ (see \cite{LT09}).

\section{Reflections and their orbits under conjugation}

A {\bf reflection} $g$ on $\Cd$ or $\Hd$ is an invertible linear map which pointwise 
fixes a subspace of dimension $d-1$ (a hyperplane), is not the identity, and hence is characterised by
\begin{equation}
\label{rankcondition}
\rank(I-g) =1.
\end{equation}
It is immediate from (\ref{rankcondition}) that
\begin{itemize}
\item Conjugation within the general linear group preserves reflections.
\end{itemize}
Thus, we may assume, without loss of generality, that finite reflection 
groups are unitary.
A unitary reflection
can be specified by the hyperplane that if fixes, or its orthogonal 
complement, which is called the {\bf root line}. Any nonzero vector $a$ in the
root line is called a 
{\bf root} ({\bf vector}). 
A (unitary) reflection $g$ is therefore 
determined by a root vector $a$ and its action on the root line, i.e., 
$$ ga = a\xi, $$
where $\xi$ is a unit scalar in $\CC$ or $\HH$, respectively.
For reflections on $\Hd$, we treat $\Hd$ as a right vector space ($\HH$-module),
and (in both cases) the above reflection is given by the formula
\begin{equation}
\label{raxidefn}
g=r_{a,\xi}:= I-{a(1-\xi)a^*\over\inpro{a,a}},
\end{equation}
with $(a,\xi)$ called a {\bf root} for the reflection.
Clearly a reflection $g=r_{a,\xi}$ has order $n$ if and only if $\xi$ is 
a primitive $n$-th root of unity, i.e., $n$ is the smallest power with
$\xi^n=1$.

Throughout, $G$ will be a finite reflection group.
Associated with a reflection in $G$, its root line, or a root $a$ for it, is
the subgroup of all reflections in $G$ with that root line (and the identity), i.e.,
\begin{equation}
\label{Radefn}
	R_a = R_{a,G} :=\{ r_{a,\xi} : r_{a,\xi} \in G\},
\end{equation}
which we call the {\bf reflection subgroup} for the root $a$, reflection $r_{a,\xi}$, root line $a\HH$, etc.
These are the maximal parabolic subgroups of rank $1$
(they fix the hyperplane $a^\perp$), 

We also define the following subgroups of the unit scalars under multiplication
\begin{equation}
\label{HaKadefn}
H_a =H_{a,G} := \{ \gl\in U(\HH): r_{a,\gl}\in G\}, 
\qquad K_a = K_{a,G} := \{ \gl\in U(\HH): a\gl \in Ga\},
\end{equation}
where $H_a$ is a normal subgroup of $K_a$. 
We observe that $H_a\to R_a:\gl\mapsto r_{a,\gl}$ is an isomorphism,
and so the elements of $R_a$ can be indexed
\begin{equation}
\label{RadefnII}
	R_a = R_{a,G} = \{r_{a,\gl}:\gl\in H_{a,G}\}. 
\end{equation}

The (unitary) matrices of the group $G$ act on the reflection subgroups for the roots via
conjugation
\begin{equation}
\label{raxi-conjugation}
gr_{a,\xi} g^{-1}= r_{ga,\xi} \Implies
R_a^g:= gR_a g^{-1} = R_{ga}, \qquad g\in G.
\end{equation}
A reflection group $G$ is therefore uniquely determined by the conjugacy
classes 
$$ R_a^G :=\{ R_a^g : g \in G\}. $$
of the reflection subgroups for its roots (root lines). 
We observe that
\begin{itemize}
\item If $H$ is a reflection subgroup of $G$, and $a$ is a root of
a reflection in $G$, then
$R_{a,H}$ is a subgroup of $R_{a,G}$, possibly the identity 
		(when $H$ has no reflections with root $a$).
\item Moreover, if $H=N$ is normal in $G$, then 
$$ R_{a,N} \subset N \Implies g(R_{a,N})g^{-1}\subset gNg^{-1}=N, 
\qquad g\in G, $$
so that the reflections $ R_{a,N}^G$ are in $N$, where
		$\hat R_a=R_{a,N}$ is a subgroup of $R_a=R_{a,G}$.
\end{itemize}

This structure characterises the normal reflection subgroups.

\begin{lemma}
\label{normalrefslatticelemma}
Let $G$ be a reflection group, and $R_{a_1}^G,\ldots,R_{a_m}^G$ 
be the orbits of the reflection subgroups for its roots.
Then the normal reflection subgroups $N$ of $G$ are given by 
\begin{equation}
\label{NormalhatRa}
	N=\inpro{\cup_j\hat R_{a_j}^G}, \qquad
\hbox{where each $\hat R_{a_j}$ is a subgroup of $R_{a_j}$},
\end{equation}
and these form a sublattice of the lattice of reflection subgroups, 
with the meet and join given by 
\begin{align}
\inpro{\cup_j\hat R_{a_j}^G} \wedge \inpro{\cup_j\tilde R_{a_j}^G}
&=\inpro{\cup_j(\hat R_{a_j}\cap \tilde R_{a_j})^G},
\label{normalrefslatticeopsI} \\
\inpro{\cup_j\hat R_{a_j}^G}\vee \inpro{\cup_j\tilde R_{a_j}^G}
&=\inpro{\cup_j\hat R_{a_j}^G} \inpro{\cup_j\tilde R_{a_j}^G}
= \inpro{ \cup_j(\hat R_{a_j}\tilde R_{a_j})^G}.
\label{normalrefslatticeopsII}
\end{align}
Moreover, the following subsets form sublattices
\begin{enumerate}[\rm (i)]
\item The subgroups with a given collineation group.
\item The primitive subgroups.
\item The imprimitive subgroups.
\item The reducible subgroups.
\item The subgroups of a given rank (or a sequence of successive ranks).
\end{enumerate}
\end{lemma}

\begin{proof}
Clearly $N=\inpro{\cup_j\hat R_{a_j}^G}$ is a normal reflection subgroup, 
since it is generated by a set of reflections which are closed under conjugation by $G$,
and we observed that all normal reflection subgroups must have this form.

Since the reflections 
in $\inpro{\cup_j\hat R_{a_j}^G} \wedge \inpro{\cup_j\tilde R_{a_j}^G}$
are $\cup_j(\hat R_{a_j}\cap \tilde R_{a_j})^G$, we obtain the 
	formula for the meet.
For the join, we first observe that since $\hat H= \inpro{\cup_j\hat R_{a_j}^G}$
and $\tilde H=\inpro{\cup_j\tilde R_{a_j}^G}$ are normal subgroups of $G$, so
is their product. Thus it suffices to show that their product is
the reflection subgroup
$$ H= \inpro{\cup_j(\hat R_{a_j}\tilde R_{a_j})^G}. $$
Since $\hat R_{a_j}$ and $\tilde R_{a_j}$ are subgroups of
$\hat R_{a_j}\tilde R_{a_j}$, we have that $\hat H\tilde H\subset H$.
On the other hand, $\hat H\tilde H$ contains the reflections which
generate $H$, so that $\hat H\tilde H=H$.

The given subsets form sublattices since the property defining them 
increases or decreases as one moves down the lattice, 
e.g., the collineation group of a subgroup 
(see \S\ref{collineationsection}) is contained within the collineation group of the group, and the rank of a subgroup is less than or equal the rank of the group.
\end{proof}

In the most general setting, i.e., for quaternionic reflection groups, we index
the normal subgroups of $G$ by the reflection subgroups that define them, i.e., 
\begin{equation}
\label{Quaternionsubgrouplabels}
N=\inpro{\cup_j\hat R_{a_j}^G} = G^{(\hat R_{a_1},\ldots, \hat R_{a_m})}. 
\end{equation}
We will say that $\hat R_{a_j}^G$ is a {\bf split orbit} of $N$ if it
is not an $N$-orbit, i.e.,
\begin{equation}
\label{splitorbitdef}
\hat R_{a_j}^N\ne \hat R_{a_j}^G.
\end{equation}

At the outset, 
it is not clear whether a normal reflection subgroup $N$ can have more than one label,
i.e., come from different choices of subgroups $\{\hat R_{a_j}\}$.
We will see that normal reflection subgroups of complex reflection groups do have a unique label,
whereas for quaternionic reflection groups this is not always the case
(Examples \ref{Pgroups-example} and \ref{typeOexample}).


If $G$ is a complex reflection group, 
then $R_{a_j}\cong H_{a_j}\subset\CC^*$ is a cyclic group of order $k_j$,
and so $R_{a_j}$ and its subgroups can be indexed by their orders 
$\ga_j$, which divide $k_j$, i.e., 
\begin{equation}
\label{Complexsubgrouporderlabels}
N=\inpro{\cup_j\hat R_{a_j}^G} = G^{(\ga_1,\ldots,\ga_m)}, \qquad
\ga_j:=|\hat R_{a_j}|,\quad \ga_j\divides k_j,
\end{equation}
where 
\begin{equation}
\label{hatRjformula}
	\hat R_{a_j} 
	= (R_{a_j})^{k_j/\ga_j}
= \{ g^{k_j/\ga_j} : g\in R_{a_j}\}
	= \{ g\in R_{a_j} : g^{\ga_j}=1 \}.
\end{equation}
The lattice operations (\ref{normalrefslatticeopsI}) and (\ref{normalrefslatticeopsII}), 
then take the 
following simple and insightful form
\begin{equation}
\label{Galatticeops}
G^\ga \wedge G^\gb = G^{\gcd(\ga,\gb)}, 
\qquad G^\ga \vee G^\gb = G^\ga G^\gb = G^{\lcm(\ga,\gb)},
\end{equation}
where the gcd (greatest common divisor) and lcm (least common multiple) 
are taken coordinatewise.
The uniqueness of the label $(\ga_1,\ldots,\ga_m)$ of 
(\ref{Complexsubgrouporderlabels}) follows from the facts
\begin{itemize}
\item If $G$ is a complex reflection group generated by the reflections $\cR$,
then any reflection in $G$ is conjugate to a power of a reflection in $\cR$
		(Lemma 4.11 (iii) \cite{C76}).
\item Each $\hat R_{a_j}$ is a group, and hence is closed under taking powers.
\end{itemize}
In effect,
the reflections in $G^\ga$ are precisely $\cup_j\hat R_{a_j}^G$ 
(see Example \ref{Galphareflects} for details). 
The first fact is not true for quaternionic reflection groups 
(see Examples \ref{Pgroups-example} and \ref{typeOexample}).

Therefore, for complex reflection groups, we have a very simple
description of the lattice of normal reflection subgroups.

\begin{theorem}
\label{Galatticethm}
Let $G$ be an irreducible complex reflection group,
and $R_{a_1}^G,\ldots R_{a_m}^G$, with $k_j=|R_{a_j}|$,
be the orbits of of the reflection subgroups for its roots.
Then the lattice of normal reflection subgroups of $G$ is isomorphic 
to the lattice of divisors of $(k_1,\ldots,k_m)$, with the
operations (\ref{Galatticeops}).
In particular, there are 
$$ \gs_0(k_1)\gs_0(k_2) \cdots \gs_0(k_m) $$
normal reflection subgroups, where $\gs_0(n)$ is the number of divisors of $n$.

Moreover, if $G$ is primitive, then every nontrivial normal subgroup is irreducible.
\end{theorem}

\begin{proof}
We have observed that the reflections in $N=G^{(\ga_1,\ldots,\ga_m)}$ of
(\ref{Complexsubgrouporderlabels}) are precisely $\cup_j\hat R_{a_j}^G$,
and so each normal reflection subgroup has a unique label.
Thus we have the lattice isomorphism claimed. 

If $N$ is a normal subgroup of $G\subset GL(\Cd)$, then it follows from Clifford's theorem 
that $\Cd$ decomposes into a direct sum of $N$-invariant subspaces which are 
conjugate to each other under the action of $G$. 
Thus if $G$ is primitive, then there
must be a single summand, i.e., $N$ must be irreducible.
\end{proof}

To describe the indexing (\ref{NormalhatRa}) of the lattice 
in Lemma \ref{normalrefslatticelemma} (and in Theorem \ref{Galatticethm}),
it suffices to know the reflection subgroups $R_{a_j}\cong H_{a_j}$ 
up to isomorphism. To capture this, and the number of conjugates 
$n_j$ of $R_{a_j}$ in $G$, we will say that the reflection group $G$ 
has {\bf reflection orbits} or {\bf reflection type}
\begin{equation}
\label{orbitlabelling}
n_1 H_{a_1}, \ldots , n_m H_{a_m},
\end{equation}
where $H_{a_j}$ can be written as an abstract group, or as a
subgroup of $U(\HH)$. The orbit size $n_j$ equals the size of the
orbit of the root line given by $a_j$ under the action of $G$.

The number of reflections of each possible order on each of the $m$ orbits 
is easily calculated from the reflection type, e.g., 
if $G$ is a complex reflection group, then its type is given by cyclic groups
\begin{equation}
\label{orbitlabellingcyclic}
n_1 C_{k_1}, \ldots , n_m C_{k_m},
\end{equation}
and
\begin{itemize}
\item The number of root lines (reflection hyperplanes) 
for $G$ is $n_1+\cdots+n_m$. 
This is equal to the sum of the codegrees of $G$ plus its rank.
\item The number of reflections of order $b\divides k_j$ on 
	the $n_j C_{k_j}$ orbit is $n_j\varphi(b)$.
\item The number of reflections of order $b$, and 
	the total number of reflections are
$$ \sum_{1\le j\le m\atop b | k_j} n_j\varphi(b), 
		\qquad \sum_{1\le j\le m} n_j(k_j-1).
$$
\end{itemize}
The reflections in $G^{(\ga_1,\ldots,\ga_m)}$ can be counted in the same way,
i.e.,
\begin{equation}
\label{Galphareflects}
\hbox{\# reflections in $G^{(\ga_1,\ldots,\ga_m)}$}
= \sum_{1\le j\le m} n_j(\ga_j-1).
\end{equation}
The reflection type gives more information than 
listing the number of reflections of each order (which is commonly done), 
e.g., the number of orbits.
For example, the primitive Shephard-Todd groups $G_{12}$ and 
$G_{13}$ have $12$ and $18$ reflections of order $2$, which 
is often written ${12}^{12}$ and ${12}^{18}$, and reflection types
$12C_2$ (one orbit) and $6C_2,12C_2$ (two orbits).

We give now give a simple illustrative example.

\begin{example}
\label{G212example}
Consider the imprimitive irreducible reflection group $G:=G(2,1,2)$,
which is generated by the reflections
$$  \pmat{0&1\cr1&0}, \qquad \pmat{1&0\cr0&-1}. $$
This real reflection group is conjugate to $G(4,4,2)$, 
	and is isomorphic to the dihedral group of order $8$ (see \cite{W26}).
It has four reflections of order $2$, which lie in two orbits
$$ \bigl\{\pmat{1&0\cr 0&-1},\pmat{-1&0\cr 0&1}\bigr\}, \qquad
	\bigl\{\pmat{0&1\cr 1&0},\pmat{0&-1\cr-1&0}\bigr\}. $$
and so $G$ has reflection type $2C_2,2C_2$, with the two reflection subgroup orbits being
\begin{equation}
\label{G212reflectionsgs}
\bigl\{\inpro{\pmat{1&0\cr 0&-1}},\inpro{\pmat{-1&0\cr 0&1}}\bigr\}, \qquad
\bigl\{\inpro{\pmat{0&1\cr 1&0}},\inpro{\pmat{0&-1\cr-1&0}}\bigr\}.
\end{equation}
Hence $G$ has four normal reflection subgroups, with the two proper ones being of order $4$
\begin{center}
\begin{tikzpicture}
    \matrix (A) [matrix of nodes, row sep=1.0cm, column sep = -1.2 cm]
    {
	    & $G^{(2,2)}=G$ & \\
	    $G^{(2,1)}=\inpro{\pmat{1&0\cr 0&-1},\pmat{-1&0\cr 0&1}}$ &&
	    $G^{(1,2)}=\inpro{\pmat{0&1\cr 1&0},\pmat{0&-1\cr-1&0}}$ \\
	    & $G=G^{(1,1)}=1$ & \\
    };
    \draw (A-1-2)--(A-2-1);
    \draw (A-1-2)--(A-2-3);
    \draw (A-2-1)--(A-3-2);
    \draw (A-2-3)--(A-3-2);
\end{tikzpicture}
\end{center}
We observe that $G^{(2,1)}\ne1$ is reducible, and so the condition that $G$ be 
	primitive in the second part of Theorem \ref{Galatticethm} is necessary.

There are five other subgroups of $G$, i.e., the normal subgroup of order $4$ given by
$$ \inpro{\pmat{-1&0\cr 0&-1},\pmat{0&-1\cr1&0}}, $$
which is not a reflection group, and the four reflection subgroups of
	(\ref{G212reflectionsgs}) of order two (which are not normal), for the roots $(0,1),(1,0),(1,-1),(1,1)$ of $G$, respectively.
\end{example}

\section{Collineations and maximal reflection groups}

We now consider how groups of matrices, such as
reflection groups, act on lines, i.e., one-dimensional subspaces. 
This leads to the notion of a ``collineation group'' \cite{B17}.

\subsection{Collineations}
\label{collineationsection}

Let $G\subset U(\Cd)$ be a group of complex unitary matrices 
(the quaternionic case will be considered later).
Then the action of $g\in G$ on a (complex) line $\CC a$ ($a\in\Cd$,  $a\ne 0$)
is the same as that of any nonzero scalar multiple $\gl g$ of $g$, i.e.,
\begin{equation}
\label{actiononClines}
\gl g(\CC a)= \CC \gl ga =\CC ga = g(\CC a). 
\end{equation}
The group $G$ factored by the (unit) scalar matrices in $G$ is 
called 
the {\bf collineation group} 
of $G$, and a matrix defined 
up to a scalar multiple is called 
a {\bf collineation} matrix.
We will denote the collineation group by $G_{\rm coll}$.

In view of (\ref{actiononClines}), there is a natural faithful action of a collineation group 
on the set of lines in $\Cd$
(see Lemma \ref{collinS4} for an example of a faithful action on a finite set of lines).
By Schur's lemma, we have
\begin{itemize}
\item If $G$ is irreducible, then its (normal) subgroup of scalar 
	matrices is its centre $Z(G)$, and so its collineation 
		group is $G_{\rm coll}=G/Z(G)$.
\end{itemize}
Collineation groups with the same collineation matrices 
(matrices up to a scalar multiple) are considered to be the
same. In the interests of giving a canonical form, there are
two approaches that have been taken:
\begin{itemize}
\item Consider subgroups of $U(\Cd)$ that contain all the unit scalar matrices
	(and factor by these).
\item Consider subgroups of the special unitary group $SU(\Cd)$ that contain
	the scalar matrices of determinant $1$, i.e., those given by the
		$d$-th roots of unity (and factor by these).
\end{itemize}
In light of the latter (every matrix can be scaled 
to have determinant $1$), the classification of finite collineation groups
follows from that of
the finite subgroups of $SU(\Cd)$.

For groups $G\subset U(\Hd)$ of quaternionic matrices, the
noncommutativity of $\HH$ implies that the analogue of 
(\ref{actiononClines}), i.e.,
\begin{equation}
\label{actiononHlines}
\gl g(\HH a) = g(\HH a),
\end{equation}
does not hold, unless $\gl$ is real, i.e., $\gl=\pm1$. Since the
centre of $M_d(\HH)$ is the real scalar matrices, we can 
define collineation groups and collineations as before, where the 
scalar matrices are taken to be real. 
We observe that under the correspondence with (complex) symplectic matrices given by
\begin{equation}
\label{C2dHdcorres}
\pmat{A&-B\cr\overline{B}&\overline{A}}\in M_{2d}(\CC)
\Iff A+Bj  \in M_d(\HH),
\end{equation}
the only scalar symplectic unitary matrices are precisely the real ones, i.e., $\pm I$.

\subsection{Hidden reflections and maximal reflection groups}

We will say that a complex collineation matrix is a {\bf reflection} if 
some scalar multiple of it is a reflection, and that a matrix is a 
{\bf hidden reflection} if some scalar multiple of it is a reflection,
but it is not a reflection.
In other words:
\begin{itemize}
\item
The collineation given by a matrix is a reflection if and only if
the matrix is reflection or a hidden reflection.
\end{itemize}
Reflections and hidden reflections are easily recognised by 
their spectral structure:
\begin{itemize}
\item A matrix $g\in U(\Cd)$ is a reflection or a hidden reflection if and only 
if it has exactly two eigenvalues $\gl_1$ and $\gl_2$ with multiplicities $1$ and $d-1$.
		It is a reflection when $\gl_2=1$ (or $\gl_1=\gl_2=1$, when $d=2$).
\item For $d=2$, there are two scalar multiples of
	a hidden reflection that give a reflection.
\item
For $d\ge3$, 
		there is a unique scalar multiple of a hidden reflection $g$ which is a 
reflection, i.e., $\gl_2^{-1} g$, which we call the {\bf reflection given by}
$g$. Furthermore, if $g$ has finite order, then $\gl_2^{-1}$ is a root
		of unity.
\end{itemize}
We can now make a key definition:

\begin{definition}
The {\bf maximal reflection group} for a complex reflection group $G$ is
the reflection group generated by its reflections and the reflections given
by its hidden reflections.
\end{definition}

We will also say that a reflection group is {\bf maximal} if it
equals its maximal reflection group, i.e., it contains no hidden
reflections. 
The identity subgroup is (technically) a maximal reflection group. We will not
fuss over this, e.g., presenting the lattice of maximal reflection subgroups without it.
A key observation is the following:
\begin{itemize}
\item If $g$ is a hidden reflection of a complex reflection group $G$, then
the reflection group generated by $G$ and the reflection $\gl_2^{-1} g$ 
is the group generated by the reflections in $G$ and the scalar matrix
		$\gl_2 I$, and so it has the same collineation group as $G$.
\end{itemize}
This implies that the number of reflections plus the number of hidden 
reflections of $G$ equals the number of reflections in its maximal reflection group.

\begin{example} The Shephard-Todd group $G_{12}\subset U(\CC^2)$ of order $48$
(with $12$ reflections) has 
maximal reflection group $G_{11}$ of order $576$, which has
$46$ reflections. Therefore, $G_{12}$ has $12$ reflections
and $34$ hidden reflections.
\end{example}

For reflection groups of rank $2$, every nonscalar matrix
is either a reflection or a hidden reflection, and so the
maximal reflection groups can be identified by counting
their reflections.

\begin{proposition}
\label{maximalrefcount}
An irreducible reflection group $G\subset U(\CC^2)$ is a maximal
reflection group if and only if it has
	$2(|G|/|Z(G)|-1)$ reflections.
\end{proposition}

\begin{proof} Since a maximal reflection group has no hidden reflections, 
	every nonscalar matrix in $G$ must be a reflection, with
	each nontrivial collineation in $G/Z(G)$ giving exactly two reflections.
\end{proof}

\begin{example} 
\label{Maximalonedim}
All the reflection groups in $U(\CC^1)$ are maximal. 
They are given by the $m$-th roots of unity, i.e., the cyclic
	group of order $m$, and we will use the notation
	$$ C_m = G(m,1,1). $$
The reducible reflection groups in $U(\CC^2)$ 
must be a product of cyclic groups.
When $C_1$ is a factor, one obtains a group generated
by a single reflection of order $m$, which we also denote by $C_m$, as we have
	done for the reflection type.
	The maximal reflection group
	for $C_{a}\times C_{b}$ is 
	$$ C_m\times C_m, \qquad m:={\rm lcm}(a,b). $$
These are all the maximal reducible reflection groups in $U(\CC^2)$.
\end{example}

\begin{lemma}
\label{normalsgmaxreflectgp}
Every complex reflection group $G$ is a normal (collineation preserving) 
subgroup of its maximal reflection group.
\end{lemma}

\begin{proof} 
We have already observed that a reflection group $G$ is
a subgroup of its maximal reflection group, which has the same collineation
group. Since the maximal reflection group of $G$ has elements of the form $\gl g$, 
where $\gl$ is a nonzero scalar and $g\in G$, the calculation
$$ (\gl g) G (\gl g)^{-1} = (\gl\gl^{-1}) g G g^{-1} =G, $$
shows that $G$ is normal in its maximal reflection group.
\end{proof}

\begin{theorem} (Classification) 
\label{complexrefgpclass}
The irreducible complex reflection groups are the collineation 
group preserving normal reflection subgroups of the maximal 
irreducible reflection groups.
\end{theorem}

\begin{proof}
In view of Lemma \ref{normalsgmaxreflectgp}, 
it suffices to observe that a complex reflection group is irreducible
	if and only its maximal reflection group is, since these 
	groups have the same matrices up to a scalar.
\end{proof}

To put Theorems \ref{Galatticethm} and \ref{complexrefgpclass} into full effect, we will
go
through the Shephard-Todd classification and 
\begin{itemize}
\item Determine the (primitive and imprimitive) maximal 
	irreducible complex reflection groups.
\item
	Find all the collineation preserving normal reflection
	subgroups of the maximal reflection groups (these have the same primitivity).
\item Explore how the indexing of (\ref{Complexsubgrouporderlabels}) 
relates to the structure of these normal reflection subgroups and their quotients.
\end{itemize}
Before doing this, we consider the normal reflection subgroups of
$G_{11}$ in great detail, including
showing how Theorem \ref{Galatticethm} allows for the 
systematic calculation of all normal reflection subgroups, 
and natural generators for them.

\begin{remark}
\label{primitivequatcentres}
The collineation group seems to play less of a role for the quaternionic 
reflection groups, e.g.,
all the primitive quaternionic reflection groups have centre $\inpro{-1}$
(see Table \ref{PrimitiveQuatOrbits}), and hence they have no hidden reflections.
\end{remark}

\section{The Shephard-Todd group $G_{11}$ and the SIC}

Here we consider the maximal reflection group $G_{11}\subset U(\CC^2)$, 
its subgroups, and their generators and action on a set of four equiangular lines.

\subsection{Generators for $G_{11}$ and equiangular lines}

A set of four equiangular lines in $\CC^2$ (known as a SIC) is given by an orbit of
a (fiducial) vector/line $v\in\CC^2$, e.g., 
\begin{equation}
\label{fiducialv2dim}
v=\pmat{1-\sqrt{3}\cr 1-i}\in\CC^2,
\end{equation}
under the action of
the Heisenberg group, which is generated by 
$$  S:=\pmat{0&1\cr1&0}\ \hbox{(cyclic shift)}, \qquad \gO:=\pmat{1&0\cr0&-1} \
\hbox{(modulation)}. $$
This is the reflection group $G(2,1,2)$ of Example \ref{G212example}.
Since $S\gO=-\gO S$, its collineation group is the Klein four-group $C_2\times C_2$. The normaliser of the Heisenberg group (as a 
collineation group) is the Clifford group,
which is generated by $S$, $\gO$ and the matrices
$$ F:={1\over\sqrt{2}}\pmat{1&1\cr1&-1} \ \hbox{(discrete Fourier transform)}, \qquad R:=\pmat{1&0\cr0&i^3}. $$
These notions generalise to a SIC of $d^2$ equiangular lines in $\Cd$
(see \cite{W18}).

The matrices $S$, $\gO$, $F$, $R$ are reflections.
The reflection group that they generate has a hidden reflection $RF$ of order $3$,
and so is not maximal (it is $G_9$).
We define the Zauner matrix, which is a reflection of 
order $3$, which fixes a fiducial vector, by
$$ Z:= e^{{2\pi i\over 24}} RF= {1\over4} \pmat{
        1+\sqrt{3}+(\sqrt{3}-1)i & 1+\sqrt{3}+(\sqrt{3}-1)i \cr \sqrt{3}-1-(\sqrt{3}+1)i
        & 1-\sqrt{3}+(\sqrt{3}+1)i}. $$
The reflection group generated by $F,R,Z$ (which contains $S,\gO$) 
is the maximal reflection group $G_{11}$.
Indeed, it has the same presentation as that given in \cite{LT09}, 
which has generators
$$ r_1=Z^2=Z^{-1}, \qquad r_3= R^2 F R^{-2}, \qquad r_4=R^3=R^{-1}. $$
A calculation shows that the reflection type (reflection orbits) of $G_{11}=\inpro{F,R,Z}$ is
\begin{equation}
\label{G11reflectiontype}
12\inpro{F},6\inpro{R},8\inpro{Z}= 12C_2,\,6C_4,8C_3,
\end{equation}
i.e., $G_{11}=G_{11}^{(2,4,3)}$, 
and $G_{11}$ has $12$ normal reflection subgroups 
(by Theorem \ref{Galatticethm}).
The reflection count of Proposition \ref{maximalrefcount} shows
that $G=G_{11}$ is a maximal reflection group, i.e.,
$$ \hbox{\# reflections in $G$} = 12(2-1)+6(4-1)+8(3-1) = 46 
= 2\Bigl({576\over24}-1\Bigr)=2\Bigl({|G|\over|Z(G)|}-1\Bigr). $$
Here the centre $Z(G)$ has order $24$, and so the collineation group has order ${|G|\over|Z(G)|}=24$.

In our computations, we will at times require 
more than one element of maximal order from a reflection orbit.
These can be easily obtained by conjugation of
a given representative with
reflections from different orbits, e.g.,
$$ F^G=\{gFg^{-1}: g \in G_{11}\}=\{gFg^{-1}: g \in \inpro{R,Z}\}. $$

We now decribe the collineation group.
Let $v$ be the fiducial vector of (\ref{fiducialv2dim}),
which is a root of the reflection $Z$.
The orbit of $v$ under the Heisenberg group gives 
the four equiangular lines $v_1=v$, $v_2=Sv$, $v_3=\gO v$, $v_4=S\gO v$,
and the orthogonal complement of these vectors/lines in $\CC^2$ 
gives a second such SIC. Thus we obtain eight lines given by the vectors
$$ v_1,v_2,v_3,v_4 = \pmat{1-\sqrt{3}\cr 1-i},
\pmat{1-i\cr1-\sqrt{3}}, \pmat{1-\sqrt{3}\cr i-1},\pmat{1-i\cr1-\sqrt{3}}, $$
$$ v_1^\sperp,v_2^\sperp,v_3^\sperp,v_4^\sperp = \pmat{1+i\cr\sqrt{3}-1}, 
\pmat{\sqrt{3}-1\cr1+i}, \pmat{1+i\cr1-\sqrt{3}}, \pmat{1-\sqrt{3}\cr1+i}. $$
The permutation action of the generators $F$, $R$, $Z$ for
$G_{11}$ (the Clifford group) on these lines
is (with the obvious notation)
$$ F:\ (1\,4^\sperp)(2\,2^\sperp)(3\,3^\sperp)(4\,1^\sperp), \qquad
 R:\ (1\,2^\sperp\,3\,4^\sperp)(2\,3^\sperp\,4\,1^\sperp), \qquad
 Z:\ (2\,3\,4)(2^\sperp\,3^\sperp\,4^\sperp). $$
Thus, $G_{11}$ (the Clifford group) acts on the four pairs of orthogonal lines $\{v_j,v_j^\perp\}$ via
\begin{equation}
\label{FRZperms}
F:\ (1\, 4), \qquad
 R:\ (1\,2\,3\,4), \qquad
 Z:\ (2\,3\,4).
\end{equation}

\begin{lemma}
\label{collinS4}
There is a surjective homomorphism $\Psi:G_{11}\to S_4$ given by
\begin{equation}
\label{FRZtoS4}
 F\mapsto (1\, 4), \qquad
 R\mapsto (1\,2\,3\,4), \qquad
 Z\mapsto (2\,3\,4),
\end{equation}
whose quotient is
an isomorphism from the collineation group to $S_4$.
In particular, the collineation group of a subgroup $H\subset G_{11}$ is 
isomorphic to
	$\Psi(H)\subset S_4$.
\end{lemma}


\begin{proof} Since the permutations 
        of (\ref{FRZperms})
generate the symmetric group $S_4$, which is of order $24$,
and a scalar matrix acts on a pair of orthogonal lines as the identity,
the map of (\ref{FRZtoS4}) gives a
surjective homomorphism $G_{11}\to S_4$, with kernel containing
	the scalar matrices in $G_{11}$. Since $Z(G_{11})$ consists of 
scalar matrices, we obtain
	the isomorphism $G_{11}/Z(G_{11})=G_{11}/\ker\Psi \cong S_4$.
\end{proof}

Lemma \ref{collinS4} gives the following geometric description of
the $46$ reflections in $G_{11}$. These correspond to the $23$ nontrivial
collineations in $S_4$.

\begin{itemize}
\item The $12$ reflections of order $2$ conjugate to $F$ act like
	$(1\,2)$, i.e., they fix two of the
	pairs of orthogonal lines and permute the other two (six collineations in total).
\item The $12$ reflections of order $4$ conjugate to $R$ or $R^{3}$ 
	act like $(1\,2\,3\,4)$, i.e., they cycle
	through the four pairs of orthogonal lines.
\item The $6$ reflections of order $2$ conjugate to $R^2$ act like
	$(1\,2)(3\,4)$, i.e., the permute two pairs of pairs of orthogonal lines.
\item The $16$ reflections of order $3$ conjugate to $Z$ or $Z^2$ act like
	$(1\,2\,3)$, i.e., they fix one pair of orthogonal lines and cycle 
		through the remaining three.
\end{itemize}


\subsection{Generators for the normal reflection subgroups}

We now use $G_{11}^{(2,4,3)}$ to illustrate
a general method for finding natural generating sets for the normal 
reflection subgroups 
of a reflection group $G=G^{(k_1,\ldots,k_m)}$, with reflection orbits
$R_{a_1}^G,\ldots,R_{a_m}^G$, and reflection type
$$ n_1 C_{k_1},\ldots, n_m C_{k_m}. $$
By (\ref{splitorbitdef}), (\ref{hatRjformula}), the $j$-th orbit of a
normal reflection subgroup $H=G^{(\ga_1,\ldots,\ga_m)}$
is split if
\begin{equation}
\label{ComplexSplitOrbit}
((R_{a_j})^{k_j/\ga_j})^H \ne ((R_{a_j})^{k_j/\ga_j})^G.
\end{equation}
Our earlier observation
\begin{itemize}
\item If $G$ is a complex reflection group generated by the reflections $\cR$,
then any reflection in $G$ is conjugate to a power of a reflection in $\cR$
                (Lemma 4.11 (iii) \cite{C76}).
\end{itemize}
implies that a generating set $\cR=\cR^{(k_1,\ldots,k_m)}$ of reflections for $G=G^{(k_1,\ldots,k_m)}$ 
must contain at least one reflection (of highest order) from each of the
$m$ reflection orbits (see Lemma \ref{abelianisationlemma}), 
i.e., $\cR$ can be partitioned
\begin{equation}
\label{Rgenpartition}
\cR=\cR_{a_1}\cup\cdots\cup \cR_{a_m}, \qquad 
\cR_{a_j}=\cR_{a_j}^{(k_1,\ldots,k_m)}:=\cR\cap R_{a_j}^G,
\end{equation}
where $|\cR_{a_j}|\ge 1$. Our choice of generators $F,R,Z$ consisting of one
reflection of maximal order from each of the three orbits is therefore a minimal
generating set of reflections. In this particular case, any such choice gives 
a minimal generating set (see Theorem \ref{nicelygenthm}).

It is also quite possible that more than one reflection may need to be taken from a
reflection orbit to obtain a generating set, e.g., $G=G_8$ has a single reflection 
orbit $6C_4$, and so at least two reflections must be taken from it to obtain
a generating set (it can in fact be generated by two reflections of order $4$, but not all such pairs).

Given a generating set $\cR$ of reflections for $G=G^{(k_1,\ldots,k_m)}$ of the form 
(\ref{Rgenpartition}), the reflections
\begin{equation}
\label{Rgenpartderivedgens}
	\cD_{\cR,a_j}^{(\ga_1,\ldots,\ga_m)}:=\{g^{k_j/\ga_j} : g\in \cR_{a_j}^{(k_1,\ldots,k_m)}\}, \qquad \ga_j\ne1, 
\end{equation}
are reflections of maximal order in $R_{a_j}^G\cap G^{(\ga_1,\ldots,\ga_m)}$,
and we call them the reflections {\bf inherited} from $\cR$.
The inherited reflections show the inclusion $G^{(\ga_1,\ldots,\ga_m)}\subset G^{(k_1,\ldots,k_m)}$
and come for free.
We will investigate to what extent they are useful for constructing 
a nice generating set of reflections for $G^{(\ga_1,\ldots,\ga_m)}$.

We now calculate the reflection subgroups of $G_{11}$ and minimal
sets of reflections which generate them. We then summarise the
general algorithm for such calculations.

\begin{figure}[h]
\begin{center}
\begin{tikzpicture}
    \matrix (A) [matrix of nodes, row sep=1.0cm, column sep = -1.2 cm]
    {
	    &&& $G_{11}^{(2,4,3)}=G_{11}$ &&& \\
	& $G_{11}^{(2,2,3)}=G_{15}$ && $G_{11}^{(2,4,1)}=G_{9}$ && $G_{11}^{(1,4,3)}=G_{10}$ & \\
	$G_{11}^{(2,1,3)}=G_{14}$ && $G_{11}^{(2,2,1)}=G_{13}$ && $G_{11}^{(1,4,1)}=G_{8}$ && \greybox{$G_{11}^{(1,2,3)}=G_7$} \\
& $G_{11}^{(2,1,1)}=G_{12}$ && \greybox{ $G_{11}^{(1,2,1)}=G(4,2,2)^*$ } && \greybox{ $G_{11}^{(1,1,3)}=G_5$ } & \\
	&& \greybox{$G_{11}^{(1,1,1)} =1^*$} &&&& \\
    };
    \draw (A-1-4)--(A-2-2); \draw (A-1-4)--(A-2-4);
    \draw (A-1-4)--(A-2-6); \draw (A-2-2)--(A-3-1);
    \draw (A-2-2)--(A-3-3); \draw (A-2-4)--(A-3-3);
    \draw (A-2-4)--(A-3-5); \draw (A-2-6)--(A-3-5);
    \draw (A-2-6)--(A-3-7); \draw (A-2-2)--(A-3-7);
    \draw (A-3-1)--(A-4-2); \draw (A-3-3)--(A-4-2); \draw (A-3-3)--(A-4-4); \draw (A-3-5)--(A-4-4); \draw (A-4-2)--(A-5-3);
    \draw (A-4-4)--(A-5-3);
    \draw (A-3-7)--(A-4-4);
    \draw (A-3-1)--(A-4-6);
    \draw (A-3-7)--(A-4-6);
    \draw (A-4-6)--(A-5-3);
\end{tikzpicture}
\end{center}
\vskip-0.5truecm
\caption{\label{G11-lattice}
The lattice of the $12$ normal reflection subgroups of $G_{11}$. 
Those in grey have a smaller collineation group, 
and the meets with a ${}^*$ are not the intersection of the groups. 
Note that all the nontrivial subgroups above are irreducible
(as per Theorem \ref{Galatticethm}),
with one being imprimitive, i.e., $G(4,2,2)$.
The groups not in grey form the sublattice of the collineation 
	preserving subgroups of $G_{11}$, with those in grey giving 
	its lower boundary.
}
\end{figure}
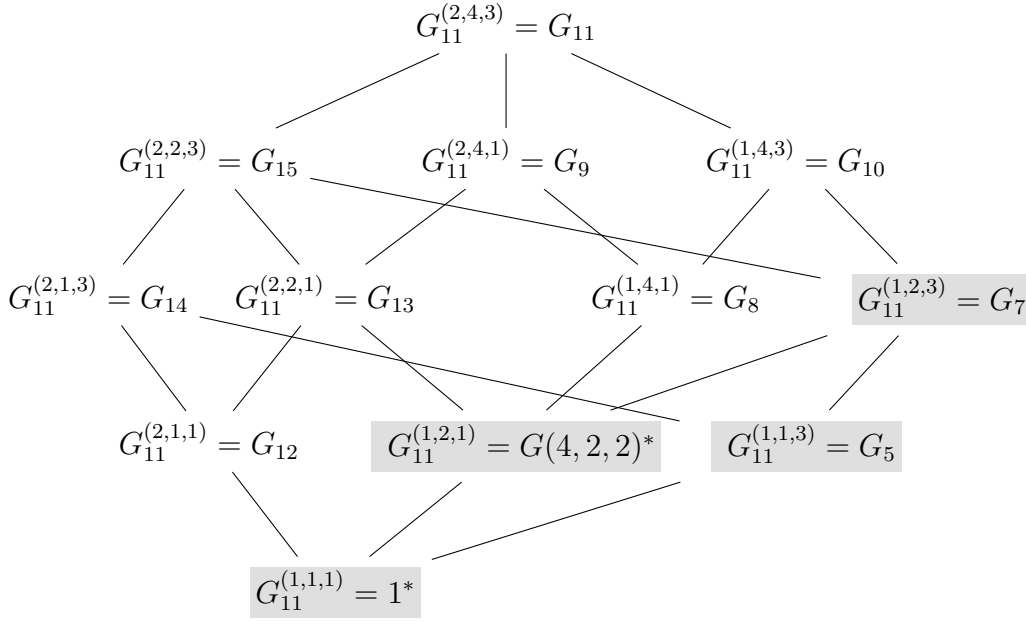

\begin{example} (Normal reflection subgroups of $G_{11}$).
We have already 
	observed that there are twelve normal reflection 
subgroups of $G_{11}=G_{11}^{(2,4,3)}=\inpro{F,R,Z}$
(see the indices in Figure \ref{G11-lattice}).
	Those which are maximal subgroups of $G_{11}$ (in the lattice)
are generated by the inherited
	reflections, and have no split orbits. They (and their reflection types) are
\begin{align*}
G_{15} = G_{11}^{(2,2,3)} =\inpro{F,R^2,Z}, & \qquad 
12\inpro{F},6\inpro{R^2},8\inpro{Z}=12 C_2,6C_2,8C_3 \cr
G_{9} = G_{11}^{(2,4,1)} =\inpro{F,R,Z^3} =\inpro{F,R}, & 
\qquad 12\inpro{F},6\inpro{R},8\inpro{Z^3} = 12C_2,6C_4, 8C_1, \cr
G_{10} = G_{11}^{(1,4,3)} =\inpro{F^2,R,Z} =\inpro{R,Z}, & 
\qquad
12\inpro{F^2},6\inpro{R},8\inpro{Z}= 12C_1,6C_4,8C_3.
\end{align*}
Here we have given the reflection type in a transparent way. 
When the identity subgroup of $R_{a_j}$ 
	is selected, i.e., $\ga_j=1$, 
the corresponding orbit disappears, i.e., no reflections are taken from it, 
and henceforth we won't record it in the reflection type, e.g., we will write
$$ G_{9} = G_{11}^{(2,4,1)} =\inpro{F,R}, 
        \qquad 12\inpro{F},6\inpro{R} = 12C_2,6C_4. $$
The next level of normal reflection subgroups in the lattice are
\begin{align*}
G_{14} = G_{11}^{(2,1,3)} = \inpro{F,Z}, & 
\qquad 12\inpro{F},8\inpro{Z}=12 C_2,8C_3, \cr
G_{13} = G_{11}^{(2,2,1)}=\inpro{F,F^Z,R^2 }, & 
\qquad 12\inpro{F},6\inpro{R^2} = 12C_2,6C_2, \cr
G_8 = G_{11}^{(1,4,1)} =\inpro{R,R^F}, & 
\qquad 6\inpro{R} = 6 C_4, \cr
G_7 = G_{11}^{(1,2,3)} = \inpro{R^2,Z,Z^R}, & 
\qquad 6\inpro{R^2},4\inpro{Z},4\inpro{Z^R}=6C_2,4C_3,4C_3.
\end{align*}
In addition to the inherited reflections, the last three require 
a second reflection from one orbit, i.e.,
$$ F^Z=ZFZ^{-1}, \qquad R^F=FRF^{-1}, \qquad Z^R=RZR^{-1}. $$
The first three have an orbit disappear, and the last $G_7$ has a 
split orbit, i.e., its reflection type has an extra orbit obtained
	by ``splitting'' $8\inpro{Z}=8C_3$ into two $G_7$-orbits
	$4C_3,4C_3$. 
	The latter fact means that we can apply Theorem \ref{Galatticethm}
to $G_7$ to find normal reflection subgroups of it which are not 
	normal in $G_{11}$, namely $G_4$ and $G_6$ (see Figure \ref{G7-lattice})

It is now a convenient time to consider the collineation group, whose
size decreases as we move down the lattice.
By Lemma \ref{collinS4}, $\Psi(G_{11})=S_4$, and
$$ \Psi(G_{14})=\inpro{\Psi(F),\Psi(Z)}
	= \inpro{(1\,4),(2\,3\,4)} = S_4, \quad
 \Psi(G_{13})=\inpro{(1\,4),(1\,3),(1\,3)(2\,4)}=S_4, $$
	$$ \Psi(G_8)=\inpro{(1\,2\,3\,4),(1\,4\,2\,3)}=S_4, \quad
	\Psi(G_7)=\inpro{(1\,3)(2\,4),(2\,3\,4),(1\,2\,3)} = A_4. $$
Thus $G_{11}$ is the maximal reflection group for all the normal reflection subgroups found so far, 
i.e., they are collineation preserving subgroups, except for $G_{7}$.
It follows that $G_{7}$ is a maximal reflection group by 
Proposition \ref{maximalrefcount}, i.e., it has
	$2(|A_4|-1)=22$ reflections.

The next level of the lattice gives the remaining
nontrivial normal reflection subgroups
\begin{align*} G_{12} = G_{11}^{(2,1,1)}=\inpro{F,F^Z,F^{Z^2}}, & 
\qquad 12\inpro{F} =12C_2, \cr
G_5 =  G_{11}^{(1,1,3)} =\inpro{Z,Z^R}, & 
\qquad
4\inpro{Z},4\inpro{Z^R} =4C_3,4C_3, \cr
	G(4,2,2) =  G_{11}^{(1,2,1)}=\inpro{R^2,(R^2)^{Z},(R^2)^{Z^2}}, & \qquad
2\inpro{R^2},2\inpro{(R^2)^{Z}},2\inpro{(R^2)^{Z^2}}
=2C_2,2C_2,2C_2,
\end{align*}
with collineation groups
$$ \Psi(G_{12}) = \inpro{(1\,4), (1\,3), (1\, 2)}=S_4, \qquad
\Psi(G_5)=\inpro{(2\,3\,4),(1\,2\,3)}=A_4, $$
$$ \Psi(G(4,2,2))= \inpro{(1\,3)(2\,4),(1\,2)(3\,4),(1\,4)(2\,3)}
= C_2\times C_2. $$
Here
the subgroups 
which do not preserve the collineation group
have split orbits, i.e., $G_5$ has $8\inpro{Z}$ split into two orbits,
and $G(4,2,2)$ has $6\inpro{R^2}$ split into three orbits.
\end{example}

The above example suggests that for a nontrivial normal reflection subgroup 
of a complex reflection group, the following are equivalent:
\begin{itemize}
\item There is a split orbit.
\item The collineation group is not preserved.
\item The number of reflections plus the number of hidden reflections is not preserved.  
\end{itemize}
This is indeed true in general, and follows from the action of 
the collineation group on the root lines (reflection subgroups).
Moreover, the first occurrence of a split orbit gives a maximal reflection group.


\begin{lemma} Let $G$ be an irreducible reflection group.
Then there is a natural 
action of the collineation group of $G$ on the
reflection subgroups for the roots given by
$$ g \cdot R_a = g R_a g^{-1}= R_{ga}, \qquad \forall g\in G. $$
By construction, this action is transitive on each of the reflection orbits $R_a^G$.
\end{lemma}

\begin{proof} Suppose without loss of generality that $G$ is unitary, and
$\gl$ is a unit scalar in the collineation group ($\gl=\pm1$ when $G$ is quaternionic). 
Then from (\ref{raxidefn}), we calculate
$$ r_{a,\xi}^{\gl g}
= \gl g\Bigl(I-{a(1-\xi)a^*\over\inpro{a,a}}\Bigr)(\gl g)^{-1}
= I- { \gl g a(1-\xi)a^* (\gl g)^*\over\inpro{a,a}}
= r_{\gl g a,\xi}
= r_{g a,\xi}, \quad g\in G, $$
so that $\gl g\cdot R_a=g\cdot R_a=R_{ga}$, and the collineation group of
$G$ acts on the reflection subgroups (and the root lines). 
\end{proof}

\begin{theorem}
\label{splitorbitcollineation}
Let $G$ be a complex reflection group, and 
$N=G^{(\hat R_{a_1},\ldots,\hat R_{a_m})}$ 
be a normal reflection subgroup of $G$.  Then 
\begin{enumerate}[\rm (i)]
\item The collineation groups of $N$ and $G$ are different if and only if 
some $\hat R_{a_j}^G$ is a split orbit, i.e., $\hat R_{a_j}^N \ne \hat R_{a_j}^G$.
\item
If $N$ is in the lower boundary of the sublattice of the collineation preserving normal 
	reflection subgroups,
i.e., there is a normal reflection subgroup immediately above it with the same
collineation group as $G$, then $N$ is a maximal reflection group.
\end{enumerate}
\end{theorem}

\begin{proof}
Some $\hat R_{a_j}^G$ is a split orbit of $N$ if and only 
the (collineation) action of $N$ on $\hat R_{a_j}^G$ is not transitive, 
	i.e., is different from the (collineation) action of $G$ (which is transitive), 
	and so the collineation groups for $N$ and $G$ must be different.

If $N$ has a split orbit, then we may add to it scalar matrices from $G$ so that its
hidden reflections become reflections, thereby obtaining $N'$ the maximal
reflection group for $N$ as a normal reflection subgroup of $G$. Since $N'$ is not collineation
preserving and $N\subset N'$, we must have $N'=N$ when $N$ is in the lower boundary of the 
collineation preserving subgroups, 
i.e., $N$ must be a maximal reflection group in this case.
\end{proof}

In other words, $N$ is a collineation preserving subgroup of $G$ 
if and only if it has no split orbits, i.e.,
$$ \hat R_{a_j}^N = \hat R_{a_j}^G, \qquad\forall j, $$
and as we work down the lattice of normal reflection subgroups as soon as 
we hit one which is not collineation preserving, it is a maximal reflection group.

Later we will show the collineation preserving normal reflection
subgroups $N=G^\ga$ of $G$ are determined 
by the
order of their abelianisation $(G^\ga)_{\rm ab}$ (Corollary \ref{abeliansplitorbitscor}).

We now give a prototypical example of calculating a minimal set of
generators when more than one element from an orbit is required.

\begin{example} 
\label{mingensexampleII}
Consider 
$G=G_{13}=\inpro{F,F^Z,R^2}$,
which has reflection type $12\inpro{F},6\inpro{R^2}$, and order $96$. 
The group $H$ generated by the six reflections 
in the $6\inpro{R^2}$ orbit, i.e.,
\begin{equation}
\label{G13R2orbit}
\bigl\{ \pm\pmat{1&0\cr0&-1},  \pm\pmat{0&1\cr1&0}, 
\pm\pmat{0&i\cr-i&0}\bigr\}, 
\end{equation}
has order $16$, and the group generated by adding any additional 
reflection from the $12\inpro{F}$ orbit has order $32$.
Hence at least two reflections must be taken from the $12\inpro{F}$ orbit
to obtain a generating set of reflections for $G_{13}$.
The action of $H$ on the $12\inpro{F}$ orbit gives three orbits
of size four, i.e.,
\begin{equation}
\label{G13Forbit}
\begin{split}
	&	\bigl\{ \pm{1\over\sqrt{2}}\pmat{1&-1\cr-1&-1},
\pm{1\over\sqrt{2}}\pmat{1&1\cr1&-1}\bigr\}, \quad
\bigl\{ \pm{1\over\sqrt{2}}\pmat{1&i\cr-i&-1},
\pm{1\over\sqrt{2}}\pmat{1&-i\cr i&-1}\bigr\}, \\
	&\qquad\qquad\qquad\quad
	\bigl\{ \pm{1\over\sqrt{2}}\pmat{0&1+i\cr1-i&0},
        \pm{1\over\sqrt{2}}\pmat{0&1-i\cr1+i&0}\bigr\}.
\end{split}
\end{equation}
It can easily be verified that the minimal generating sets 
	for $G_{13}$ are given by 
\begin{itemize} 
\item One reflection from the $6\inpro{R^2}$ orbit (\ref{G13R2orbit}).
\item Two reflections from different $H$-orbits (\ref{G13Forbit}) of
	the $12\inpro{F}$ orbit.
\end{itemize}
The action of the group of order $48$ generated by the reflections in 
	the $12\inpro{F}$ orbit on the $6\inpro{R^2}$ orbit gives
	a single orbit.
\end{example}

In summary, minimal sets of generating reflections can be 
calculated as follows:

\vskip0.7truecm
{\bf Basic algorithm for calculating the normal
reflection subgroups of $G^{(k_1,\ldots,k_m)}$}
\begin{itemize}
\item Find a minimal set of reflections that generate $G=G^{(k_1,\ldots,k_m)}$, 
by taking at least one element of maximal order from each of 
its reflection orbits $R_{a_j}^G$, and then additional maximal order reflections
from the 
$\inpro{\cup_{\ell \ne j} R_{a_\ell}^G}$-orbits 
of $R_{a_j}^G$, 
as required (in Example \ref{mingensexampleII} there was at most one element
		taken from each of these orbits).
\item For a normal reflection subgroup $H=G^{(\ga_1,\ldots,\ga_m)}$ of $G$ which
is a subgroup of some $K=G^{(\gb_1,\ldots,\gb_m)}$ in the lattice, 
take the reflections inherited from a minimal generating set of reflections 
for $K$, adding additional reflections of maximal order from the
reflection orbits of $K$, as required (see above).
In particular, additional reflections will be required when
		a reflection orbit of $K$ splits into more parts than the
		number of inherited reflections from $K$ that are in it.
\end{itemize}

We have not been completely prescriptive above, 
as calculations can easily be done by hand, 
and there is an art to it: one would like generators that reflect 
the lattice structure and are obviously related to the original ones. 
If calculating all the normal subgroups, then one
would work down the lattice considering the inherited reflections for
all the reflection subgroups $K$ immediately above a given $H$.


\subsection{All the reflection subgroups of $G_{11}$} 

We will now calculate all the reflection subgroups of $G_{11}$ by 
identifying its set of maximal reflection subgroups,
and then working down from these.

\begin{example} 
\label{reflectsubgroupsG11}
(Reflection subgroups of $G_{11}$)
The maximal reflection group $G_{11}$ has $135$ subgroups 
	(up to conjugacy), of which $38$ are reflection subgroups,
	with $12$ being normal, i.e.,
\begin{equation}
\label{list1}
1,\, G_5,\, G_7,\, G_8,\, G_9,\, G_{10},\, G_{11},\, G_{12},\,
G_{13},\, G_{14},\, G_{15},\, G(4,2,2).
\end{equation}
The remaining $26$ non-normal reflection subgroups must be normal
subgroups in their maximal reflection group (since they are contained in the
maximal reflection group $G_{11}$). 

The following normal reflection subgroups of the maximal reflection group
$G_{7}$ 
are not normal in $G_{11}$
\begin{equation}
\label{list2}
G_4,\,  G_6, 
\end{equation}
and from the normal maximal reflection subgroup $G(4,2,2)$ of (\ref{list1}), 
we obtain
\begin{equation}
\label{list3}
B_2,\,C_2\times C_2, \qquad \hbox{where $B_2=G(2,1,2)\cong G(4,4,2)$}.
\end{equation}
On the lattice of Figure \ref{G2p22-lattice} 
	($p=2$), these appear as $G(2,1,2)$, $G(2,1,2)$,
$G(4,4,2)$, and $C_2\times C_2$, $G(2,2,2)$, $G(2,2,2)$, respectively,
with each triple being conjugate in $G_{11}$.
The reflection group $C_2\times C_2$ above is also maximal, 
and its normal subgroup $C_2$ corresponds to a reflection orbit. 
We will consider this later with the other such cases.

	The remaining non-normal maximal reflection subgroups of $G_{11}$
are contained in the non-normal maximal reflection subgroups 
$G(6,2,2)$ and $G(8,2,2)$. 
	
The seven nontrivial normal reflection subgroups of 
$G(6,2,2)$ give new reflection subgroups of $G_{11}$
	(see  Figure \ref{G2p22-lattice}, $p=3$), i.e., 
\begin{equation}
\label{list4}
G(6,2,2),\, G(3,1,2),\,  G(3,1,2),\, G(6,6,2),\,
C_3\times C_3,\, G(3,3,2),\, G(3,3,2).
\end{equation}
Here the repeated groups $G(3,1,2)$ and $G(3,3,2)$ are not 
conjugate in $G_{11}$. The reflection group $C_3\times C_3$ above
is a maximal reflection group, and we will consider its normal 
reflection subgroups later.

We now consider the maximal subgroup $G(8,2,2)$.
This appears as the reflection subgroup of 
	$G_{11}$ generated by the monomial reflections (and hence is imprimitive), 
i.e.,
$$ G(8,2,2)= H = H^{(4,2,2)} = 
\inpro{R,S,F^{Z^2}}, \qquad
2\inpro{R},4\inpro{S},4\inpro{F^{Z^2}}. $$
This has three conjugates in $G_{11}$ (and so is not normal).
Ten of the twelve normal reflection subgroups of $G(8,2,2)$ give new reflection
subgroups of $G_{11}$ (see Fig.\ \ref{G822-lattice}), i.e.,
\begin{align}
\label{list5}
& G(8,2,2),\ G(8,4,2),\, G(4,1,2),\, G(4,1,2),\, G(8,8,2),\, \cr
& G(4,2,2),\, C_4\times C_4,\, B_2,\, B_2,\, C_2\times C_2.
\end{align}
Here $G(4,2,2)$ appears twice as a subgroup of $G(8,2,2)$, including
	once as a normal subgroup of $G_{11}$, which was already counted in  (\ref{list1}),
	and $B_2$ appears twice as $G(4,4,2)$. 
The subgroup $G(4,2,2)$ above is a maximal reflection group, and its 
normal reflection subgroups give no new reflection subgroups of $G_{11}$.

Finally, we consider the reflection subgroups of the maximal reflection groups
$$ C_2\times C_2=\inpro{F,-F}, \qquad C_3\times C_3=\inpro{Z,e^{2\pi i\over3} Z}, \qquad
C_4\times C_4 = \inpro{R,iR} \supset \inpro{R^2,-R^2} = C_2\times C_2, $$
of (\ref{list3}), (\ref{list4}) and (\ref{list5}). 
These are (see Example \ref{Maximalonedim})
\begin{equation}
\label{list6}
C_2=\inpro{F}, \qquad C_3=\inpro{Z}, \qquad 
	C_2\times C_4=\inpro{R^2,iR}, \quad
C_4=\inpro{R}, \quad C_2=\inpro{R^2}, 
\end{equation}
where $C_2=\inpro{R^2}$ is not parabolic.
Thus, the $38=12+2+2+7+10+5$ reflection subgroups of $G_{11}$ are given by
(\ref{list1}), 
(\ref{list2}), \ldots
(\ref{list6}).
\end{example}

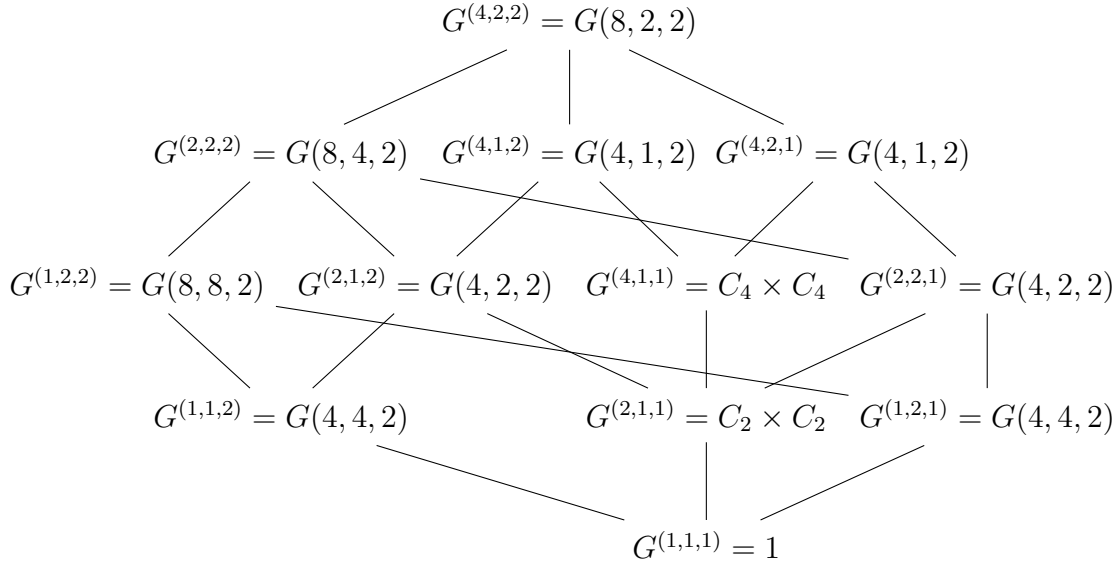
\begin{figure}[h]
\begin{center}
\begin{tikzpicture}
    \matrix (A) [matrix of nodes, row sep=1.0cm, column sep = -1.75 cm]
    {
	    &&& $G^{(4,2,2)}=G(8,2,2)$ &&& \\
	& $G^{(2,2,2)}=G(8,4,2)$ && $G^{(4,1,2)}=G(4,1,2)$ && $G^{(4,2,1)}=G(4,1,2)$ & \\
	$G^{(1,2,2)}=G(8,8,2)$ && $G^{(2,1,2)}=G(4,2,2)$ && $G^{(4,1,1)}=C_4\times C_4$ && $G^{(2,2,1)}=G(4,2,2)$ \\
	& $G^{(1,1,2)}=G(4,4,2)$ &&& $G^{(2,1,1)}=C_2\times C_2$ && $G^{(1,2,1)}=G(4,4,2)$ \\
	&&&& $G^{(1,1,1)}=1$ && \\
    };
    \draw (A-1-4)--(A-2-2);
    \draw (A-1-4)--(A-2-4);
    \draw (A-1-4)--(A-2-6);
    \draw (A-2-2)--(A-3-1);
    \draw (A-2-2)--(A-3-3);
    \draw (A-2-2)--(A-3-7);
    \draw (A-2-4)--(A-3-3);
    \draw (A-2-4)--(A-3-5);
    \draw (A-2-6)--(A-3-5);
    \draw (A-2-6)--(A-3-7);
    \draw (A-3-1)--(A-4-2);
    \draw (A-3-3)--(A-4-2);
    \draw (A-3-3)--(A-4-5);
    \draw (A-3-5)--(A-4-5);
    \draw (A-3-7)--(A-4-5);
    \draw (A-3-7)--(A-4-7);
    \draw (A-3-1)--(A-4-7);

    \draw (A-4-2)--(A-5-5);
    \draw (A-4-5)--(A-5-5);
    \draw (A-4-7)--(A-5-5);
\end{tikzpicture}
\end{center}
\vskip -0.7 truecm
\caption{\label{G822-lattice}
	The lattice of normal reflection subgroups of 
	$G^{(4,2,2)}=G(8,2,2)\subset G_{11}$, $2C_4,4C_2,4C_2$,
where $C_2\times C_2=G(2,1,2)$, $G(4,4,2)=B_2$. }
\end{figure}

\begin{example}
\label{G(8,2,2)example}
It is easy to calculate minimal generating sets of reflections for the $12$ normal reflection subgroups of
$$ G(8,2,2)= H = H^{(4,2,2)} = 
\inpro{R,S,F^{Z^2}}, \qquad
2\inpro{R},4\inpro{S},4\inpro{F^{Z^2}}, $$
which give (\ref{list5}).
Those which preserve the collineation group are
(see Fig.\ \ref{G822-lattice})	
\begin{align*}
G(8,4,2) = H^{(2,2,2)} = \inpro{R^2,S,F^{Z^2}}, & \qquad
  2\inpro{R^2},4\inpro{S},4\inpro{F^{Z^2}}, \cr
G(4,1,2) = H^{(4,1,2)} = \inpro{R,F^{Z^2}},& \qquad
  2\inpro{R},4\inpro{F^{Z^2}}, \cr
G(4,1,2) = H^{(4,2,1)} = \inpro{R,S}, & \qquad
  2\inpro{R},4\inpro{S}, \cr
G(8,8,2) = H^{(1,2,2)} = \inpro{S,F^{Z^2}}, & \qquad
4\inpro{S},4\inpro{F^{Z^2}}.
\end{align*}
The remaining nontrivial subgroups 
have split orbits and smaller collineation groups
\begin{align*}
G(4,2,2) = H^{(2,2,1)} = \inpro{R^2,S,S^R}, & \qquad
  2\inpro{R^2}, 2\inpro{S}, 2\inpro{S^R}, \cr
G(4,2,2) = H^{(2,1,2)} = \inpro{R^2,F^{Z^2},F^{RZ^2}}, & \qquad
  2\inpro{R^2}, 2\inpro{F^{Z^2}}, 2\inpro{F^{RZ^2}}, \cr
C_4\times C_4 = H^{(4,1,1)} = \inpro{R,R^S}, & \qquad
  \inpro{R},\inpro{R^S}, \cr
G(4,4,2) = H^{(1,2,1)} = \inpro{S,S^R}, & \qquad
  2\inpro{S},2\inpro{S^R}, \cr
G(4,4,2) = H^{(1,1,2)} = \inpro{F^{Z^2},F^{RZ^2}}, & \qquad
  2\inpro{F^{Z^2}},2\inpro{F^{RZ^2}}, \cr
C_2\times C_2 = H^{(2,1,1)} = \inpro{R^2,(R^2)^S}, & \qquad
  \inpro{R^2},\inpro{(R^2)^S}.
\end{align*}
\end{example}

We observe that the normal reflection subgroups $C_3\times C_3$ and $C_2\times C_2,C_4\times C_4$
of the imprimitive reflection groups $G(6,2,2)$ and $G(8,2,2)$, respectively, are
not irreducible, as is allowed by Theorem \ref{Galatticethm}.

In summary, 
all the reflection subgroups of $G_{11}=\inpro{F,R,Z}$ can be calculated by knowing that
$$ G(8,2,2)=\inpro{R,S,F^{Z^2}}, \qquad G(6,2,2)= \inpro{Z,F^{Z^2R^2},F^{RZ^2}} $$
are reflection subgroups of $G_{11}$ which are not normal and are maximal reflection groups.
The reflection subgroups of $G_{11}$ can then be found by calculating the normal 
reflection subgroups of $G_{11}$, $G(8,2,2)$, $G(6,2,2)$, moving down the lattice 
to the point where another maximal reflection group is found,
i.e., a split orbit (see Figure \ref{G11-lattice-maximal}),
and then applying this process to that maximal reflection group.

\begin{figure}[H]
\begin{center}
\begin{tikzpicture}
    \matrix (A) [matrix of nodes, row sep=1.0cm, column sep = 0.2 cm]
    {
  & $G_{11}$ && \\
  $G_7$ & $G(8,2,2)$ & $G(6,2,2)$ & \\
  $G(4,2,2)$ & $G(4,2,2)$ & $C_4\times C_4$ & $C_3\times C_3$ \\
	$C_2\times C_2$ & $C_2\times C_2$ && \\
    };
    \draw (A-1-2)--(A-2-1);
    \draw (A-1-2)--(A-2-2);
    \draw (A-1-2)--(A-2-3);
    \draw (A-2-1)--(A-3-1);
    \draw (A-2-2)--(A-3-1);
    \draw (A-2-2)--(A-3-2);
    \draw (A-2-2)--(A-3-3);
    \draw (A-2-3)--(A-3-4);
    \draw (A-3-1)--(A-4-1);
    \draw (A-3-2)--(A-4-1);
    \draw (A-3-2)--(A-4-2);
    \draw (A-3-3)--(A-4-2);
\end{tikzpicture}
\end{center}
\vskip -0.5truecm
\caption{\label{G11-lattice-maximal}
	The lattice of the maximal reflection subgroups of $G_{11}$, up to conjugacy
	($G(8,2,2)$ and $G(6,2,2)$ are not normal, and have three and four conjugates). 
Here $C_3\times C_3=\inpro{Z,e^{2\pi i\over3}Z}$,
	$C_4\times C_4=\inpro{R,iR}\supset C_2\times C_2=\inpro{R^2,-R^2}$,
	$C_2\times C_2=\inpro{F,-F}$.
See Table \ref{G11maximalTable} for the reflection type, collineation group, and
	generators for each subgroup.}
\end{figure}
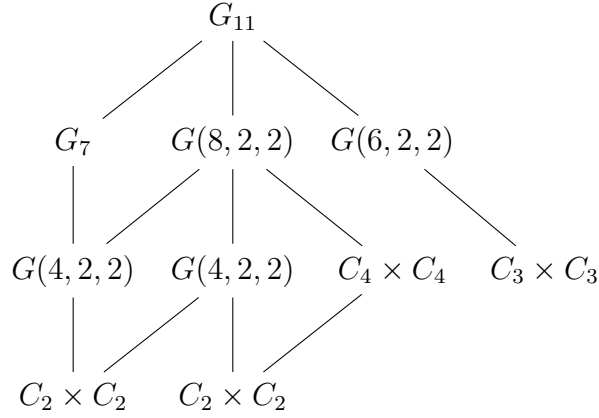

\begin{table}[h]
\caption{
\label{G11maximalTable}
	The maximal reflection subgroups of $G_{11}$ as in 
 the lattice of Figure \ref{G11-lattice-maximal}. }
\smallskip
\begin{tabular}{ |  >{$}l<{$} | >{$}c<{$} | >{$}l<{$} | >{$}l<{$} | >{$}l<{$} | >{$}l<{$} | >{$}l<{$} |  >{$}l<{$} | }
\hline
	&&&&&&&\\[-0.3cm]
\hbox{ST} & \hbox{Rank} & \hbox{Order} & \hbox{Reflection type} & \hbox{$K_a$} &
	\hbox{Reflections} & G_{\rm coll} & \hbox{Generators} \\[0.1cm]
\hline
	&&&&&&&\\[-0.3cm]
G_{11} & 2 & 576 & 12 C_2, 6 C_4, 8 C_3 & C_{24} & 2^{18} 3^{16} 4^{12} & S_4 & F,R,Z \\
G_7 & 2 & 144 & 6 C_2, 4 C_3, 4 C_3 & C_{12} & 2^6 3^{16} & A_4 & R^2,Z,Z^R \\
	&&&&&&&\\[-0.4cm]
G(8,2,2) & 2 & 64 & 2C_4,4C_2,4C_2 & C_8 & 2^{10}4^{4} & D_4 & R,S,F^{Z^2} \\
G(6,2,2) & 2 & 36 & 2C_3,3C_2,3C_2 & C_6 & 2^{6}3^{4} & D_3 & Z^S,F,F^{Z^2} \\
G(4,2,2) & 2 & 16 & 2C_2,2C_2,2C_2 & C_4 & 2^{6} & C_2\times C_2 & R^2,S,S^R \\
G(4,2,2) & 2 & 16 & 2C_2,2C_2,2C_2 & C_4 & 2^{6} & C_2\times C_2 & R^2,F^{Z^2},F^{SZ^2} \\
	&&&&&&&\\[-0.4cm]
C_4\times C_4 & 2 & 16 & C_4, C_4 & C_{4} & 2^2 4^{4} & C_4 & R,iR \\
C_3\times C_3 & 2 & 9 & C_3, C_3 & C_{3} & 3^{4} & C_3 & Z,e^{2\pi i\over 3}Z \\
C_2\times C_2 & 2 & 4 & C_2, C_2 & C_{2} & 2^{2} & C_2 & F,-F \\
C_2\times C_2 & 2 & 4 & C_2, C_2 & C_{2} & 2^{2} & C_2 & R^2,-R^2 \\ [0.1cm]
\hline
\end{tabular}
\end{table}

\section{The normal reflection subgroups of the primitive complex reflection groups} 

We are now in a position to easily identify the remaining primitive complex reflection
groups which are maximal reflection groups, and thereby complete their classification.

First we have $G_7$ and its normal reflection subgroups (with generators).

\begin{example} (Normal reflection subgroups of $G_{7}$).
There are eight normal reflection subgroups (see the indices in Figure \ref{G7-lattice}) of 
the maximal reflection group
$$ G_7=G_7^{(2,3,3)}=\inpro{R^2,Z,Z^R}, \qquad
6\inpro{R^2},4\inpro{Z},4\inpro{Z^R} =6C_2,4C_3,4C_3, $$
which has collineation group $A_4$.
The first level down in the lattice, we have
\begin{align*}
G_5 = G_7^{(1,3,3)} =\inpro{Z,Z^R}, &\qquad 4\inpro{Z},4\inpro{Z^R} =4C_3,4C_3, \cr
G_6 = G_7^{(2,1,3)} =\inpro{R^2,Z^R}, &\qquad 6\inpro{R^2},4\inpro{Z^R} = 6C_2,4C_3, \cr
G_6 = G_7^{(2,3,1)} =\inpro{R^2,Z}, &\qquad 6\inpro{R^2},4\inpro{Z}= 6C_2,4C_3.
\end{align*}
At the next level, we have
\begin{align*}
G_4 = G_7^{(1,1,3)} =\inpro{Z^R,(Z^R)^S}, &\qquad 4\inpro{Z^R} =4C_3, \cr
G(4,2,2) = G_7^{(2,1,1)} =\inpro{R^2,S,S^R}, &\qquad 2\inpro{R^2},2\inpro{S},2\inpro{S^R} =2C_2,2C_2,2C_2, \quad\hbox{(split orbit)} \cr
G_4 = G_7^{(1,3,1)} =\inpro{Z,Z^S}, &\qquad 4\inpro{Z} =4C_3. 
\end{align*}
The subgroup $G(4,2,2)$ is the only one above with a split orbit, and hence 
a smaller collineation group, i.e., $C_2\times C_2$.
\end{example}

\begin{figure}[H]
\begin{center}
\begin{tikzpicture}
    \matrix (A) [matrix of nodes, row sep=1.0cm, column sep = -1.1 cm]
    {
	    &&& $G_{7}^{(2,3,3)}=G_{7}$ &&& \\
	& $G_{7}^{(1,3,3)}=G_{5}$ && $G_{7}^{(2,1,3)}=G_{6}$ && $G_{7}^{(2,3,1)}=G_{6}$ & \\
	&& $G_{7}^{(1,1,3)}=G_{4}$ && \greybox{$G_{7}^{(2,1,1)}=G(4,2,2)$} && $G_{7}^{(1,3,1)}=G_4$ \\
	& && \greybox{$G_{7}^{(1,1,1)}=1^*$} &&& \\
    };
    \draw (A-1-4)--(A-2-2);
    \draw (A-1-4)--(A-2-4);
    \draw (A-1-4)--(A-2-6);

    \draw (A-2-2)--(A-3-3);
    \draw (A-2-4)--(A-3-3);
    \draw (A-2-4)--(A-3-5);
    \draw (A-2-6)--(A-3-5);
    \draw (A-2-6)--(A-3-7);
    \draw (A-2-2)--(A-3-7);

    \draw (A-3-3)--(A-4-4);
    \draw (A-3-5)--(A-4-4);
    \draw (A-3-7)--(A-4-4);

\end{tikzpicture}
\end{center}
\caption{\label{G7-lattice} The lattice of normal reflection subgroups 
	of $G_{7}=G_{7}^{(2,3,3)}$ (those in grey have smaller collineation groups, i.e., split orbits).
We observe that $G_6$ and $G_4$ appear twice as
	normal subgroups.}
\end{figure}
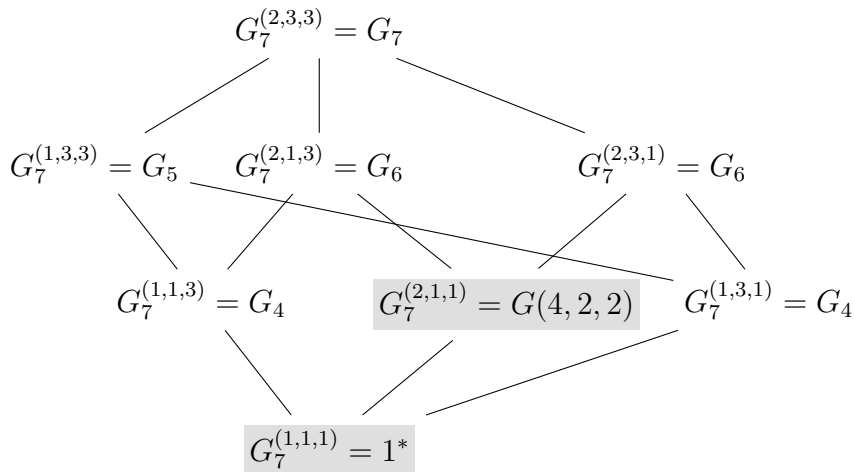

We now consider the maximal reflection group $G_{19}$. 
Define a reflection of order $5$ by 
$$ M:=\hbox{$(-{\tau\over4}+{\sqrt{1-\tau^2/4}\over2} i)$}
\pmat{-\tau+i&1-\tau\cr \tau-1&-\tau-i}, \qquad
\tau={1+\sqrt{5}\over2}. $$ 
and (in the notation of \cite{LT09}, $r=R^2$, $r_1=Z^2$, $r_5=M$)
$$ G_{19}=G_{19}^{(2,3,5)}=\inpro{R^2,Z,M}. $$

\begin{figure}[H]
\begin{center}
\begin{tikzpicture}
    \matrix (A) [matrix of nodes, row sep=1.0cm, column sep = 0.4 cm]
    {
 & $G_{19}^{(2,3,5)}=G_{19}$ & \\
$G_{19}^{(2,3,1)}=G_{21}$ & $G_{19}^{(2,1,5)}=G_{17}$ & $G_{19}^{(1,3,5)}=G_{18}$ \\
$G_{19}^{(2,1,1)}=G_{22}$ & $G_{19}^{(1,3,1)}=G_{20}$ & $G_{19}^{(1,1,5)}=G_{16}$ \\
        & \greybox{$G_{19}^{(1,1,1)}=1^*$} &  \\
    };
    \draw (A-1-2)--(A-2-1);
    \draw (A-1-2)--(A-2-2);
    \draw (A-1-2)--(A-2-3);

    \draw (A-2-1)--(A-3-1);
    \draw (A-2-1)--(A-3-2);
    \draw (A-2-2)--(A-3-1);
    \draw (A-2-2)--(A-3-3);
    \draw (A-2-3)--(A-3-2);
    \draw (A-2-3)--(A-3-3);

    \draw (A-3-1)--(A-4-2);
    \draw (A-3-2)--(A-4-2);
    \draw (A-3-3)--(A-4-2);
\end{tikzpicture}
\end{center}
\vskip -0.6truecm
\caption{\label{G19-lattice} The lattice of normal reflection subgroups for $G_{19}$.
Here the intersection $G_{22}\cap G_{20}\cap G_{16}$ is a collineation preserving normal
	subgroup of $G_{19}$, with centre $\inpro{-1}$.}
\end{figure}
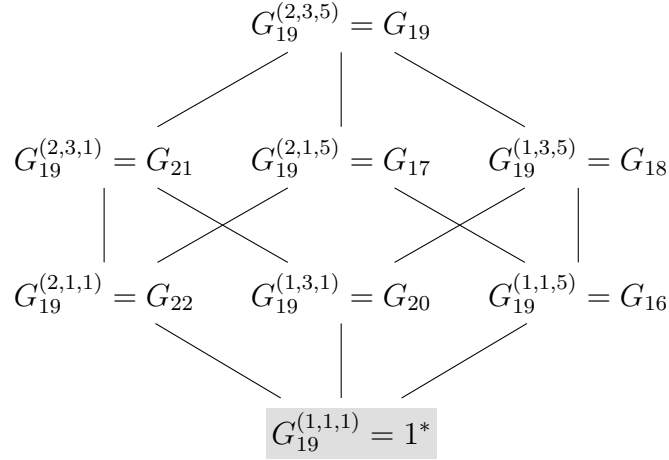

\begin{example} (Normal reflection subgroups of $G_{19}$).
There are eight normal reflection subgroups (see the indices in Figure \ref{G19-lattice}) of
the maximal reflection group
	$$ G_{19}=G_{19}^{(2,3,5)}=\inpro{R^2,Z,M}, \qquad
30\inpro{R^2},20\inpro{Z},12\inpro{M} =30C_2,20C_3,12C_3, $$
with no split orbits and no repeated subgroups. The common collineation
group of the nontrivial normal reflection subgroups is $A_5$.
The first level down in the lattice, we have
\begin{align*}
G_{18} = G_{19}^{(1,3,5)} =\inpro{Z,M}, &\qquad 20\inpro{Z},12\inpro{M} =20C_3,12C_5, \cr
G_{17} = G_{19}^{(2,1,5)} =\inpro{R^2,M}, &\qquad 30\inpro{R^2},12\inpro{M} =30C_2,12C_5, \cr
G_{21} = G_{19}^{(2,3,1)} =\inpro{R^2,Z^M}, &\qquad 30\inpro{R^2},12\inpro{Z^M} =30C_2,20C_3.
\end{align*}
We observe that the reflections $Z\in G_{11}$ and $Z^M\not\in G_{11}$ of order $3$, 
play different roles, since
\begin{equation}
\label{notnormalsubgroupsex}
Z\in G_{21} =\inpro{R^2,Z^M}, \qquad \inpro{R^2,Z}=G_6.
\end{equation}
At the next level, we have
\begin{align*}
G_{16} = G_{19}^{(1,1,5)} =\inpro{M,M^{R^2}}, &12\inpro{M} =12C_5, \cr
G_{20} = G_{19}^{(1,3,1)} =\inpro{Z,Z^M}, &\qquad 20\inpro{Z} =20C_3, \cr
G_{22} = G_{19}^{(2,1,1)} =\inpro{R^2,(R^2)^Z,(R^2)^M}, & \qquad 30\inpro{R^2}= 30C_2, 
\end{align*}
with the identity group below that.
The intersection $G_{16}\cap G_{20} \cap G_{22}$
is a normal collineation preserving subgroup of $G_{19}$ with order $120$
	(which is not a reflection group).
\end{example}

The observation (\ref{notnormalsubgroupsex}) that $G_6$ is a necessarily not normal
subgroup of $G_{21}$ suggests other such subgroup relationships between the normal
subgroups of $G_7$, $G_{11}$ and $G_{19}$. 
We observe that the reflection types of these groups are
$$ G_6:\ 6C_2,4C_3, \qquad  G_{21}:\ 30C_2,20C_3, $$
and so such an inclusion could hold, but not as a normal subgroup. 
By considering such possibilities, it is easy to determine all
such subgroup relationships (see Figure \ref{GjInclusions}). 
We first observe that since $S_4$ is not a subgroup of $A_5$ (by Lagrange's theorem),
the groups with maximal reflection group $G_{11}$ (collineation group $S_4$) 
and with maximal reflection group $G_{19}$ (collineation group $A_5$)
cannot be subgroups of each other. Therefore, only a subgroup of $G_7$ could be
contained in a collineation preserving subgroup of $G_{11}$ or $G_{19}$. 
A basic set of such inclusions (for the generators of Table \ref{GjTablerank2}), 
which implies all others, is
$$ G_7=G_7\cap G_{21} \subset G_{21} \qquad \hbox{($G_7$ has five conjugates in $G_{21}$ and $G_{19}$)}, $$
$$ G_5=G_5\cap G_{20} \subset G_{20} \qquad \hbox{($G_5$ has five conjugates in $G_{20}$ and $G_{19}$)}. $$
In particular, the first inclusion gives our initial 
observation: $G_6\subset G_7\subset G_{21}$, and we observe that
$$ G_{11} \cap G_{19} = G_7. $$

\begin{figure}[!h]
\begin{center}
\begin{tikzpicture}
    \matrix (A) [matrix of nodes, row sep=0.5cm, column sep = 0.2 cm]
	{ &&& \greybox{$G_{11}$} &&&& & && \greybox{$G_{19}$} &\\
	& $G_{15}$ && $G_{9}$ && $G_{10}$ &&&& $G_{21}$ & $G_{17}$ & $G_{18}$   \\
	$G_{14}$ && $G_{13}$ && $G_{8}$ && \greybox{$G_7$} &&& $G_{22}$ & $G_{20}$ & $G_{16}$ \\
	& $G_{12}$ && && $G_5$&& $G_6$ && &&  \\
	& && && & $G_4$ && &&& \\
    };
    \draw (A-1-4)--(A-2-2); \draw (A-1-4)--(A-2-4); \draw (A-1-4)--(A-2-6);
    \draw (A-2-2)--(A-3-1); \draw (A-2-2)--(A-3-3); \draw (A-2-4)--(A-3-3);
    \draw (A-2-4)--(A-3-5); \draw (A-2-6)--(A-3-5); \draw (A-2-6)--(A-3-7);
    \draw  (A-2-2)--(A-3-7); \draw (A-3-1)--(A-4-2); \draw (A-3-3)--(A-4-2);
    \draw  (A-3-1)--(A-4-6); \draw (A-3-7)--(A-4-6); \draw (A-3-7)--(A-4-8);
    \draw (A-4-6)--(A-5-7); \draw (A-4-8)--(A-5-7); 
	\draw (A-1-11)--(A-2-10);
    \draw (A-1-11)--(A-2-11); \draw (A-1-11)--(A-2-12); \draw (A-2-10)--(A-3-10);
    \draw (A-2-10)--(A-3-11); \draw (A-2-11)--(A-3-10); \draw (A-2-11)--(A-3-12);
	\draw [red,thick,dashed] (A-4-6)--(A-3-11);
	\draw [red,thick,dashed] (A-2-10)--(A-3-7);
	\draw (A-2-12)--(A-3-12);
	\draw (A-2-12)--(A-3-11);
\end{tikzpicture}
\end{center}
\vskip-0.6truecm
\caption{
\label{GjInclusions}
The inclusions between the primitive complex reflection groups
$G_4,\ldots,G_{22}$. 
The lines indicate normal subgroups of $G_{11}$, $G_7$, $G_{19}$,
and the dashed lines indicate the further inclusions
$G_7\subset G_{21}$ and $G_5\subset G_{20}$,
with $G_7$ and $G_5$ having five conjugates in those subgroups
	of $G_{19}$ that contain them (and hence not being normal).}
\end{figure}
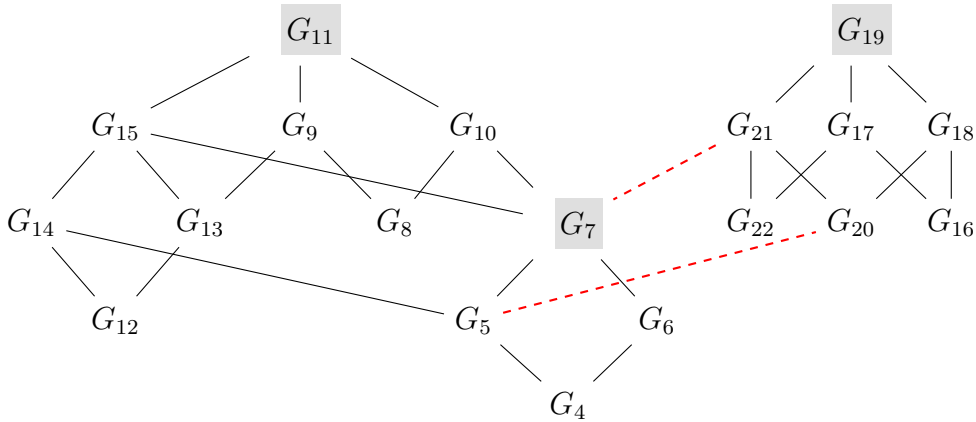

The remaining primitive complex reflection groups have reflections of 
order $2$, $3$ (which are not composite), 
and so only those with more than one reflection orbit can have normal subgroups.
These groups are $G_{26}$ and $G_{28}$ (see Table \ref{GjTablerankdge3}), 
and the corresponding normal subgroups are given in Figure \ref{Maximal2orbits}.

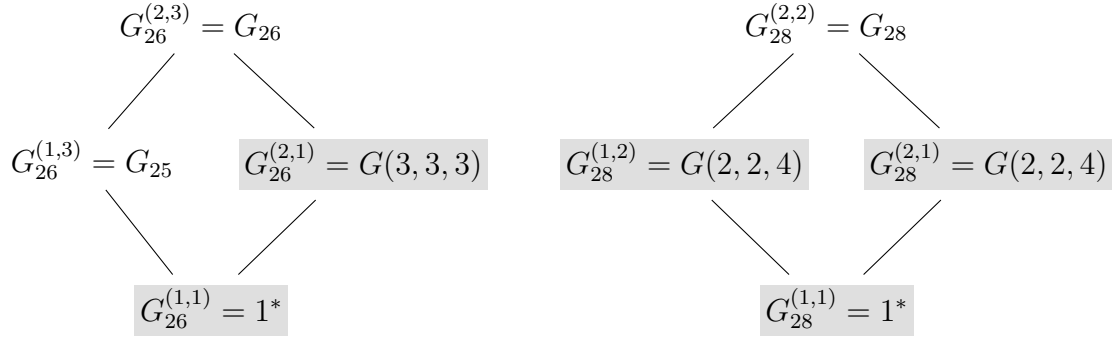
\begin{figure}[H]
\begin{center}
\begin{tikzpicture}
    \matrix (A) [matrix of nodes, row sep=1.0cm, column sep = -1.2 cm]
    {
            & $G_{26}^{(2,3)}=G_{26}$ & \\
  \quad $G_{26}^{(1,3)}=G_{25}$ \quad & & \greybox{$G_{26}^{(2,1)}=G(3,3,3)$}  \\
        & \greybox{$G_{26}^{(1,1)}=1^*$} &  \\
    };
    \draw (A-1-2)--(A-2-1);
    \draw (A-1-2)--(A-2-3);

    \draw (A-2-1)--(A-3-2);
    \draw (A-2-3)--(A-3-2);
\end{tikzpicture}
\begin{tikzpicture}
    \matrix (A) [matrix of nodes, row sep=1.0cm, column sep = -1.2 cm]
    {
 & $G_{28}^{(2,2)}=G_{28}$ & \\
        \greybox{$G_{28}^{(1,2)}=G(2,2,4)$} & & \greybox{$G_{28}^{(2,1)}=G(2,2,4)$} \\
	& \greybox{$G_{28}^{(1,1)}=1^*$} &  \\
    };
    \draw (A-1-2)--(A-2-1);
    \draw (A-1-2)--(A-2-3);
    \draw (A-2-1)--(A-3-2);
    \draw (A-2-3)--(A-3-2);
\end{tikzpicture}
\end{center}
\vskip-0.6truecm
\caption{The normal reflection subgroup lattices for $G_{26}$ (rank $3$)
and $G_{28}$ (rank $4$). 
Here $G_{25}$ is a collineation preserving subgroup of
(its maximal reflection group)$G_{26}$.}
\label{Maximal2orbits}
\end{figure}

We now collect our results to characterise the primitive complex 
reflection groups which are maximal reflection groups.

\begin{theorem}
\label{primitivemaximal}
The primitive complex reflection groups which are maximal reflection groups are
\begin{equation}
\label{maximalprimitive}
G_7,\ G_{11},\ G_{19},\ G_{23},\ G_{24},\ G_{26},\ldots,G_{37},
\end{equation}
with the remaining ones being proper normal subgroups of the maximal reflection groups
that contain them. The only other nontrivial normal reflection subgroups of the 
$17$ groups of (\ref{maximalprimitive}) are the following three
full rank imprimitive complex reflection groups
\begin{equation}
\label{imprimnormalsgs}
G(4,2,2)\lhd G_7, \quad G(4,2,2) \lhd G_{11},  \quad
G(3,3,3)\lhd G_{26}, \quad G(2,2,4) \lhd G_{28},  \quad
\end{equation}
which are precisely the three (imprimitive) complex reflection groups with 
more than one system of imprimitivity.
\end{theorem}

\begin{proof}
We have already observed which primitive complex reflection 
groups are normal subgroups of their maximal reflection group.

We further observe the primitive reflection groups associated with a 
given maximal reflection group are uniquely determined by the
order of their collineation group. Indeed, the only cases 
where the collineation group of maximal reflection groups have equal orders
is for $G_{19}$ and $G_{23}$ (order $60$, isomorphic to $A_5$),
$G_{32}$ and $G_{33}$ (order $25920$, isomorphic), with all these groups
having different ranks.

It is a simple observation that the three imprimitive complex reflection groups
	of (\ref{imprimnormalsgs}) give all the complex reflection groups with 
	multiple systems of imprimitivity, e.g., Theorem 2.16 of \cite{LT09}
	gives these groups together with $G(2,1,2)\cong G(4,4,2)$, which is 
	a real reflection group (see Example \ref{G212example}).
We also note that the last three of (\ref{imprimnormalsgs}) appear as Exercises 
2.8, 2.6, 2.9 (respectively) of \cite{LT09}.
\end{proof}

By considering the collineation groups and reflection types, we came 
across the following observations
(compare the first with the Remark \ref{primitivequatcentres}).

\begin{example}
The group $G_{35}$ is the only primitive complex reflection group with trivial centre, 
i.e., it is an ordinary representation of its abstract group, or, equivalently 
equals its collineation group.
\end{example}

\begin{example}
\label{oneorbitexample}
The $19$ primitive complex reflection groups (of $34$)
\begin{equation}
\label{oneprimeorbit}
G_4,\ G_8,\ G_{12},\ G_{16},\ G_{20},\ G_{22},\ldots,G_{25},\ G_{27},\ G_{29},\ldots,G_{37}
\end{equation}
have no nontrivial normal reflection subgroups, by virtue of each having 
a single reflection orbit corresponding to a reflection of prime order 
(either $2$, $3$ or $5$), see Figure \ref{oneorbitlattice}.
\end{example}

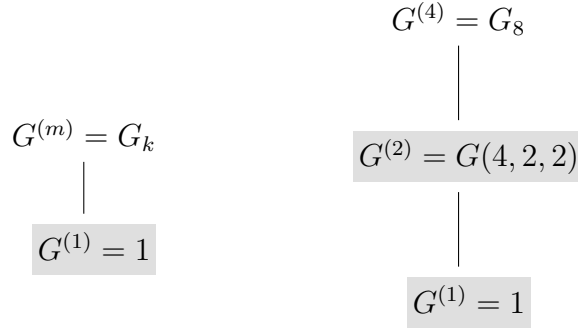
\begin{figure}[h]
\begin{center}
\vskip-0.3 truecm
\begin{tikzpicture}
    \matrix (A) [matrix of nodes, row sep=0.7cm, column sep = -0.9 cm]
    {
            $G^{(m)}=G_k$ \\
            \greybox{$G^{(1)}=1$} \\
             \\
    };
    \draw (A-1-1)--(A-2-1);
\end{tikzpicture}
\qquad\qquad
\begin{tikzpicture}
    \matrix (A) [matrix of nodes, row sep=1.0cm, column sep = -0.9 cm]
    {
            $G^{(4)}=G_8$ \\
            \greybox{$G^{(2)}=G(4,2,2)$} \\
            \greybox{$G^{(1)}=1$} \\
    };
    \draw (A-1-1)--(A-2-1);
    \draw (A-2-1)--(A-3-1);
\end{tikzpicture}
\end{center}
\vskip-0.5 truecm
\caption{
\label{oneorbitlattice}
The lattice of normal reflection subgroups for those $G_k$ of 
(\ref{oneprimeorbit}) with a single orbit of reflections of prime order $m\in\{2,3,5\}$,
and of $G_8$ a single $6C_4$ orbit (reflections of orders $2$ and $4$), with
the normal subgroup $G(4,2,2)$ having the orbit split into three. }
\end{figure}

\begin{example}
In addition to the inclusions of Figure \ref{GjInclusions}, the groups
$G_{27}$ (rank $3$) and $G_{31}$ (rank $4$) contain 
reflection subgroups of the same rank which are not normal and are maximal reflection
groups (see Figure \ref{incidentalinclusions}).
\end{example}


\begin{figure}[H]
\begin{center}
\vskip-0.5 truecm
\begin{tikzpicture}
    \matrix (A) [matrix of nodes, row sep=1.0cm, column sep = 0.2 cm]
    {
            $G_{20}$ \\
            $G_{5}$ \\
    };
    \draw (A-1-1)--(A-2-1);
\end{tikzpicture}
\quad\quad
\begin{tikzpicture}
    \matrix (A) [matrix of nodes, row sep=1.0cm, column sep = 0.2 cm]
    {
            $G_{21}$ \\
            $G_{7}$ \\
    };
    \draw (A-1-1)--(A-2-1);
\end{tikzpicture}
\quad\quad
\begin{tikzpicture}
    \matrix (A) [matrix of nodes, row sep=1.0cm, column sep = 0.2 cm]
    {
            $G_{27}$ \\
            $G_{23}$ \\
    };
    \draw (A-1-1)--(A-2-1);
\end{tikzpicture}
\qquad\qquad
\begin{tikzpicture}
    \matrix (A) [matrix of nodes, row sep=1.0cm, column sep = 0.2 cm]
    {
            & $G_{31}$ & \\
            $G_{28}$ && $G_{29}$ \\
    };
    \draw (A-1-2)--(A-2-1);
    \draw (A-1-2)--(A-2-3);
\end{tikzpicture}
\end{center}
\vskip-0.6truecm
\caption{
\label{incidentalinclusions}
Some incidental inclusions,
i.e., those not occurring as normal subgroups.}
\end{figure}
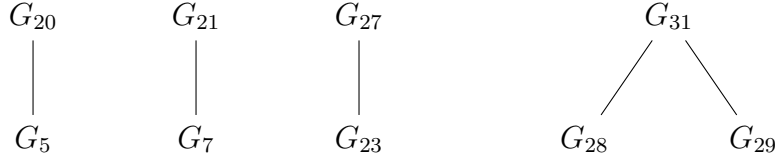

Since the complex reflection groups of Example \ref{oneorbitexample}
which have one orbit are not cyclic groups, a minimal generating 
set of the reflections for them will require at least two reflections
from the orbit.
At the other extreme,
we will say that a reflection group is {\bf nicely generated} if its 
minimal generating sets of reflections are precisely those
obtained by taking a reflection from each orbit.
A rank $d$ complex reflection group has a
minimal generating set of $d$ or $d+1$ reflections, with those
generated by $d$ reflections said to be {\bf well-generated}.


\begin{theorem}
\label{nicelygenthm}
The nicely generated primitive complex reflection groups are precisely
$$ G_7,\, G_{11},\, G_{19}, \quad \hbox{(three generators)}, \qquad
 G_6,\, G_{10},\, G_{18}, \quad \hbox{(two generators)}. $$
	The groups $G_6$, $G_{10}$, $G_{18}$ are also well generated (giving
	starry complex polygons).
\end{theorem}

\begin{proof} 
A minimal generating set of reflections for a complex 
reflection group must contain such a set of reflections
	(see Lemma \ref{abelianisationlemma}),
and it is easily verified that all such subsets generate the group.
\end{proof}


\vfil\eject

\begin{table}[H]
\caption{
\label{GjTablerank2}
The reflection types and collineation groups of the primitive complex reflection groups of rank $2$ (with generators).  }
\smallskip
\begin{tabular}{ |  >{$}l<{$} | >{$}l<{$} | >{$}l<{$} | >{$}l<{$} | >{$}l<{$} | >{$}l<{$} | >{$}l<{$} |  >{$}l<{$} | }
\hline
	&&&&&&&\\[-0.3cm]
	\hbox{ST} & \hbox{Order} & \hbox{Reflection type} & |Z(G)| & \hbox{$K_a$} &
	\hbox{Reflections} & G_{\rm coll} & \hbox{Generators} \\[0.1cm]
\hline
  &&&&&&&\\[-0.3cm]
G_4 &  24 & 4 C_3 & 2 & C_{6} & 3^{8} & A_4 & Z,Z^S \hbox{ or } Z^R, Z^{SR} \\ 
G_5 &  72 & 4 C_3, 4 C_3 & 6 & C_{6} & 3^{16} & A_4 & Z, Z^R \\ 
G_6 &  48 & 6 C_2, 4 C_3 & 4 & C_{4}, C_{12}  & 2^{6} 3^{8} & A_4 & R^2, Z \hbox{ or } R^2, Z^R \\ 
G_7 &  144 & 6 C_2, 4 C_3, 4 C_3 & 12 & C_{12} & 2^6 3^{16} & A_4 & R^2, Z, Z^R \\
  &&&&&&&\\[-0.4cm]
G_8 &  96 & 6 C_4 & 4 & C_{4} & 2^6 4^{12} & S_4 & R, R^F \\
G_9 &  192 & 12 C_2, 6 C_4 & 8 & C_{8} & 2^{18} 4^{12} & S_4 & F, R \\
G_{10} & 288 & 6 C_4, 8 C_3 & 12 & C_{12} & 2^{6} 3^{16} 4^{12} & S_4 & R, Z \\
G_{11} & 576 & 12 C_2, 6 C_4, 8 C_3 & 24 & C_{24} & 2^{18} 3^{16} 4^{12} & S_4 & F, R, Z \\
G_{12} & 48 & 12 C_2 & 2 & C_{2} & 2^{12} & S_4 &  F, F^Z, F^{Z^2} \\
G_{13} & 96 & 12 C_2, 6 C_2 & 4 & C_4, C_8 & 2^{18} & S_4 & F, F^Z, R^2 \\
G_{14} & 144 & 12 C_2, 8 C_3 & 6 & C_6 & 2^{12}3^{16} & S_4 & F, Z \\
G_{15} & 288 & 12 C_2, 6C_2, 8 C_3 & 12 & C_{12}, C_{24}, C_{12} & 2^{18}3^{16} & S_4 & F, R^2, Z \\

	&&&&&&&\\[-0.4cm]
G_{16} & 600 & 12C_5 & 10 & C_{10} & 5^{48} & A_5 & M, M^{R^2} \\
G_{17} & 1200 & 30C_2, 12C_5 & 20 & C_{20} & 2^{30}5^{48} & A_5 & R^2, M \\
G_{18} & 1800 & 20C_3, 12C_5 & 30 & C_{30} & 3^{40}5^{48} & A_5 & Z, M \\
G_{19} & 3600 & 30C_2, 20C_3, 12C_5 & 60 & C_{60} & 2^{30}3^{40}5^{48} & A_5 & R^2, Z, M \\
G_{20} & 360 & 20C_3 & 6 & C_6 & 3^{40} & A_5 & Z, Z^M \\
G_{21} & 720 & 30C_2, 20C_3 & 12 & C_{12} & 2^{30}3^{40} & A_5 & R^2, Z^M \\
G_{22} & 240 & 30C_2 & 4 & C_4 & 2^{30} & A_5 & R^2, (R^2)^Z, (R^2)^M \\ [0.1cm]

\hline
\end{tabular}
\end{table}

\vskip -0.4truecm

\begin{table}[H]
\caption{
\label{GjTablerankdge3}
The reflection types of the primitive complex reflection groups of rank $d\ge 3$.
}
\smallskip
\begin{tabular}{ |  >{$}l<{$} | >{$}c<{$} | >{$}l<{$} | >{$}l<{$} | >{$}l<{$} | >{$}l<{$} | >{$}l<{$} | >{$}l<{$} | }
\hline
&&&&&&&\\[-0.3cm]
\hbox{ST} & \hbox{Rank} & \hbox{Order} & \hbox{Reflection type} & |Z(G)| & \hbox{$K_a$} &
	\hbox{Reflections} & |G_{\rm coll}|\, \hbox{($G_{\rm coll}$ simple)} \\[0.1cm]
\hline
&&&&&&&\\[-0.3cm]
G_{23} & 3 & 120 & 15C_2 & 2 & C_2 & 2^{15} & 60\ \hbox{(simple)} \\
G_{24} & 3 & 336 & 21C_2 & 2 & C_2 & 2^{21} & 168\ \hbox{(simple)} \\
G_{25} & 3 & 648 & 12C_3 & 3 & C_6 & 3^{24}  & 216 \\
G_{26} & 3 & 1296 & 9C_2, 12C_3 & 6 & C_6 & 2^9 3^{24} & 216 \\
G_{27} & 3 & 2160 & 45C_2 & 6 & C_6 & 2^{45} & 360\ \hbox{(simple)} \\
G_{28} & 4 & 1152 & 12C_2, 12C_2 & 2 & C_2 & 2^{24} & 576 \\
G_{29} & 4 & 7680 & 40C_2 & 4 & C_4 & 2^{40} & 1920 \\
G_{30} & 4 & 14400 & 60C_2 & 2 & C_2 & 2^{60} & 7200 \\
G_{31} & 4 & 46080 & 60C_2 & 4 & C_4 & 2^{60} & 11520 \\
G_{32} & 4 & 155520 & 40C_3 & 6 & C_6 & 3^{80} & 25920\ \hbox{(simple)} \\
G_{33} & 5 & 51840 & 45C_2 & 2 & C_6 & 2^{45} & 25920\ \hbox{(simple)} \\
G_{34} & 6 & 39191040 & 126C_2 & 6 & C_6 & 2^{126} & 6531840 \\
G_{35} & 6 & 51840 & 36C_2 & 1 & C_2 & 2^{36} & 51840 \\
G_{36} & 7 & 2903040 & 63C_2 & 2 & C_2 & 2^{63} & 1451520\ \hbox{(simple)} \\
G_{37} & 8 & 696729600 & 120C_2 & 2 & C_2 & 2^{120} & 348364800\ \hbox{(simple)} \\ [0.1 cm]
\hline
\end{tabular}
\end{table}

The nontrivial normal reflection subgroups of the primitive reflection 
groups are primitive (Theorem \ref{Galatticethm}), and hence of full rank. 
We now show that the subgroups of irreducible reflection groups 
which are not of full rank are not normal.

A subgroup of a reflection group is {\bf parabolic} if
it is the pointwise stabiliser of some set of points (and hence 
of their linear span). The stabilisers of larger subspaces give smaller
groups, with the identity stabilising the whole space.
It is a classical result of Steinberg \cite{S64} that parabolic subgroups
of complex reflection groups are reflection groups, 
and this extends to the quaternionic reflection groups \cite{BST23}.
Clearly
\begin{itemize}
\item The parabolic subgroups of $G$ have a smaller rank (by definition).
\item The reflection subgroups $R_a$ for a given root (which pointwise 
fix $a^\perp$) are the maximal parabolic subgroups of rank one.
\item The smallest nontrivial subgroups of the $R_a$ are (by Steinberg's theorem)
	the minimal nontrivial parabolic subgroups of $G$.
\end{itemize}

\begin{proposition}
\label{Parabolicnotnormal}
A parabolic subgroup of an (irreducible) reflection group
is not normal unless it is the trivial subgroup.
\end{proposition}

\begin{proof}
Let $H$ be a parabolic subgroup of an irreducible 
	reflection group $G$ acting on $V$, so that there exists a nonzero $v\in V$ with
$$ hv=v, \qquad \forall h\in H.  $$ 
Suppose that $H$ is normal in $G$. Then
\begin{align*}
g^{-1}hgv = v, \quad\forall g\in G
& \Implies h(gv)=gv, \quad\forall g\in G \cr
&\Implies \hbox{$H$ pointwise fixes $\spam\{gv\}_{g\in G}$}.
\end{align*}
Since $G$ is irreducible, i.e., $\spam\{gv\}_{g\in G}=V$,
this implies that $H$ pointwise fixes $V$,
and hence is the identity group (since it acts faithfully on $V$).
\end{proof}

We observe that the condition $G$ be a reflection group is not required in
the above argument.
In particular, we have the following.

\begin{example}
The reflection subgroups $R_a$ of an irreducible reflection group $G$ are not normal, 
and hence the reflection
orbits $R_a^G$ contain more than one point.

	If $G$ is reducible, then $R_a^G$ can have a single point, e.g., 
take $G$ to be the reducible subgroups $C_4\times C_4$ and $C_2\times C_2$
	of $G(8,2,2)$ in Example \ref{G(8,2,2)example}.
\end{example}

Before considering the normal subgroups of the imprimitive reflection groups,
we use our examples for primitive groups 
to introduce a geometric description of
the reflection type and its relationship with the abelianisation
of a reflection group.

\section{A geometric description of the reflection type}
\label{domino-sect}

For a complex reflection group $G$ with reflection type
$$ n_1C_{k_1}, \ldots\, n_{m} C_{k_m}, $$
we can visualise the reflection orbit $R_{a_j}^G$, of type $n_jC_{k_j}$,
as a $k_j\times n_j$ matrix (block) 
$$ [R_{a_j},R_{a_j}^{g_2},\ldots, R_{a_j}^{g_{n_j}}], $$
where $g_1=1,g_2,\ldots,g_{n_j}$ is a transversal of $G/R_{a_j}$.
The $i$-th $k_j\times1$ column contains the $k_j$ elements of the
cyclic group $R_{a_j}^{g_i}$, $1\le i\le n_j$, which are
$k_j-1$ reflections and the identity, which we give some fixed ordering
with the identity at the bottom (below the dotted red line), e.g., 
for $G=G_{11}=G^{(2,4,3)}$, of reflection type $12C_2,6C_4,8C_3$, we have
$$ \mat{G=G_{11} &
\begin{tikzpicture}[scale=0.30]
\draw[black, thick] (0,0) rectangle (12,2);
\foreach \x in {0,1,2,3,4,5,6,7,8,9,10,11} { \foreach \y in {0,1} {
  \draw[black, thin] (\x+0,\y+0) rectangle (\x+1,\y+1);
  }}
\draw[red, very thick, dashed] (0,1) -- (12,1);
\end{tikzpicture}
&
\begin{tikzpicture}[scale=0.30]
\draw[black, thick] (0,0) rectangle (6,4);
\foreach \x in {0,1,2,3,4,5} { \foreach \y in {0,1,2,3} {
  \draw[black, thin] (\x+0,\y+0) rectangle (\x+1,\y+1);
  }}
\draw[red, very thick, dashed] (0,1) -- (6,1);
\end{tikzpicture}
&
\begin{tikzpicture}[scale=0.30]
\draw[black, thick] (0,0) rectangle (8,3);
\foreach \x in {0,1,2,3,4,5,6,7} { \foreach \y in {0,1,2} {
  \draw[black, thin] (\x+0,\y+0) rectangle (\x+1,\y+1);
  }}
\draw[red, very thick, dashed] (0,1) -- (8,1);
\end{tikzpicture}
\quad 12C_2,6C_4,8C_3.
} $$
We could remove the squares below the red line, but prefer 
to leave them, so that each column gives all the elements of 
a copy of the group
$C_{k_j}$. We have
\begin{itemize}
\item Each small square above the red line corresponds to a reflection in 
the orbit, with those below being the identity,
e.g, $G_{11}$ has $12\cdot 1+6\cdot 3+8\cdot 2=46$ reflections.
\item The conjugation action of $G$ acts transitively on the columns of
	each block, giving the reflection orbit.
\item Taking the $G$-orbit of a subgroup $\hat R_a$ if $R_a$ corresponds
	to removing rows of the corresponding block (with $G$ continuing to act
		transitively on the subcolumns).
\end{itemize}
Each reflection subgroup $H$ of $G$ corresponds its set of reflections, which we will
indicate by enclosing them in a blue box (if there are reflections
taken from the orbit), with the small boxes corresponding to
individual elements not printed for clarity. Moreover, we will indicate the $H$-orbits
orbits of the reflection subgroups (columns) by sub-boxes if they do not equal
the $G$-orbits, e.g., for a split orbit of a normal reflection subgroup.
For example, for $G^{(2,2,3)}=G_{15}$, 
we have
$$ \mat{G^{(2,2,3)}=G_{15} &
\begin{tikzpicture}[scale=0.30]
\draw[black, thick] (0,0) rectangle (12,2);
\draw[blue, thick] (0,0) rectangle (12,2);
\draw[red, very thick, dashed] (0,1) -- (12,1);
\end{tikzpicture}
&
\begin{tikzpicture}[scale=0.30]
\draw[black, thick] (0,0) rectangle (6,4);
\draw[blue, thick] (0,0) rectangle (6,2);
\draw[red, very thick, dashed] (0,1) -- (6,1);
\end{tikzpicture}
&
\begin{tikzpicture}[scale=0.30]
\draw[black, thick] (0,0) rectangle (8,3);
\draw[blue, thick] (0,0) rectangle (8,3);
\draw[red, very thick, dashed] (0,1) -- (8,1);
\end{tikzpicture}
\quad 12C_2,6C_2,8C_3.
} $$
and for $G^{(1,2,3)}=G_7$, which has a split orbit, we have
$$ \mat{G^{(1,2,3)}=G_7 &
\begin{tikzpicture}[scale=0.30]
\draw[black, thick] (0,0) rectangle (12,2);
\draw[red, very thick, dashed] (0,1) -- (12,1);
\end{tikzpicture}
&
\begin{tikzpicture}[scale=0.30]
\draw[black, thick] (0,0) rectangle (6,4);
\draw[blue, thick] (0,0) rectangle (6,2);
\draw[red, very thick, dashed] (0,1) -- (6,1);
\end{tikzpicture}
&
\begin{tikzpicture}[scale=0.30]
\draw[black, thick] (0,0) rectangle (8,3);
\draw[blue, thick] (0,0) rectangle (4,3);
\draw[blue, thick] (4,0) rectangle (8,3);
\draw[red, very thick, dashed] (0,1) -- (8,1);
\end{tikzpicture}
\quad6C_2,4C_3,4C_3 \ \hbox{(split orbit)}. 
} $$
The normal reflection subgroups of $G$ are characterised by 
\begin{itemize}
\item The union of the blue boxes in each orbit is a single box extending
	across the whole block, i.e., the reflections are obtained by 
		removing rows of the block (i.e., taking a subgroup of $R_{a_j}$).

\item Moreover, any partitioning of the large blue box is into blue boxes
of the same size (length and height) corresponds to the splitting of the orbit.

\end{itemize}

In light of the above, we now consider the orbit structure of the 
reflection subgroups which are not normal. For the maximal reflection
subgroups $G(8,2,2)$ and $G(6,2,2)$ of $G_{11}$ 
(see Figure \ref{G11-lattice-maximal}), which are not normal, we have
$$ \mat{G(8,2,2) &
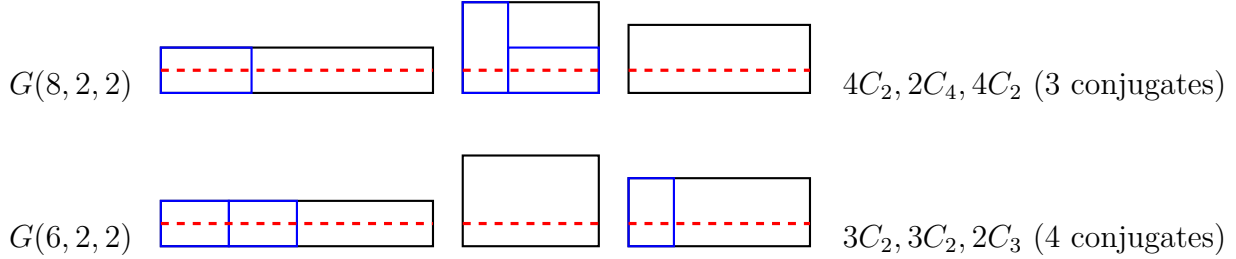
\begin{tikzpicture}[scale=0.30]
\draw[black, thick] (0,0) rectangle (12,2);
\draw[blue, thick] (0,0) rectangle (4,2);
\draw[red, very thick, dashed] (0,1) -- (12,1);
\end{tikzpicture}
&
\begin{tikzpicture}[scale=0.30]
\draw[black, thick] (0,0) rectangle (6,4);
\draw[blue, thick] (0,0) rectangle (2,4);
\draw[blue, thick] (2,0) rectangle (6,2);
\draw[red, very thick, dashed] (0,1) -- (6,1);
\end{tikzpicture}
&
\begin{tikzpicture}[scale=0.30]
\draw[black, thick] (0,0) rectangle (8,3);
\draw[red, very thick, dashed] (0,1) -- (8,1);
\end{tikzpicture}
\quad 4C_2,2C_4,4C_2 \ \hbox{($3$ conjugates)} 
} $$

$$ \mat{G(6,2,2) &
\begin{tikzpicture}[scale=0.30]
\draw[black, thick] (0,0) rectangle (12,2);
\draw[blue, thick] (0,0) rectangle (3,2);
\draw[blue, thick] (3,0) rectangle (6,2);
\draw[red, very thick, dashed] (0,1) -- (12,1);
\end{tikzpicture}
&
\begin{tikzpicture}[scale=0.30]
\draw[black, thick] (0,0) rectangle (6,4);
\draw[red, very thick, dashed] (0,1) -- (6,1);
\end{tikzpicture}
&
\begin{tikzpicture}[scale=0.30]
\draw[black, thick] (0,0) rectangle (8,3);
\draw[blue, thick] (0,0) rectangle (2,3);
\draw[red, very thick, dashed] (0,1) -- (8,1);
\end{tikzpicture}
\quad 3C_2,3C_2,2C_3 \ \hbox{($4$ conjugates)} 
} $$
Since the conjugation action is transitive on columns (and subcolumns), it
is clear that these groups are not normal. Moreover, since
$|G_{11}|/|G(8,2,2)|=9$, by the orbit size theorem, $G(8,2,2)$ must have
$3$ or $9$ conjugates in $G_{11}$, with the different images of the blue boxes giving 
different conjugates.
It so happens, that the images of the $4C_2$ boxes and of
the $2C_4$ boxes are disjoint, and 
we conclude that there are $3$ conjugates of $G(8,2,2)$. 
These are uniquely determined by
the images of the $4C_2$ reflections inside the original $12C_2$ orbit,
which must partition it (and similarly for the $2C_4$ orbit).
Similarly, $G(6,2,2)$ has $4$ conjugates. This leads
to the following notion.

We will call a subset of a reflection orbit of size $n_j$ a 
{\bf domino} if 
\begin{itemize}
\item It consists of columns $[\hat R_{a_j}^{g_1},\ldots,\hat R_{a_j}^{g_\ell}]$
for some fixed nontrivial subgroup $\hat R_{a_j}$ of $R_{a_j}$.
\item Its $G$-orbit 
	partitions $\hat R_{a_j}^G$, equivalently,
	the $G$-orbit of $[R_{a_j}^{g_1},\ldots,R_{a_j}^{g_\ell}]$ partitions
		$R_{a_j}^G$.
\item The reflection group that it generates contains no additional reflections.
\end{itemize}
The number of conjugates of a reflection subgroup constructed using dominos 
is relatively easy to calculate,
since conjugation permutes the set of dominos, i.e., conjugate dominos have
no overlap (by definition). Every reflection orbit has dominos of 
lengths $n_j$ and $1$, and we have
\begin{itemize}
\item A reflection subgroup is normal (has one conjugate) if and only if its
	reflections consist of dominos of length $n_j$.
\item The (parabolic) reflection groups of rank one obtained by taking a single domino
$[\hat R_{a_j}]$ of length $1$ are not normal, and have $n_j$ conjugates.
\end{itemize}
For $G_{11}$ there are also dominos of other lengths, and these are 
counted 
up to conjugation. As with reflection orbits, 
we will say that an $\ga_j\times\ell$ domino 
$[\hat R_{a_j}^{g_1},\ldots,\hat R_{a_j}^{g_\ell}]$, $\ga_j=|\hat R_{a_j}|$, has
{\bf reflection type} $\ell C_{\ga_j}$, and that it is {\bf unique} if its
$G$-orbit gives the only partition of $\hat R_{a_j}^G$ 
into dominos of that size.

\begin{lemma}
\label{dominolemma}
The complex reflection group $G_{11}$, of type $12C_2, 6C_4, 8C_3$, has dominos of the 
following sizes (reflection types)
\begin{align*} 
12C_2: & \qquad C_2, \quad
2C_2 \ \hbox{($6$ dominos)}, \quad
3C_2 \ \hbox{($4$ dominos)}, \quad
4C_2 \ \hbox{($3$ dominos)}, \quad
12C_2, \cr
6C_4: & \qquad C_4,\, C_2, \quad
2C_4,\, 2C_2, \ \hbox{($3$ dominos)}, \quad
6C_4,\, 6C_2, \cr
8C_3: & \qquad C_3, \quad
2C_3, \ \hbox{($4$ dominos)}, \quad
4C_3, \ \hbox{($2$ dominos)}, \quad
8C_3,
\end{align*}
and these are unique, except for the $3C_2$ dominos of which
there are two (partitions). 
\end{lemma}

\begin{proof} By direct computation in Magma.
\end{proof}

\begin{example}
Taking a single domino gives 
a reflection subgroup of the same type.
The one-domino subgroups given by the dominos in Lemma \ref{dominolemma} are
\begin{align*}
12C_2: & \qquad C_2, \quad
C_2\times C_2, \quad
G(3,3,2), \quad
	B_2=G(4,4,2), \quad
G_{12}, \cr
6C_4: & \qquad C_4,\, C_2, \quad
C_4\times C_4,\, C_2\times C_2, \quad
	G_8,\, G(4,2,2), \cr
8C_3: & \qquad C_3, \quad
C_3\times C_3, \quad
G_4, \quad
G_5,
\end{align*}
and their number of conjugates is the number dominos in the partition,
i.e., $n_j$ divided by the domino length. Each appears once (up to conjugacy),
except for $G(3,3,2)$ (the domino is not unique) which occurs twice,
see (\ref{list4}).
\end{example}

It turns out that the reflection orbit of $G(8,2,2)$ which is not a domino, but
a disjoint union of two dominos, is an exceptional case.

\begin{theorem} 
\label{domino-orbits-thm}
Only four of the $38$ reflection subgroups of
$G_{11}$ have reflection orbits which do not consist of (single) dominos,
namely $G(8,2,2)$ 
and its subgroups $G(4,1,2)$, $G(8,4,2)$, $G(8,8,2)$. 
These subgroups all 
have a common reflection orbit which is the disjoint union of two
$2C_2$ dominos sitting inside the $6C_4$ reflection orbit of $G_{11}$,
with all the other reflection orbits being dominos,
and 
	they each have three conjugates in $G_{11}$.
\end{theorem}

\begin{proof} By direct computation in Magma, or by using the generators 
of Table \ref{GjTablerank2}.
\end{proof}

\begin{example} The reflection types of the four reflection subgroups of $G_{11}$
in Theorem \ref{domino-orbits-thm}, which have a $4C_2$ orbit consisting of
	a disjoint union of two $2C_2$ (square) dominos.
$$ \mat{
G(8,2,2) &
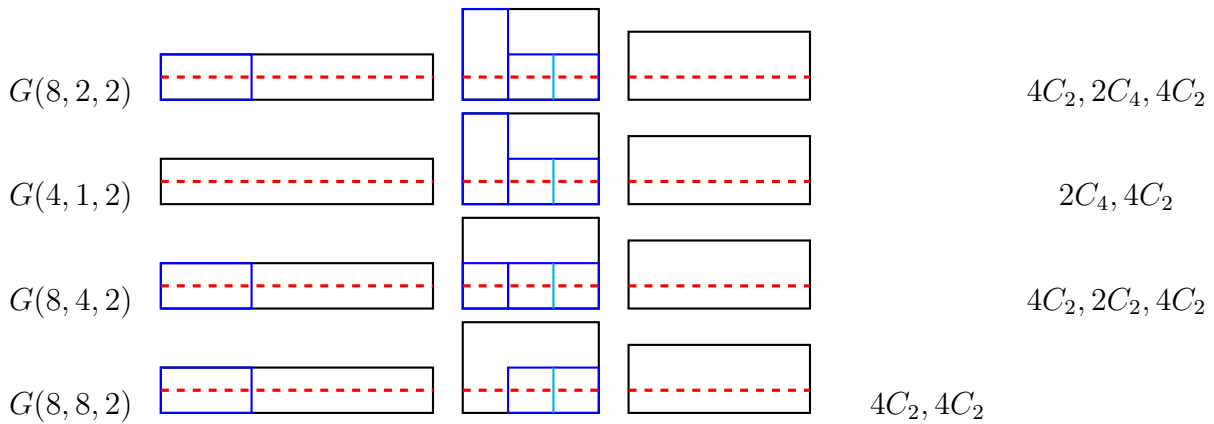
\begin{tikzpicture}[scale=0.30]
\draw[black, thick] (0,0) rectangle (12,2);
\draw[blue, thick] (0,0) rectangle (4,2);
\draw[red, very thick, dashed] (0,1) -- (12,1);
\end{tikzpicture}
&
\begin{tikzpicture}[scale=0.30]
\draw[black, thick] (0,0) rectangle (6,4);
\draw[blue, thick] (0,0) rectangle (2,4);
\draw[blue, thick] (2,0) rectangle (6,2);
\draw[red, very thick, dashed] (0,1) -- (6,1);
\draw[cyan, thick] (4,0) -- (4,2);
\end{tikzpicture}
&
\begin{tikzpicture}[scale=0.30]
\draw[black, thick] (0,0) rectangle (8,3);
\draw[red, very thick, dashed] (0,1) -- (8,1);
\end{tikzpicture}
&\quad 4C_2,2C_4,4C_2 
\cr
G(4,1,2) &
\begin{tikzpicture}[scale=0.30]
\draw[black, thick] (0,0) rectangle (12,2);
\draw[red, very thick, dashed] (0,1) -- (12,1);
\end{tikzpicture}
&
\begin{tikzpicture}[scale=0.30]
\draw[black, thick] (0,0) rectangle (6,4);
\draw[blue, thick] (0,0) rectangle (2,4);
\draw[blue, thick] (2,0) rectangle (6,2);
\draw[red, very thick, dashed] (0,1) -- (6,1);
\draw[cyan, thick] (4,0) -- (4,2);
\end{tikzpicture}
&
\begin{tikzpicture}[scale=0.30]
\draw[black, thick] (0,0) rectangle (8,3);
\draw[red, very thick, dashed] (0,1) -- (8,1);
\end{tikzpicture}
&\quad 2C_4,4C_2 
\cr
G(8,4,2) &
\begin{tikzpicture}[scale=0.30]
\draw[black, thick] (0,0) rectangle (12,2);
\draw[blue, thick] (0,0) rectangle (4,2);
\draw[red, very thick, dashed] (0,1) -- (12,1);
\end{tikzpicture}
&
\begin{tikzpicture}[scale=0.30]
\draw[black, thick] (0,0) rectangle (6,4);
\draw[blue, thick] (0,0) rectangle (2,2);
\draw[blue, thick] (2,0) rectangle (6,2);
\draw[red, very thick, dashed] (0,1) -- (6,1);
\draw[cyan, thick] (4,0) -- (4,2);
\end{tikzpicture}
&
\begin{tikzpicture}[scale=0.30]
\draw[black, thick] (0,0) rectangle (8,3);
\draw[red, very thick, dashed] (0,1) -- (8,1);
\end{tikzpicture}
&\quad 4C_2,2C_2,4C_2 
\cr
G(8,8,2) &
\begin{tikzpicture}[scale=0.30]
\draw[black, thick] (0,0) rectangle (12,2);
\draw[blue, thick] (0,0) rectangle (4,2);
\draw[red, very thick, dashed] (0,1) -- (12,1);
\end{tikzpicture}
&
\begin{tikzpicture}[scale=0.30]
\draw[black, thick] (0,0) rectangle (6,4);
\draw[blue, thick] (2,0) rectangle (6,2);
\draw[red, very thick, dashed] (0,1) -- (6,1);
\draw[cyan, thick] (4,0) -- (4,2);
\end{tikzpicture}
&
\begin{tikzpicture}[scale=0.30]
\draw[black, thick] (0,0) rectangle (8,3);
\draw[red, very thick, dashed] (0,1) -- (8,1);
\end{tikzpicture}
&\quad 4C_2,4C_2 \hskip5truecm 
} $$
\end{example}

\begin{corollary}
The $G_{11}$-orbit of every reflection orbit of the reflection subgroups of
$G_{11}$ partitions the $G_{11}$ reflection orbit it is 
	contained in, except for the
$4C_2$ reflection orbit of $G(4,1,2)$ and one of the $4C_2$ reflection orbits
of $G(8,2,2)$, $G(8,4,2)$ and $G(8,8,2)$.
\end{corollary}

The reflection subgroups can be constructed from and understood using domino 
decompositions of the reflection orbits. We will not develop the entire 
theory here, but instead give some indicative examples.

The reflection subgroups are defined by their reflection orbits, and their
occurrences and number of conjugacies by the decomposition into dominos
(which is unique), 
e.g., $G(4,2,2)$ occurs twice as a subgroup of $G_{11}$, with one occurrence 
being normal.

\begin{example} The decompositions of the reflections for 
the two occurrences of the subgroup 
$G(4,2,2)$ of $G_{11}$, of reflection type $2C_2,2C_2,2C_2$, into dominos is as follows. 
$$ \mat{G(4,2,2) &
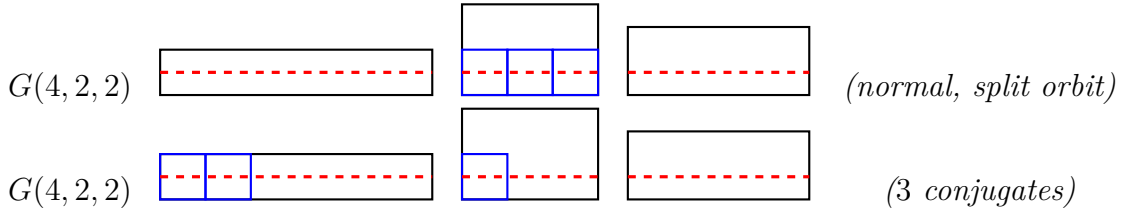
\begin{tikzpicture}[scale=0.30]
\draw[black, thick] (0,0) rectangle (12,2);
\draw[red, very thick, dashed] (0,1) -- (12,1);
\end{tikzpicture}
&
\begin{tikzpicture}[scale=0.30]
\draw[black, thick] (0,0) rectangle (6,4);
\draw[blue, thick] (0,0) rectangle (2,2);
\draw[blue, thick] (2,0) rectangle (4,2);
\draw[blue, thick] (4,0) rectangle (6,2);
\draw[red, very thick, dashed] (0,1) -- (6,1);
\end{tikzpicture}
&
\begin{tikzpicture}[scale=0.30]
\draw[black, thick] (0,0) rectangle (8,3);
\draw[red, very thick, dashed] (0,1) -- (8,1);
\end{tikzpicture}
&\, 
	\hbox{(normal, split orbit)} \cr 
G(4,2,2) &
\begin{tikzpicture}[scale=0.30]
\draw[black, thick] (0,0) rectangle (12,2);
\draw[blue, thick] (0,0) rectangle (2,2);
\draw[blue, thick] (2,0) rectangle (4,2);
\draw[red, very thick, dashed] (0,1) -- (12,1);
\end{tikzpicture}
&
\begin{tikzpicture}[scale=0.30]
\draw[black, thick] (0,0) rectangle (6,4);
\draw[blue, thick] (0,0) rectangle (2,2);
\draw[red, very thick, dashed] (0,1) -- (6,1);
\end{tikzpicture}
&
\begin{tikzpicture}[scale=0.30]
\draw[black, thick] (0,0) rectangle (8,3);
\draw[red, very thick, dashed] (0,1) -- (8,1);
\end{tikzpicture}
	&\, 
	\hbox{($3$ conjugates)}
} $$
We observe the first copy above is normal and the second is not.
\end{example}

To construct reflection subgroups, we can begin with a domino in some reflection 
orbit and consider the domino structure of those reflections (from the other orbits) 
which stabilise it (this is most often a single domino). 
The stabiliser of a domino is generally not a reflection group,
but is does contain unique reflection subgroup generated by its reflections.

\begin{example}
For the dominos of length $1$ given by the reflections for a single root
on the reflection orbits $12C_2,6C_4,8C_3$ of $G_{11}$, the stabiliser groups
have orders $48,96,72$, and contain $2,6,4$ reflections on the given orbit (those
for the root and its conjuguate). These stablising reflections give the
$2C_m$ dominos (of length $2$), which correspond to the reflection subgroups $C_m\times C_m$.
\end{example}

\begin{example} The dominos of length $n_j$ are fixed by all reflections, and so each 
	union of them gives a different group, i.e., the normal reflection subgroups.
\end{example}

\begin{example} There is just one reflection subgroup of $G_{11}$ given by a single domino 
which has two conjugates, 
namely the $4C_3$ domino which gives $G_4$.
Hence, any other reflection subgroups with two conjugates
would be given by the $4C_3$ domino together with dominos of maximum length from
the other orbits which stabilise it. There is one
such case, namely the $4C_3$ domino and the $6C_2$ domino, which gives $G_6$. 
$$ \mat{G_6 &
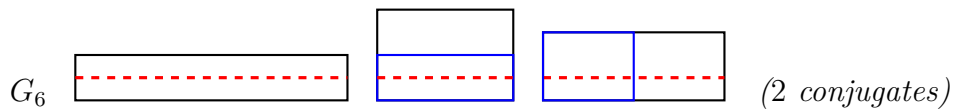
\begin{tikzpicture}[scale=0.30]
\draw[black, thick] (0,0) rectangle (12,2);
\draw[red, very thick, dashed] (0,1) -- (12,1);
\end{tikzpicture}
&
\begin{tikzpicture}[scale=0.30]
\draw[black, thick] (0,0) rectangle (6,4);
\draw[blue, thick] (0,0) rectangle (6,2);
\draw[red, very thick, dashed] (0,1) -- (6,1);
\end{tikzpicture}
&
\begin{tikzpicture}[scale=0.30]
\draw[black, thick] (0,0) rectangle (8,3);
\draw[red, very thick, dashed] (0,1) -- (8,1);
\draw[blue, thick] (0,0) rectangle (4,3);
\end{tikzpicture}
&\, 
        \hbox{($2$ conjugates)}
} $$
Here the reflections which stabilise the $4C_3$ domino 
are the $8C_3$ reflection orbit it belongs to and the $6C_2$ domino
inside the $6C_4$ orbit, a union of dominos of maximum length, which 
therefore generate a normal reflection subgroup.
\end{example}

\vfil\eject
\section{The abelianisation of a reflection group and its reflection type}

Here we consider the close 
relationship between the reflection type of a complex
reflection group $G$ and its abelianisation $G_{\rm ab}$ (\cite{B11}
has a very general development including the relationship with braid groups
and Stanley-Springer’s Theorem).
As a consequence, we explain why the normal reflection subgroups $G^\ga$ of $G$ 
are all different (Example \ref{Galphareflects}).

Let $G'=[G,G]$ be the (derived) commutator subgroup of a finite group $G$.
Then the {\bf abelianisation} of $G$ is the surjective homomorphism (or its image)
$$ G\to {G\over [G,G]}, $$
where $G_{\rm ab}:={G\over [G,G]}$ 
is the largest quotient of $G$ by a normal subgroup which is abelian.

\begin{lemma} 
\label{abelianisationlemma}
If $G=G^{(k_1,\ldots,k_m)}$ is complex reflection group, then its abelianisation
maps reflections to reflections of the same order, with each reflection 
	orbit $n_j C_{k_j}$ mapping onto the internal direct summand
	$C_{k_j}$ of 
	$$ G_{\rm ab}={G\over [G,G]} = C_{k_1}\times\cdots\times C_{k_m}. $$
i.e., the abelianisation of $G^{(k_1,\ldots,k_m)}$ is $C_{k_1}\times\cdots\times C_{k_m}$.
As a consequence,
a generating set of reflections for $G$ must contain at least one reflection 
from each reflection orbit,
	and $[G,G]$ is a reflection free normal subgroup of $G$.
\end{lemma}

\begin{proof}
Consider the group homomorphism (linear character) given by taking the determinant
$$ \det : G\to\CC^*, $$ 
which maps reflections to elements of the same order, which are not in the kernel of $\det$.
We observe that $\det(G)$ is abelian, and hence $\det|_G$ factors through $G_{\rm ab}$.

Since the abelianisation maps the conjugacy class of an element to a single element, 
the image of a sequence of reflections $r_j\in n_j C_{k_j}$ of order $k_j$, gives
	commuting elements of orders $k_1,\ldots,k_m$ which generate ${G\over [G,G]}$,
	which is therefore equal to $C_{k_1}\times\cdots\times C_{k_m}$.

By taking the natural action of $C_{k_1}\times\cdots\times C_{k_m}$ on $\CC^m$, the 
images of the reflections can be viewed as reflections.
Indeed, if $G=C_{k_1}\times\cdots\times C_{k_m}$ then $G=G^{(k_1,\ldots,k_m)}$.
\end{proof}

\begin{example}
\label{Galphareflects}
For a complex reflection group $G=G^{(k_1,\ldots,k_m)}$, 
the image of the generating reflections for the normal reflection subgroup 
$N=G^{(\ga)}=G^{(\ga_1,\ldots,\ga_m)}$ of (\ref{Complexsubgrouporderlabels})
under the abelianisation of $G$ gives the subgroup 
$C_{\ga_1}\times\cdots\times C_{\ga_m}$ of $G_{\rm ab}=C_{k_1}\times\cdots\times C_{k_m}$, 
so that different choices of $\ga$ give different normal subgroups $G^{\ga}$.

	The reflections in $G^{\ga}$ are precisely $((R_{a_j})^{k_j/\ga_j})^G$,
and so the subgroup $C_{\ga_1}\times\cdots\times C_{\ga_m}$ of $G_{\rm ab}$ above
is the abelianisation of $G^\ga$ if and only if $G^\ga$ has no split orbits, i.e.,
there is equality in 
$$ |(G^\ga)_{\rm ab}| \ge \ga_1\ga_2\cdots\ga_m. $$

\end{example}

Thus we may add another equivalence to Theorem \ref{splitorbitcollineation}.

\begin{corollary}
\label{abeliansplitorbitscor}
A normal reflection subgroup $N=G^{\ga}$ of a complex reflection group $G=G^{(k_1,\ldots,k_m)}$ 
is collineation preserving (has no split orbits) if and only of
\begin{equation}
\label{abelianisationcdn}
|(G^\ga)_{\rm ab}| = \ga_1\ga_2\cdots\ga_m.
\end{equation}
\end{corollary}

\begin{proof}
We have observed in Example \ref{Galphareflects}, 
that $\hat R_{a_j}^G$ of (\ref{Complexsubgrouporderlabels}) is split orbit
if and only if its image under the abelianisation of $G$ is not cyclic, 
	i.e., the image is $C_{\ga_j}^{\ell_j}$, where $\ell_j$ is the number of $N$-orbits
	that $\hat R_{a_j}^G$ splits into.
Indeed,
$$ (G^\ga)_{\rm ab} = C_{\ga_1}^{\ell_1}\times\cdots\times C_{\ga_m}^{\ell_m}
\Implies |(G^\ga)_{\rm ab}|={\ga_1}^{\ell_1}{\ga_2}^{\ell_2}\cdots{\ga_m}^{\ell_m}
\ge \ga_1\ga_2\cdots\ga_m, $$
where there is equality in $|(G^\ga)_{\rm ab}| \ge \ga_1\ga_2\cdots\ga_m$
above if and only if $\ell_1=\ldots\ell_m=1$, i.e., 
there is no split orbit.
\end{proof}

As an example of the mechanics of Corollary \ref{abeliansplitorbitscor}, we have
$$ G_7=G_{11}^\ga=G_{11}^{(1,2,3)}\lhd G_{11}=G_{11}^{(2,4,3)},\qquad
G_7=G_7^{(2,3,3)}, $$ 
so that
$$ (G_7)_{\rm ab}= C_2\times C_3\times C_3 \ne C_2\times C_3 
=C_{\ga_1}\times\cdots\times C_{\ga_m}, $$
which gives
$$|(G_7)_{\rm ab}|=18>6=\ga_1\ga_2\cdots\ga_m, $$
and hence $G_7$ is not a collineation preserving normal subgroup of $G_{11}$.

\begin{example}
Let $G=G^{(k_1,\ldots,k_m)}$ be a complex reflection group.
If $G$ has more than one reflection orbit, i.e., $m\ge2$, then
the intersection of the reflection subgroups generated by each orbit, i.e.,
\begin{equation}
\label{Hgroup}
H=H_G:=G^{(k_1,1,\ldots,1)}\cap G^{(1,k_2,1,\ldots,1)} \cap \cdots \cap G^{(1,\ldots,1,k_m)},
\end{equation}
is a reflection free normal subgroup of $[G,G]$ and of $G$.
This follows from the fact the abelianisation of $G$ maps each of the
above reflection subgroups to the corresponding internal direct summand 
of $G_{\rm ab}=C_{k_1}\times\cdots\times C_{k_m}$, and hence to $1$, 
so that $H$ is contained in the kernel of the abelianisation map,
which is $[G,G]$. Moreover, $H$ is an intersection of normal subgroups of $G$,
and hence is normal in $G$ (and therefore in $[G,G]$).

Typically, $H$ is not the identity for $G$ nonabelian, 
thereby giving  
an example where the intersection
of normal reflection subgroups is not a reflection group, e.g.,
for $G_{11}$, 
the subgroups generated by the orbits are
$$ G_{12}=\inpro{F^{G_{11}}} = \inpro{F,F^Z,F^{Z^2}},\quad
G_8=\inpro{R^{G_{11}}} = \inpro{R,R^F},\quad
G_5=\inpro{Z^{G_{11}}} = \inpro{Z,Z^R}, $$
and the intersection of any two of them gives $H$, which has order $24$.
It is easily verified that $H$ is a collineation preserving subgroup 
of $G_7$ ($22$ reflections, $0$ hidden reflections), 
which has $22$ hidden reflections (and $0$ reflections),
with $H=[G_{11},G_{11}]$, and hence
	$$ {G_{11}\over H}=C_2\times C_4\times C_3. $$

It may be that $H$ is a proper subgroup of $[G,G]$, thereby giving a 
nonabelian quotient, e.g., for the primitive reflection groups 
$G=G_{13},G_{15},G_{26},G_{28}$, 
we have  
\begin{align*}
G_{13}: & \qquad {G/ H}\cong G(6,6,2), \quad {G/ [G,G]} \cong C_2\times C_2, \cr
G_{15}: & \qquad {G/ H}\cong G(6,2,2), \qquad {G/ [G,G]} \cong C_2\times C_2\times C_3, \cr
G_{26}: & \qquad {G/ H}\cong C_2\times G_4, \qquad {G/ [G,G]} \cong C_2\times C_3, \cr
G_{28}: & \qquad {G/ H}\cong C_2\times G(3,1,2),\qquad {G/ [G,G]} \cong C_2\times C_2.
\end{align*}
The quotients $G/H$ above were identified with reflection groups as abstract groups,
which suggests the normal subgroups $H$ may be examples of the 
``bon sous-groupes distingu\'es'' of \cite{BBR02} 
discussed in Section \ref{quotientsSection}.

Interestingly, there is a collineation preserving reflection-free 
normal subgroup of $G_{11}$ of order $48$ (with $46$ hidden reflections),
which contains the group $H$ of (\ref{Hgroup}).
We have no pursued such ``reflection-free reflection groups''
	(though there are other cases, e.g., for $G_{19}$ and $G_{33}$).
\end{example}

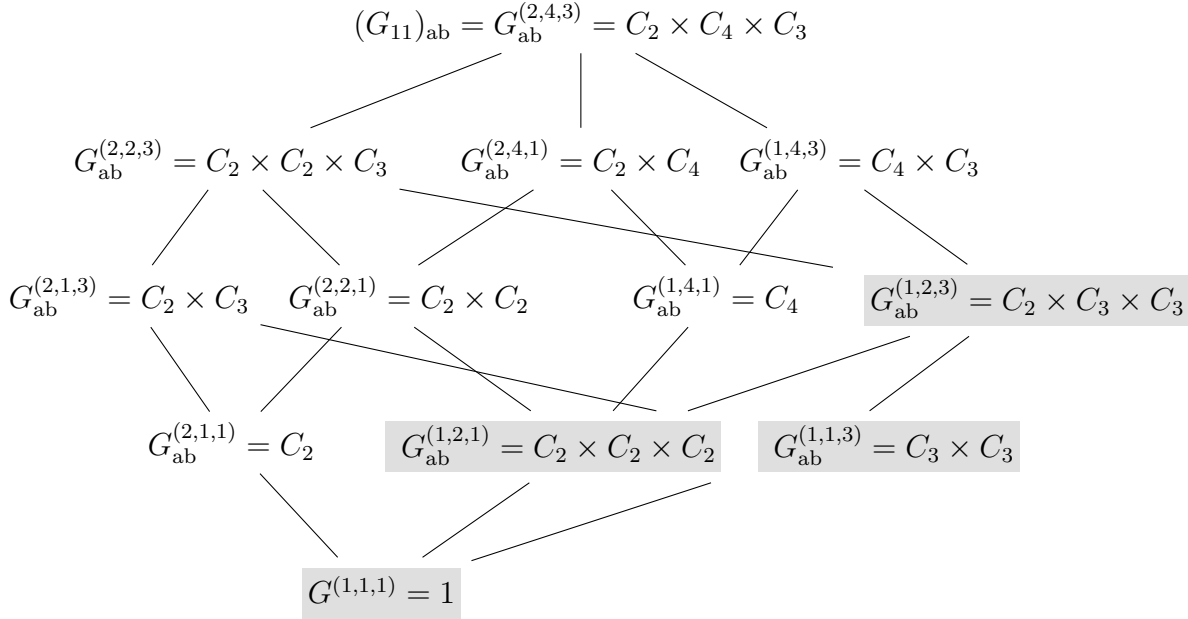
\begin{figure}[h]
\begin{center}
\begin{tikzpicture}
    \matrix (A) [matrix of nodes, row sep=1.0cm, column sep = -2.6 cm]
    {
	    &&& $(G_{11})_{\rm ab}=G_{\rm ab}^{(2,4,3)}=C_2\times C_4\times C_3$ &&& \\
	& $G_{\rm ab}^{(2,2,3)}=C_2\times C_2\times C_3$ && $G_{\rm ab}^{(2,4,1)}=C_2\times C_4$ && \hskip 1truecm $G_{\rm ab}^{(1,4,3)}=C_4\times C_3$ & \\
	$G_{\rm ab}^{(2,1,3)}=C_2\times C_3$ && \hskip 1truecm $G_{\rm ab}^{(2,2,1)}=C_2\times C_2$ && $G_{\rm ab}^{(1,4,1)}=C_4$ && \greybox{$G_{\rm ab}^{(1,2,3)}=C_2\times C_3\times C_3$} \\
& $G_{\rm ab}^{(2,1,1)}=C_2$ && \hskip-1truecm \greybox{ $G_{\rm ab}^{(1,2,1)}=C_2\times C_2\times C_2$ } && \hskip 1.5truecm \greybox{ $G_{\rm ab}^{(1,1,3)}=C_3\times C_3$ } & \\
	&& \greybox{$G_{\rm}^{(1,1,1)} =1$} &&&& \\
    };
    \draw (A-1-4)--(A-2-2); \draw (A-1-4)--(A-2-4);
    \draw (A-1-4)--(A-2-6); \draw (A-2-2)--(A-3-1);
    \draw (A-2-2)--(A-3-3); \draw (A-2-4)--(A-3-3);
    \draw (A-2-4)--(A-3-5); \draw (A-2-6)--(A-3-5);
    \draw (A-2-6)--(A-3-7); \draw (A-2-2)--(A-3-7);
    \draw (A-3-1)--(A-4-2); \draw (A-3-3)--(A-4-2); \draw (A-3-3)--(A-4-4); \draw (A-3-5)--(A-4-4); \draw (A-4-2)--(A-5-3);
    \draw (A-4-4)--(A-5-3);
    \draw (A-3-7)--(A-4-4);
    \draw (A-3-1)--(A-4-6);
    \draw (A-3-7)--(A-4-6);
    \draw (A-4-6)--(A-5-3);
\end{tikzpicture}
\end{center}
\vskip-0.5truecm
\caption{\label{G11-lattice-abelianisation}
The abelianisations of the normal reflection subgroups of $G_{11}$ placed 
in the lattice of Figure \ref{G11-lattice}. Note the subgroups $G^\ga\ne1$
which are not collineation preserving (i.e., with split orbits) in grey,
which are characterised by the inequality $|G^\ga_{\rm ab}|>\ga_1\ga_2\ga_3$.
}
\end{figure}

For the quaternionic reflection groups, 
the Lemma \ref{abelianisationlemma} does not hold.
Indeed, the abelianisation may be trivial.

\begin{example}
The primitive quaternion reflection groups $G=O_1,O_2,P_2,R,S_3$ of
\cite{C80}
have $G_{\rm ab}=1$ (see Table \ref{PrimitiveQuatOrbits}).
Groups such as these which equal their commutator, i.e., $G=[G,G]$, 
and hence have no nontrivial abelian quotients,
are said to be {\bf perfect}. The given reflection groups are not simple,
though the collineation groups for $O_1,O_2,R$ are. 
\end{example}

\section{The 
imprimitive complex reflection groups}

We now seek to identify the imprimitive complex reflection groups which are maximal
reflection groups, and their normal reflection subgroups 
(which are imprimitive, but not necessarily irreducible).
We first determine the reflection type of 
$$ G(m,p,d), \qquad p\divides m, $$
which acts on $\CC^d$. 
Let $\go:=e^{2\pi i\over m}$, a primitive $m$-th root of unity.
We recall that the 
the reflections in $G(m,p,d)$ have two types, i.e.,
those with root $e_j$, the $j$-th standard basis vector,
which gives the diagonal reflections
\begin{equation}
\label{diagonalreflections}
r_{e_j,\xi}=\diag(1,\ldots,1,\xi,1,\ldots,1), \qquad \xi^{m\over p}=1,
\end{equation}
where $\xi$ is a power of $\go^p$, i.e., is an $({m\over p})$-th root of unity,
and those of order $2$ with roots 
$\zeta e_j-e_k$, $\zeta^m=1$,
in the ${1\over2}d(d-1)$ subspaces $\spam\{e_j,e_k\}$, $j\ne k$, 
for which
(\ref{raxidefn}) gives
\begin{equation}
\label{offdiagonalreflections}
r_{\zeta e_j-e_k,-1}\Big|_{\spam\{e_j,e_k\}}
=\pmat{1&0\cr 0&1}-{2\over2}\pmat{\zeta\cr-1}\pmat{\overline{\zeta}&-1}
=\pmat{0&\zeta\cr \overline{\zeta}&1}, \qquad \zeta^m=1, 
\end{equation}
i.e., $\zeta$ is an $m$-th root of unity.

\begin{lemma}
\label{G(m,p,d)-orbits}
The rank $2$ (imprimitive) reflection group $G(m,p,2)$ 
can have one, two or three orbits of reflections, as follows
\begin{align*}
m=p: & \qquad mC_2 \quad\hbox{\rm ($p$ odd)}, \qquad\qquad\
{m\over2}C_2, {m\over2}C_2 \quad\hbox{\rm ($p$ even)}, \cr
m\ne p: & \qquad mC_2, 2C_{m\over p} 
\quad\hbox{\rm ($p$ odd)}, \qquad {m\over2}C_2, {m\over2}C_2, 2C_{m\over p} 
	\quad\hbox{\rm ($p$ even)},
\end{align*}
and the rank $d\ge3$ reflection group $G(m,p,d)$ 
can have one or two orbits of reflections, as follows
\begin{align*}
        m=p: & \qquad {1\over2}d(d-1)m C_2, \cr
m\ne p: & \qquad {1\over2}d(d-1)mC_2, dC_{m\over p}.
\end{align*}
In summary, we have
\begin{equation}
\label{G(n,p,d)ab}
G(m,p,d)_{\rm ab} = C_2\times C_{\gcd(2,p)} \times C_{m\over p},
\end{equation}
\end{lemma}

\begin{proof}
It suffices to find which reflection orbit the reflections of (\ref{diagonalreflections})
and (\ref{offdiagonalreflections}) lie on.
We observe that $g_{jk}:=r_{e_j-e_k,-1}\in G(m,p,d)$ is the permutation matrix which exchanges 
$e_j$ and $e_k$, and so 
	$$ \pmat{0&1\cr1&0}^{-1}\pmat{\xi&0\cr0&1} \pmat{0&1\cr1&0} = \pmat{1&0\cr0&\xi}
	\Implies g_{jk}^{-1} r_{e_j,\xi} g_{jk} = r_{e_k,\xi}, $$
and hence there is a single orbit of the $d$ reflection subgroups $C_{m\over p}$ ($m\ne p$) for
the diagonal reflections, and these reflections generate the
	reducible reflection subgroup 
	$(C_{m\over p})^d$.

For the ${d\choose 2}m$ reflections (\ref{offdiagonalreflections}) of order $2$, 
we first consider the case $d=2$. Now
$$ \pmat{\xi&0\cr0&1}^{-1}\pmat{0&1\cr1&0}\pmat{\xi&0\cr0&1}
=\pmat{0&\overline{\xi}\cr\xi&0}, \quad
\pmat{0&\zeta\cr\overline{\zeta}&0}^{-1}\pmat{0&\overline{\xi}\cr\xi&0}
\pmat{0&\zeta\cr\overline{\zeta}&0} = \pmat{0&\zeta^2\xi\cr\overline{\zeta}^2\overline{\xi}&0}, $$
so that the orbit of $r_{e_1-e_2,-1}$ is
$$ \pmat{0&1\cr1&0}^{G(m,p,2)} =
\big\{ \pmat{0&\zeta\cr\overline{\zeta}&0}: \zeta=\go^{2j+pk},j,k\in\ZZ_m\big\}, $$
which is all the reflections of type (\ref{offdiagonalreflections}), 
unless $p$ is even (and so $m$ is even), in which case there are two orbits, i.e.,
$$ \pmat{0&1\cr1&0}^{G(m,p,2)} =
\bigl\{ \pmat{0&\go^{2j}\cr\overline{\go}^{2j}&0}\bigr\}_{1\le j\le{m\over2}} ,\qquad
\pmat{0&\go\cr\overline{\go}&0}^{G(m,p,2)} =
\bigl\{ \pmat{0&\go^{2j-1}\cr\overline{\go}^{2j-1}&0}\bigr\}_{1\le j\le{m\over2}}. $$

For $d\ge3$ and any value of $p$, let $V:=\spam\{e_1,e_2,e_3\}$ and
$$ g:= (r_{e_2-e_3,-1})(r_{\zeta e_2-e_3,-1})(r_{e_1-e_2,-1}), \quad \zeta^m=1,
	\qquad
g\bigl|_V = \pmat{0&1&0\cr\overline{\zeta}&0&0\cr0&0&\zeta}. $$
Then 
$$ g\bigl|_V (r_{\zeta e_1-e_2,-1}) g^{-1}\bigl|_V
	= g|_V \pmat{0&\zeta&0\cr\overline{\zeta}&0&0\cr0&0&1} g^{-1}|_V
	= \pmat{0&1&0\cr1&0&0\cr0&0&1}=r_{e_1-e_2,-1}, $$
so that $r_{\zeta e_1-e_2,-1}$ is in the orbit of $r_{e_1-e_2,-1}$.
Since all the permutation matrices are contained in $G(m,p,d)$,
$r_{\zeta e_1-e_2,-1}$ can be conjugated to any $r_{\zeta e_j-e_k,-1}$,
and so there is a single orbit of reflections of the type (\ref{offdiagonalreflections}).
\end{proof}

We observe that for $m=p$, we have $C_{m\over p}=1$, 
which contains no reflections, 
and so the separate treatment of the cases $m=p$ and $m\ne p$ in Lemma \ref{G(m,p,d)-orbits}
is simply for clarity.


\begin{example}
From Lemma \ref{G(m,p,d)-orbits}, we may write $G(m,p,d)=G^{(k_1,\ldots,k_j)}$,
$j\in\{1,2,3\}$,
and therefore construct the lattice of its normal reflection subgroups
	(see, e.g., Figures \ref{G822-lattice}, \ref{G2p22-lattice}).
For the six cases, in the order presented, we have
$$ 
G^{(2)}, \quad G^{(2,2)}, \qquad
G^{(2,{m\over p})}, \quad G^{(2,2,{m\over p})}, \qquad
G^{(2)}, \quad G^{(2,{m\over p})}. 
$$
In this way, many familiar facts about $G(m,p,d)$ and its normal reflection subgroups
can be obtained and understood in a transparent way, e.g.,
from the exercises of \cite{LT09}, we have
\begin{itemize}
\item $G(m,p,d)\lhd G(m,1,d)$ (Exer.\ 2.5).
\item $G(m,p,2)\lhd G(2m,2,2)$ (Exer.\ 2.7).
\item For $d\ge3$, the reflections in $G(m,m,d)$ form a single conjugacy class (Exer.\ 2.14).
\item Every reflection in $G(m,m,d)$ has order $2$ if and only 
	$m=p$ or $2p$ (Exer.\ 2.15).
\end{itemize}
From the first two cases, we have that
	$G(m,m,2)$,
the dihedral group of order $2m$, $m>2$, has a (nontrivial) normal reflection
subgroup if and only if $m$ is even, i.e., $G({m\over2},{m\over2},2)$ which
occurs twice as a normal subgroup.
\end{example}

\begin{example} We have $G(1,1,n)=S_n$, which has rank $n-1$. 
	Its reflections $r_{e_j-e_k}$ act as transpositions (odd permutations) on $e_1,\ldots,e_n$. Since $G(1,1,n)=G^{(2)}$ it has no normal reflection subgroups.
We observe that for $n\ge 5$ the normal subgroups of $S_n$ are the alternating group
$A_n$ which is not a reflection group, since it contains no reflections.
\end{example}

\begin{example}
The imprimitive 
	reflection group $G(m,p,d)$ has a unique
	maximal 
	abelian normal reflection subgroup $(C_{m\over p})^d$, i.e., the one generated by
	its diagonal reflections. 
\end{example}



\begin{figure}[h]
\begin{center}
\begin{tikzpicture}
    \matrix (A) [matrix of nodes, row sep=1.0cm, column sep = -1.5 cm]
    {
	    && $G^{(2,2,p)}=G(2p,2,2)$ &&& \\
	$G^{(1,2,p)}=G(p,1,2)$ && $G^{(2,1,p)}=G(p,1,2)$ && $G^{(2,2,1)}=G(2p,2p,2)$ & \\
	& \greybox{$G^{(1,1,p)}=C_p\times C_p$} && $G^{(1,2,1)}=G(p,p,2)$ && $G^{(2,1,1)}=G(p,p,2)$ \\
	&&&  \greybox{$G^{(1,1,1)} =1^*$} && \\
    };
    \draw (A-1-3)--(A-2-1);
    \draw (A-1-3)--(A-2-3);
    \draw (A-1-3)--(A-2-5);
    \draw (A-2-1)--(A-3-2);
    \draw (A-2-1)--(A-3-4);
    \draw (A-2-3)--(A-3-2);
    \draw (A-2-3)--(A-3-6);
    \draw (A-2-5)--(A-3-4);
    \draw (A-2-5)--(A-3-6);
    \draw (A-3-2)--(A-4-4);
    \draw (A-3-4)--(A-4-4);
    \draw (A-3-6)--(A-4-4);
\end{tikzpicture}
\end{center}
\vskip -0.7 truecm
\caption{\label{G2p22-lattice}
 The lattice of normal reflection subgroups of 
	$G^{(2,2,p)}=G(2p,2,2)$, for $p$ prime (which is maximal).
Note the repeated normal subgroups $G(p,1,2)$ and $G(p,p,2)$.}
\end{figure}
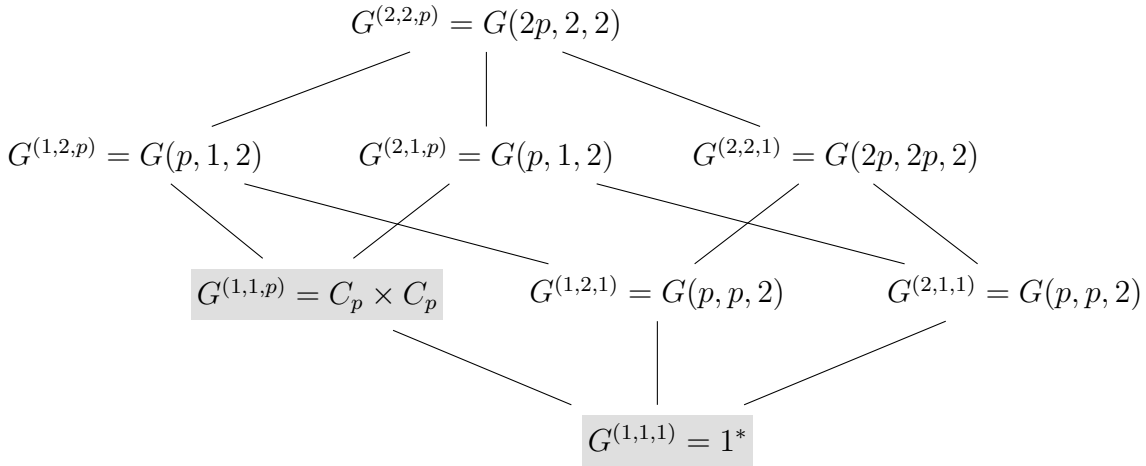

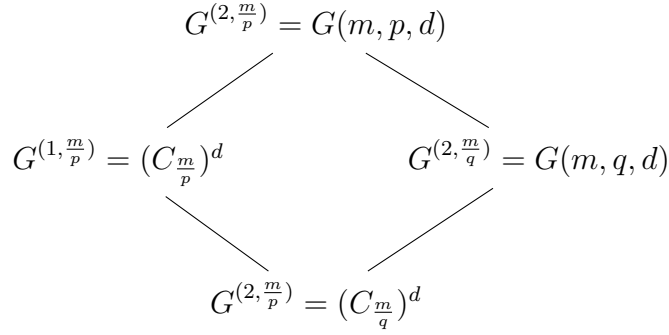
\begin{figure}[h]
\begin{center}
\begin{tikzpicture}
    \matrix (A) [matrix of nodes, row sep=1.0cm, column sep = -0.8 cm]
    {
	    & $G^{(2,{m\over p})}=G(m,p,d) $ & \\
	    $G^{(1,{m\over p})}=(C_{m\over p})^d $ && $G^{(2,{m\over q})}=G(m,q,d) $ \\
	    & $G^{(2,{m\over p})}= (C_{m\over q})^d $ & \\
    };
    \draw (A-1-2)--(A-2-1);
    \draw (A-1-2)--(A-2-3);
    \draw (A-2-1)--(A-3-2);
    \draw (A-2-3)--(A-3-2);
\end{tikzpicture}
\end{center}
\vskip -0.7 truecm
	\caption{\label{G(m,p,d)-lattice}
A sublattice of normal reflection subgroups of
	$G^{(2,{m\over p})}=G(m,p,d)$, $d\ge3$. }
\end{figure}

We now seek those $G(m,p,d)$ which are maximal reflection groups.
By the celebrated Chevalley-Shephard-Todd theorem, a finite group
$G\subset GL(\CC^n)$ of rank $n$ is a complex reflection group if and only if 
its ring of invariants is a polynomial ring, i.e., there
exist algebraically independent homogeneous polynomials
$p_1,\ldots,p_n$ of
degrees $d_1,\ldots,d_n$
for which $\CC[p_1,\ldots,p_n]$ is the ring of polynomial invariants
$\CC[x_1,\ldots,x_n]^G$.
The fundamental invariants $d_1,\ldots,d_n$ are called the
{\bf degrees} of the reflection group $G$. 

The centre of an irreducible complex reflection group $G$ is a group
of scalar matrices, and hence is cyclic, of order
\begin{equation}
\label{centreorder}
	|Z(G)|=\gcd(d_1,\ldots,d_n),
\end{equation}
where $d_1,\ldots,d_n$ are the degrees of $G$,
and 
$$ |G|=d_2d_2\cdots d_n. $$
For $G=G(m,p,d)$, the degrees are
$$ m, 2m,\ldots,(d-1)m, \ dm/p, $$
which gives
$$ |Z(G(m,p,d))|=\gcd(m,d{m\over p}) =\gcd(p{m\over p},d{m\over p})
        = {m\over p} \gcd(p,d). $$
Since $G(m,p,d)$ is irreducible for $(m,p,d)\ne(2,2,2)$, 
its collineation group has order
$$ {|G(m,p,d)|\over|Z(G(m,p,d))|}
= {d! m^{d-1} (m/p)\over \gcd(p,d) (m/p)}
= {d! m^{d-1}\over \gcd(p,d)}, $$
and we obtain in all cases (including $(m,p,d)=(2,2,2)$ and $d=1$) that
\begin{equation}
\label{Gmpdcollinorder}
|G(m,p,d)_{\rm coll}| = {d! m^{d-1}\over \gcd(p,d)}, 
\end{equation}
which gives the collineation preserving inclusion 
\begin{equation}
\label{collinbasicinclusion}
G(m,p,d) \lhd G(m,\gcd(p,d),d).
\end{equation}

In view of (\ref{collinbasicinclusion}), the candidate maximum 
reflection subgroups are 
\begin{equation}
\label{G(m,p,d)candidatemaximal}
G(ka,k,d), \qquad k\divides d, \quad a=1,2,3,\ldots.
\end{equation}
For $d=2$, 
we have the collineation preserving inclusion
(see Exercise 2.7 of \cite{LT09})
\begin{equation}
\label{G(m,1,2)inclustion}
G(m,1,2)\lhd G(2m,2,2),
\end{equation}
and so taking $k=1$ in (\ref{G(m,p,d)candidatemaximal}) 
does not give a maximal reflection group.

For $d=2$ (two degrees), the number of reflections is $d_1+d_2-2$, and so
we can use Proposition \ref{maximalrefcount} to determine when a rank $2$ 
complex reflection group is a maximal reflection group, i.e., when
$$ d_1+d_2-2 = 2\Bigl({d_1d_2\over\gcd(d_1,d_2)}-1\Bigr) 
\Iff d_1=d_2. $$
For primitive groups, we recover the maximal reflection groups
$G_7$, $G_{11}$, $G_{19}$, and for imprimitive groups we have the following.

\begin{example}
\label{G(m,p,2)maximalex}
The degrees of $G(m,p,2)$ are equal, giving a maximal reflection group, 
if and only if
$$ m=2{m\over p} \Iff p=2. $$
Since $p\divides m$, $m$ must be even, and so 
	$$ G(2a,2,2), \qquad a=2,3,4,\ldots  $$
are all the maximal imprimitive reflection groups of rank $2$.
Here we exclude the case $a=1$, which gives a maximal reflection group
	$G(2,2,2)$ which is reducible.

From (\ref{collinbasicinclusion}), we have
$$ G(m,p,2) \lhd G(m,\gcd(p,2),2) \qquad \hbox{(collineation preserving)}, $$
and this and (\ref{G(m,1,2)inclustion}) give
the collineation preserving inclusions
\begin{align*} 
G(m,p,2) &\lhd G(m,1,2) \lhd G(2m,2,2)\qquad \hbox{($m$ odd)}, \cr
G(m,p,2) &\lhd G(m,2,2) \qquad \qquad\qquad\qquad\hbox{($m$ even, $p$ even)}, \cr
G(m,p,2) &\lhd G(m,1,2) \lhd G(2m,2,2) \qquad \hbox{($m$ even, $p$ odd)}.
\end{align*}
Therefore the maximal reflection group which contains $G(m,p,2)$
as a normal subgroup is as follows
\begin{equation}
G(m,p,2) \lhd G({2m\over\gcd(m,p,2)},2,2) \qquad
	\hbox{(maximal reflection group)}.
\end{equation}
\end{example}

We now present the general case.

\begin{theorem}
\label{imprimitivemaximal}
The imprimitive complex reflection groups which are maximal 
reflection groups are
\begin{align}
\label{maximalimprimitive}
G(2a,2,2), & \qquad d=2, \quad a=2,3,4,\ldots,\cr
G(ka,k,d), & \qquad d\ge3, \quad k\divides d, \quad a=1,2,3,\ldots,
\end{align}
with the remaining ones being proper normal subgroups of the 
maximal reflection groups that contain them, i.e., 
\begin{align} 
\label{G(m,p,d)inclusion}
G(m,p,2) &\lhd G({2m\over\gcd(m,p,2)},2,2), \qquad d=2, \cr
G(m,p,d) &\lhd G(m,\gcd(p,d),d), \qquad\quad d\ge3.
\end{align}
\end{theorem}

\begin{proof} We have already considered the case $d=2$ in 
Example \ref{G(m,p,2)maximalex}. 
We now consider the case $d\ge3$,
which is 
simpler because the reflection orbits 
are not as complicated.

If some of the candidate groups in (\ref{G(m,p,d)candidatemaximal}) 
are not maximal reflection groups, then they appear as a 
collineation preserving normal subgroups of a larger such group,
i.e., there is an analogue of (\ref{G(m,1,2)inclustion}).
Therefore, we now consider normal reflection subgroups of 
the groups $G(ka,k,d)$ of (\ref{G(m,p,d)candidatemaximal}),
as given by the reflection types of Lemma \ref{G(m,p,d)-orbits}.

For $a=1$ ($m=p$), we have $G(k,k,d)=G^{(2)}$, which has 
no nontrivial normal reflection subgroups.
For $a\ge2$ ($m\ne p$), we have
$$ G(ka,k,d)=G^{(2,a)},  $$
giving the normal reflection subgroups
$$ G^{(1,a)}, \qquad\qquad
G^{(2,{a\over r})}=G(ka,kr,d),  \quad r\divides a. $$
The subgroup $G^{(1,a)}$ is generated by the diagonal reflections
in $G(ka,k,d)$, and hence is 
reducible.
By (\ref{Gmpdcollinorder}), the subgroup 
	$G^{(2,{a\over r})}=G(ka,kr,d)=G(kr({a\over r}),kr,d)$
is collineation preserving if and only if
$$ d!{(ka)^{d-1}\over\gcd(k,d)} = d!{(ka)^{d-1}\over\gcd(kr,d)}
\Iff \gcd(k,d)=\gcd(kr,d) .  $$
	For $G(ka,kr,d)$ to have the form of (\ref{maximalimprimitive}), we 
	must have $kr\divides d$, and so 
$$ \gcd(k,d)=\gcd(kr,d) \Implies k=kr \Implies r=1, $$
and so none of the proper normal reflection subgroups of $G(ka,k,d)$ are 
collineation preserving, i.e., (\ref{maximalimprimitive}) gives 
all the maximal reflection groups.

The maximal reflection group for $G(m,p,2)$ was given in Example \ref{G(m,p,2)maximalex}.
For $d\ge 3$, we observe that (\ref{collinbasicinclusion}) gives the desired 
maximal reflection group, since $p\divides m$ gives
$$ G(m,p,d) \lhd G(m,\gcd(p,d),d)
=G( k {m\over\gcd(p,d)} ,k,d), \qquad k=\gcd(p,d). $$
i.e., since $k\divides d$, the larger group in the collineation preserving inclusion above is a maximal reflection group.
\end{proof}

\begin{table}[H]
\caption{
\label{G(m,p,d)TableOrbits}
The reflection types of the imprimitive complex reflection groups. }
\smallskip
\begin{tabular}{ |  >{$}l<{$} | >{$}l<{$} | >{$}l<{$} | >{$}l<{$} | >{$}l<{$} | >{$}l<{$} | }
\hline
&&&&&\\[-0.3cm]
\hbox{ST} & \hbox{Rank} & \hbox{Order} & \hbox{Reflection type} & |Z(G)| & 
	 |G_{\rm coll}| \\[0.1cm]
\hline
&&&&&\\[-0.2cm]
	G(1,1,d),\ d\ge2 & d-1 & d! & {1\over2}d(d-1)C_2 & 1 & d! \\
&&&&&\\[-0.3cm]
G(m,m,2),\ {\scriptstyle\hbox{$m$ odd, $m\ne1$}} & 2 & 2m & mC_2 & 1 & 2m \\
G(m,m,2),\ \hbox{$m$ even, $m\ne2$} & 2 & 2m & {m\over2}C_2,{m\over2}C_2 & 2 &  m \\

G(m,m,2),\ \hbox{$p$ odd, $p\ne m$} & 2 & 2m & mC_2, 2C_{m\over p} & 1 & 2m \\
G(m,p,2),\ \hbox{$p$ even, $p\ne m$} & 2 & 2m & {m\over2}C_2,{m\over2}C_2, 2C_{m\over p} & 2 &  m \\
&&&&&\\[-0.3cm]
G(m,m,d),\ d\ge3 & d & d!  m^{d-1} & {1\over2}d(d-1)mC_2 & \gcd(m,d) & {d! m^{d-1}\over \gcd(m,d)} \\
G(m,p,d),\ d\ge 3, m\ne p & d & d! {m^d\over p} & {1\over2}d(d-1)mC_2, 2C_{m\over p} & {m\over p}\gcd(p,d) & {d! m^{d-1}\over \gcd(p,d)} \\ [0.3 cm]
\hline
\end{tabular}
\end{table}

The maximal reflection groups of (\ref{maximalimprimitive}) have
types 
$$G(2a,2,2)=G^{(2,2,a)}, \qquad G(ka,k,d)=G^{(2,a)}, $$
and hence have $4\gs_0(a)$ and $2\gs_0(a)$ normal reflection subgroups, respectively.

\begin{corollary}
The complex reflection 
group $G(m,p,d)$, $d\ge3$, is a maximal reflection group
if and only $p\divides d$.
\end{corollary}

\begin{proof} Since $p\divides m$, we have $G(m,p,d)=G(p({m\over p}),p,d)$.
\end{proof}

\begin{example} As simple example of (\ref{G(m,p,d)inclusion}), we have the 
collineation preserving inclusion
$$ G(4,2,3) \lhd G(4,\gcd(2,3),3)=G(4,1,3), $$
which shows that the condition $k\divides d$ is necessary to obtain
a maximal reflection group.
\end{example}

\section{The quotient by a normal reflection subgroup}
\label{quotientsSection}

Given that we have calculated the normal reflection subgroups of
the complex reflection groups of rank $d$, 
it is now natural to consider the quotients.
These were shown to be reflection groups (acting on an appropriate
subspace of $\Cd$) by \cite{BBR02}, with the quotients (mostly) calculated
in \cite{AW23}. We now give some details.

For a complex reflection group $G$ acting on $V=\Cd$, \cite{BBR02}
characterised those ``nice'' normal subgroups 
(``bon sous-groupes distingu\'es'') $N$ for which the quotient group $G/N$ acts faithfully on the tangent space to the
quotient variety $V/N$ at the point $0$ as a reflection group,
and that such groups include the normal reflection subgroups.
For $N$ a normal reflection subgroup, \cite{AW20} (proof of 
Theorem 1.2) effectively show that this action is given 
by the action of $G$ on the span of the
subset of the basic $N$-invariant polynomials $f_1,\ldots,f_d$ 
with $R_G(f_j)=0$, where 
$R_G$ the Reynolds operator for $G$, i.e.,
$$ R_G(f):={1\over|G|}\sum_{g\in G} f\circ g. $$
For example, if $N=1$, basic $N$-invariant polynomials are 
$x_1,\ldots,x_d$, and $R_G(x_j)=0$ is equivalent to $\sum_{g\in G}g=0$,
which holds for $G$ irreducible, in which case the action of 
$G$ on $x_1,\ldots,x_d$ is given by its action on $(x_1,\ldots,x_d)$.
We now consider some examples.

Basic invariant polynomials for $G(m,p,d)$ are given by
\begin{equation}
\label{G(m,p,d)invariantpolys}
\sum_{j=1}^d x_j^{mk}, \quad 1\le k\le d-1, \qquad
(x_1x_2\cdots x_d)^{m\over p},
\end{equation}
which have degrees $m,2m,\ldots,(d-1)m$, and $d{m\over p}$.
The ``power sum'' polynomials above can be replaced by 
$e_k(x_1^m,\ldots,x_d^m)$, 
where $e_k$ is the $k$-th elementary symmetric function
(see \cite{KM25} for a discussion about ``good basic invariants'').

\begin{example} 
\label{G26quotexample}
Consider the primitive complex reflection group  $G=G_{26}$ (order $1296$)
and its normal reflection subgroup $N=G(3,3,3)$ (order $54$), which 
are given by 
$$ G=G_{26}=\inpro{r_1,r_2,P_{(12)}}, 
	\qquad N=G(3,3,3)=\inpro{r_3,P_{(12)},P_{(23)}}, $$
where
$$ r_1 = \pmat{\go&0&0\cr 0&1&0\cr 0&0&1}, \quad
	r_2 = {-i\over\sqrt{3}} \pmat{\go&\overline{\go}&\overline{\go}\cr \overline{\go}&\go&\overline{\go}\cr \overline{\go}&\overline{\go}&\go}, \quad
r_3 = \pmat{0&\go &0 \cr \overline{\go}&0&0\cr 0&0&1}, \qquad
	\go:=e^{2\pi i\over3}, $$
$$ P_{(12)} = \pmat{0&1&0\cr 1&0&0\cr 0&0&1}, \quad
P_{(23)} = \pmat{1&0&0\cr 0&0&1\cr 0&1&0}. $$
By (\ref{G(m,p,d)invariantpolys}), 
invariant polynomials for $N$ (of degrees $3,6,3$) are given by 
$$ x^3+y^3+z^3, \qquad 
x^6+y^6+z^6, 
\qquad xyz. $$
Applying the Reynolds operator for $G$ to these gives
$$ R_G(x^3+y^3+z^3)=0, \qquad R_G(xyz)=0,  $$
$$ R_G(x^6+y^6+z^6)={1\over6}\bigl(x^6+y^6+z^6-10(x^3y^3+x^3z^3+y^3z^3)\bigr), $$
and so the $N$-invariant polynomial of degree $6$ can be replaced by 
a $G$-invariant one (on which $G$ and hence $G/N$ acts trivially).
On the span of the remaining $N$-invariant polynomials $G/N$ acts
faithfully as a reflection group 
(i.e., reflections map to reflections). 
Let's consider the action of the inverse of the generators for $G$.
For $r_1$, we have
$$ x^3+y^3+z^3\mapsto (\go x)^3+y^3+z^3=x^3+y^3+z^3, \quad
	xyz\mapsto (\go x)yz = \go (xyz), $$
and for $r_2$, we have
$$ x^3+y^3+z^3\mapsto {i\over\sqrt{3}}\bigl\{ x^3+y^3+z^3+6\overline{\go}xyz\bigr\}, \quad 
xyz\mapsto {i\over3\sqrt{3}}\bigl\{ \overline{\go}(x^3+y^3+z^3)-3\go xyz\bigl\}, $$
with the action of $P_{(12)}$ being trivial.
Hence the faithful action of $G/N$ on 
$$ V=\spam_\CC\{x^3+y^3+z^3,xyz\} $$ 
gives the complex reflection group generated by
$$ [r_1] = \pmat{1&0\cr0&\go}, \qquad
	[r_2] ={i\over\sqrt{3}} \pmat{1&\overline{\go}/3 \cr 6\overline{\go}& -\go}, $$
	which is easily seen to be the imprimitive complex reflection group $G_4$ (order $24$).
\end{example}

\begin{example}
\label{G7quotexample}
The normal reflection subgroups of $G_7$ and their invariant polynomials are as 
follows:
\begin{align*}
G^{(1,1,1)}=1 : &\qquad x,\ y, \cr
G^{(1,1,3)}=G_4 : &\qquad x^4 - 2\sqrt{3}i x^2y^2 + y^4, \ x^5y-xy^5, \cr
G^{(1,3,1)}=G_4 : &\qquad x^4 + 2\sqrt{3}i x^2y^2 + y^4, \ x^5y-xy^5, \cr
G^{(2,1,1)}=G(4,2,2) : &\qquad x^4 + y^4, \ x^2y^2, \cr
G^{(1,3,3)}=G_5 : &\qquad x^5y-xy^5, \ (x^4 \pm 2\sqrt{3}i x^2y^2 + y^4)^3, \cr
G^{(2,1,3)}=G_6 : &\qquad x^4 - 2\sqrt{3}i x^2y^2 + y^4, \ (x^5y-xy^5)^2, \cr
G^{(2,3,1)}=G_6 : &\qquad x^4 + 2\sqrt{3}i x^2y^2 + y^4, \ (x^5y-xy^5)^2, \cr
G^{(2,3,3)}=G_7 : &\qquad (x^4 \pm 2\sqrt{3}i x^2y^2 + y^4)^3, \quad (x^5y-xy^5)^2, 
\end{align*}
These were obtained by applying the Reynolds operator to homogeneous
polynomials (of the appropriate degree) in the invariant polynomials for
	subgroups, starting with the monomials (for the trivial subgroup $1$). We observe the following linear relations
\begin{align*}
(x^4+2\sqrt{3}i x^2y^2+y^4)^3+(x^4-2\sqrt{3}i x^2y^2+y^4)^3 
& =2(x^{12}-33x^8y^4-33x^4y^8+ y^{12}), \cr
(x^4 + 2\sqrt{3}i x^2y^2 + y^4)^3-(x^4 - 2\sqrt{3}i x^2y^2 + y^4)^3 
&=12\sqrt{3}i(x^5y-xy^5)^2,
\end{align*}
and so the first invariant polynomial for $G_7$ given above can be replaced by
one with integer coefficients, if so desired. Calculating the action of the generators 
for $G_7$ on the invariant polynomials given for $G_4$ gives the following
\begin{align*}
G^{(1,1,3)}=G_4 : &\qquad [R^2]=\pmat{1&0\cr0&-1}, \ [Z]=\pmat{\go&0\cr0&1}, \ [Z^R]=\pmat{1&0\cr0&1}, \cr
G^{(1,3,1)}=G_4 : &\qquad [R^2]=\pmat{1&0\cr0&-1}, \ [Z]=\pmat{1&0\cr0&1}, \ [Z^R]=\pmat{\go&0\cr0&1},
\end{align*}
so that 
	$$ {G_7\over G_4} = C_2\times C_3\qquad\hbox{(in both cases)}. $$
\end{example}

\begin{example} 
\label{G(4,2,2)quotsExample}
We consider 
$N=G(4,2,2)\lhd G_6,G_7,G_8,G_9,G_{10},G_{11},G_{13},G_{15}\subset G_{11}$.
A set of basic invariant polynomials for $N=G(4,2,2)$ is
$$ x^4+y^4, \quad x^2y^2. $$
With respect to this basis, the action of the generators for $G_{11}$
	on $\spam_\CC\{x^4+y^4,x^2y^2\}$ is given by
$$ [F]=\pmat{{1\over2}&{1\over4}\cr 3&-{1\over2}}, \quad
[R]=\pmat{1&0\cr 0&-1}, \quad
	[Z]=\pmat{-{1\over2}\overline{\go}&{1\over4}\overline{\go}\cr 
	-3\overline{\go}&-{1\over2}\overline{\go}}.  $$
These matrices are reflections (not all unitary), and it is not
obvious that they generate an imprimitive reflection group.
A basis of eigenvectors for $[Z]$ is given by 
	$$ B:=\pmat{1&-1\cr2\sqrt{3}i&2\sqrt{3}i}, $$
and in this basis we have
$$ [F]_B 
	=\pmat{0&\go\cr\overline{\go}&0}, \quad
[R]_B 
	=\pmat{0&-1\cr-1&0}, \quad
[Z]_B 
=\pmat{1&0\cr0&\go}, $$
where $[g]_B=B^{-1}[g]B$. Thus
$$ {G_{11}\over G(4,2,2)}\cong 
\inpro{\pmat{0&\go\cr\overline{\go}&0},\pmat{0&-1\cr-1&0},\pmat{1&0\cr0&\go}}
=G(6,2,2), $$
with the other quotients giving reflection subgroups of $G(6,2,2)$,
i.e.,  
\begin{align*}
{G_{15}\over G(4,2,2)} &\cong \inpro{[F]_B,[R^2]_B,[Z]_B}
=\inpro{\pmat{0&\go\cr\overline{\go}&0},\pmat{1&0\cr0&1},\pmat{1&0\cr0&\go}}
        \cong G(3,1,2), \cr
{G_{9}\over G(4,2,2)} &\cong \inpro{[F]_B,[R]_B}
=\inpro{\pmat{0&\go\cr\overline{\go}&0},\pmat{0&-1\cr-1&0}}
        \cong G(6,6,2), \cr
{G_{10}\over G(4,2,2)} &\cong \inpro{[R]_B,[Z]_B}
=\inpro{\pmat{0&-1\cr-1&0},\pmat{1&0\cr0&\go}}
        \cong G(3,1,2), \cr
{G_{13}\over G(4,2,2)} &\cong \inpro{[F]_B,[F^Z]_B,[R^2]_B}
=\inpro{\pmat{0&\go\cr\overline{\go}&0},\pmat{0&1\cr1&0},\pmat{1&0\cr0&1}}
\cong G(3,3,2)\cong  G(1,1,3), \cr
{G_{8}\over G(4,2,2)} & \cong \inpro{[R]_B,[R^F]_B}
       =\inpro{\pmat{0&-1\cr-1&0},\pmat{0&-\overline{\go}\cr-\go&0}}
        \cong G(3,3,2)\cong G(1,1,3), \cr
	{G_{7}\over G(4,2,2)} &\cong \inpro{[R^2]_B,[Z]_B,[Z^R]_B}
=\inpro{\pmat{1&0\cr0&1},\pmat{1&0\cr0&\go},\pmat{\go&0\cr0&1}}
\cong C_3\times C_3, \cr
	{G_{6}\over G(4,2,2)} & \cong 
\inpro{[R^2]_B,[Z]_B} =\inpro{\pmat{1&0\cr0&\go}},
\inpro{[R^2]_B,[Z^R]_B} =\inpro{\pmat{\go&0\cr0&1}}
        \cong C_3.
\end{align*}
\end{example}

From the reflection orbits of Lemma \ref{G(m,p,d)-orbits}, we obtain the following.

\begin{example} 
\label{G(m,p,2)quotExample}
For $m$ odd, $G(m,m,2)=G^{(2)}$ has no normal reflection subgroups,
and for $m$ even, we have $G(m,m,2)=G^{(2,2)}$, 
$G^{(1,2)}\cong G^{(2,1)}\cong G({m\over2},{m\over2},2)$, and 
\begin{equation}
\label{G(2,2)}
{G^{(2,2)}\over G^{(1,2)}}, {G^{(2,2)}\over G^{(2,1)}} 
= {G(m,m,2)\over G({m\over2},{m\over2},2)} \cong C_2, \qquad
\hbox{($m$ even)}.
\end{equation}
Now suppose that $p\divides m$, $p\ne m$, 
and ${m\over q}\divides {m\over p}$, i.e., $p\divides q$ and $q\divides m$.

For $p$ odd (two orbits), we have $G(m,p,2)=G^{(2,{m\over p})}$, 
and
\begin{align}
{G^{(2,{m\over p})}\over G^{(1,{m\over q})}}
&=  {G(m,p,2)\over C_{m\over q}\times C_{m\over q}} \cong G(q,p,2),
\qquad \hbox{($p$ odd)}, \label{G(1,m/q)} \\
{G^{(2,{m\over p})}\over G^{(2,{m\over q})}}
&=  {G(m,p,2)\over G(m,q,2)} \cong C_{q\over p},\qquad \hbox{($p$ odd)}.
\label{G(2,m/q)}
\end{align}
For $p$ even (three orbits), we have $G(m,p,2)=G^{(2,2,{m\over p})}$, and
\begin{align}
{G^{(2,2,{m\over p})}\over G^{(1,1,{m\over q})}} 
&=  {G(m,p,2)\over C_{m\over q}\times C_{m\over q}} \cong G(q,p,2),
\qquad \hbox{($p$ even)}, \label{G(1,1,m/q)} \\
{G^{(2,2,{m\over p})}\over G^{(2,2,{m\over q})}}
&=  {G(m,p,2)\over G(m,q,2)} \cong C_{q\over p},  \qquad \hbox{($p$ even)},
	\label{G(2,2,m/q)} \\
{G^{(2,2,{m\over p})}\over G^{(1,2,{m\over q})}},
{G^{(2,2,{m\over p})}\over G^{(2,1,{m\over q})}} 
	&= {G(m,p,2)\over G({m\over2},{q\over2},2) }
\cong  C_2\times C_{q\over p}, \qquad \hbox{($p$ even)}.
\label{G(1,2,m/q)}
\end{align}
By collecting cases with different orbit structures, 
one obtains  
\begin{equation}
\label{quotswithoutorbits}
{G(m,p,2)\over C_{m\over q}\times C_{m\over q}} \cong G(q,p,2), \qquad
{G(m,p,2)\over G(m,q,2)} \cong C_{q\over p}, \qquad
{G(m,p,2)\over G({m\over2},{q\over2},2) }
	\cong C_2\times C_{q\over p} \quad\hbox{($p$ even)}.
\end{equation}
which are (2.a), (2.b), (2.c) of Theorem 6.1 of \cite{AW23}.
\end{example}

\begin{example}
\label{G(m,p,d)normalsgs}
For $d\ge3$, $G(m,m,d)=G^{(2)}$ has no normal reflection subgroups,
	and for $G(m,p,d)=G^{(2,{m\over p})}$, $p\ne m$, we have
\begin{equation}
\label{G(1,m/q)dge3}
{G^{(2,{m\over p})}\over G^{(1,{m\over q})}}
	=  {G(m,p,d)\over (C_{m\over q})^d} \cong G(q,p,d),
\qquad d\ge3,
\end{equation}
\begin{equation}
\label{G(2,m/q)dge3}
{G^{(2,{m\over p})}\over G^{(2,{m\over q})}}
=  {G(m,p,d)\over G(m,q,d)} \cong C_{q\over p},
\qquad d\ge3,
\end{equation}
where $p\divides q$ and $q\divides m$.
\end{example}

\begin{example}
The complex reflection group $G=G_{28}$ is generated by
the reflections 
$$ \pmat{1&0&0&0\cr 0&0&1&0\cr 0&1&0&0\cr 0&0&0&1}, \quad
 \pmat{1&0&0&0\cr 0&1&0&0\cr 0&0&0&1\cr 0&0&1&0},  \quad
 \pmat{1&0&0&0\cr 0&1&0&0\cr 0&0&1&0\cr 0&0&0&-1}, \quad
	{1\over2}\pmat{1&1&1&1\cr 1&1&-1&-1\cr 1&-1&1&-1\cr1&-1&-1&1}. $$
Invariant polynomials for $N=G(2,2,4)$ are
$$ x_1^2+x_2^2+x_3^2+x_4^2, \quad x_1^4+x_2^4+x_3^4+x_4^4, \quad
x_1^6+x_2^6+x_3^6+x_4^6, \quad x_1x_2x_3x_4, $$
with the first also being $G$-invariant, and 
for the third $R_G(x_1^6+x_2^6+x_3^6+x_4^6)\ne0$. 
Thus, we consider the action of $G$ on 
$x_1^4+x_2^4+x_3^4+x_4^4$ and $x_1x_2x_3x_4$.
These polynomials are fixed by the permutation matrices, 
and the action of the diagonal reflection is
$$ x_1^4+x_2^4+x_3^4+x_4^4\mapsto x_1^4+x_2^4+x_3^4+x_4^4, 
	\quad x_1x_2x_3x_4\mapsto -x_1x_2x_3x_4. $$
The final (non-monomial) generating reflection maps as follows
\begin{align*}
x_1^4+x_2^4+x_3^4+x_4^4 &\mapsto -{1\over2}(x_1^4+x_2^4+x_3^4+x_4^4)+6(x_1x_2x_3x_4)
        +{3\over 4} (x_1^2+x_2^2+x_3^2)^2, \cr
	x_1x_2x_3x_4 &\mapsto   {1\over8}(x_1^4+x_2^4+x_3^4+x_4^4)
        +{1\over2} (x_1x_2x_3x_4)-{1\over16}(x_1^2+x_2^2+x_3^2+x_4^2)^2,
\end{align*}
so that, in this case, $G$ does not map the span of the given basic invariant polynomials
for $N$ with
$R_G(f)=0$ to itself, and so instead we must take the image of this space
under the action of $G$. In this way, the action of 
the last two reflections on the span of 
$(x_1^2+x_2^2+x_3^2+x_4^2)^2$, $x_1^4+x_2^4+x_3^4+x_4^4$ and $x_1x_2x_3x_4$ 
is given by the matrices
	$$ \pmat{1&0&0\cr 0&1&0\cr 0&0&-1}, \quad
\pmat{1&{3\over4}&-{1\over16}\cr 0&-{1\over2}&{1\over8}\cr 0&6&{1\over2}}, $$
	which generate a rank two reflection group, 
	namely $G/N\cong G(3,3,2)\cong G(1,1,3)\cong S_3$.
\end{example}

Let $G=G^{(k_1,\ldots,k_m)}$ be a complex reflection group,
$N=G^\ga=G^{(\ga_1,\ldots,\ga_m)}$ be a normal 
reflection subgroup, and $r_{a_j}\in R_{a_j}^G$ be reflection of 
(maximal) order $k_j$. Then
\begin{itemize}
\item $r_{a_j} N$ is a reflection in $G/N$.
\item The order of $r_{a_j} N$ is ${k_j/\ga_j}$.
\end{itemize}
The latter follows from the fact $r_{a_j}^m\in R_{a_j}^G$ and
$R_{a_j}^G\cap R_{a_j}^N=\inpro{r_{a_j}^{k_j/\ga_j}}$, which gives
$$ (r_{a_j} N)^m=r_{a_j}^m N=N 
\Iff r_{a_j}^m \subset \inpro{r_{a_j}^{k_j/\ga_j}}
\Iff (k_j/\ga_j)\divides m. $$
Moreover, $G\to G/N$ maps reflection orbits to reflection orbits,
and so Lemma \ref{abelianisationlemma} gives
\begin{equation}
\label{abelianisationG/N}
\bigl({G\over N}\bigr)_{\rm ab}=
\bigl({G^{(k_1,\ldots,k_m)}\over G^{(\ga_1,\ldots,\ga_m)}}\bigr)_{\rm ab}
\cong C_{k_1\over\ga_1}\times \cdots\times C_{k_m\over\ga_m}.
\end{equation}

From Example \ref{G(4,2,2)quotsExample}, we collect the nonabelian
quotients by 
$G(4,2,2)$.
\begin{align}
\label{G(4,2,2)quots}
& {G_{11}\over G(4,2,2)} \cong G(6,2,2), \quad
{G_{15}\over G(4,2,2)} \cong G(3,1,2), \quad
{G_9\over G(4,2,2)} \cong G(6,6,2), \cr
& {G_{10}\over G(4,2,2)} \cong G(3,1,2), \quad
{G_{13}\over G(4,2,2)} \cong G(1,1,3), \quad
{G_8\over G(4,2,2)} \cong G(1,1,3).
\end{align}

We can now describe all quotients by normal reflection subgroups.

\begin{theorem}
\label{QuotientsTheorem}
Let $G=G^{(k_1,\ldots,k_m)}$ be a complex reflection group
and $N=G^{(\ga_1,\ldots,\ga_m)}\ne1$ be a 
normal reflection 
subgroup, then the reflection group $G/N$ is abelian, 
	i.e., $[G,G]\lhd N$,
	given by
\begin{equation}
\label{G/Nequation}
{G\over N}={G^{(k_1,\ldots,k_m)}\over G^{(\ga_1,\ldots,\ga_m)}}
\cong C_{k_1\over\ga_1}\times \cdots\times C_{k_m\over\ga_m},
\end{equation}
except for the following cases
\begin{enumerate}[\rm (a)]
\item $\displaystyle {G(m,p,d)\over(C_{m\over q})^d} \cong G(m,q,d)$, 
		$d\ge2$,
where $p\divides q$ and $q\divides m$.
%

\item $G=G_8,G_9,G_{10},G_{11},G_{13},G_{15}$ and $N=G(4,2,2)$, with
$G/N$ given by (\ref{G(4,2,2)quots}).

\item $\displaystyle {G_{26}\over G(3,3,3)} \cong G_4$ \ (the only case where
	$G/N$ is primitive).

\item $\displaystyle {G_{28}\over G(2,2,4)} \cong G(3,3,2)\cong G(1,1,3)\cong S_3$.
\end{enumerate}
We observe that the normal subgroups $G(4,2,2)$, $G(3,3,3)$, $G(2,2,4)$ 
above are precisely the imprimitive complex reflection groups with 
more than one system imprimitivity.

Moreover, if $N$ is a collineation preserving subgroup of $G$ then
	$G/N$ is cyclic,
which is equivalent to 
\begin{equation}
\label{quotcyclicCondition}
{|G^{(k_1,\ldots,k_m)}|\over|G^{(\ga_1,\dots,\ga_m)}|}
=\lcm({k_1\over\ga_1},\ldots,{k_m\over\ga_m}). 
\end{equation}
\end{theorem}

\begin{proof} By the definition of the abelianisation and (\ref{abelianisationG/N}),
we have that
	$$ {G\over N}= {G^{(k_1,\ldots,k_m)}\over G^{(\ga_1,\ldots,\ga_m)}}\to
\bigl({G^{(k_1,\ldots,k_m)}\over G^{(\ga_1,\ldots,\ga_m)}}\bigr)_{\rm ab} 
\cong C_{k_1\over\ga_1}\times\cdots\times C_{k_m\over\ga_m} $$
is a surjective homomorphism, which is injective, i.e., $G/N$ is abelian
	and given by (\ref{G/Nequation}), if
and only if the group order is preserved, i.e.,
\begin{equation}
\label{G/Nabeliancdn}
{G^{(k_1,\ldots,k_m)}\over G^{(\ga_1,\dots,\ga_m)}} \hbox{ is abelian} \Iff
{|G^{(k_1,\ldots,k_m)}|\over|G^{(\ga_1,\dots,\ga_m)}|}
={k_1 k_2\cdots k_m\over\ga_1\ga_2\cdots\ga_m}.
\end{equation}
It is easy to verify when (\ref{G/Nabeliancdn}) holds from the reflection 
group orders and reflection types (see Example \ref{quotientineqexample}), 
which leads to the 
exceptions (a), (b), (c), (d). The nonabelian quotients are calculated in 
Examples \ref{G26quotexample}, \ref{G7quotexample},
\ref{G(4,2,2)quotsExample}, \ref{G(m,p,2)quotExample}, 
\ref{G(m,p,d)normalsgs}. 

If $N$ is collineation preserving, then $N$ is irreducible, 
so that $G=\inpro{\zeta}N$, where $\zeta$ is a primitive $|Z(G)|/|Z(N)|=|G|/|N|$
root of unity, and $G/N\cong\inpro{\zeta}$ is cyclic, so we must have
\begin{equation}
\label{G/Ncyclic-cdn}
C_{k_1\over\ga_1}\times \cdots\times C_{k_m\over\ga_m} \hbox{ is cyclic}
\Iff {k_1 k_2\cdots k_m\over\ga_1\ga_2\cdots\ga_m}=\lcm({k_1\over\ga_1},\ldots,{k_m\over\ga_m}).
\end{equation}
Further, since $G/N$ is abelian, (\ref{G/Nabeliancdn}) gives
$$ {G^{(k_1,\ldots,k_m)}\over G^{(\ga_1,\dots,\ga_m)}} \hbox{ is abelian} \Iff
{|G^{(k_1,\ldots,k_m)}|\over|G^{(\ga_1,\dots,\ga_m)}|}
={k_1 k_2\cdots k_m\over\ga_1\ga_2\cdots\ga_m}=\lcm({k_1\over\ga_1},\ldots,{k_m\over\ga_m}), 
$$ 
which is (\ref{quotcyclicCondition}).
Conversely, if (\ref{quotcyclicCondition}) holds, then we must have equality in 
$${|G^{(k_1,\ldots,k_m)}|\over|G^{(\ga_1,\dots,\ga_m)}|}
\ge {k_1 k_2\cdots k_m\over\ga_1\ga_2\cdots\ga_m}=\lcm({k_1\over\ga_1},\ldots,{k_m\over\ga_m}), $$
so that (\ref{G/Nabeliancdn}) gives that $G/N$ is abelian, 
and (\ref{G/Ncyclic-cdn}) that it is cyclic.
\end{proof}

From the proof of Theorem \ref{QuotientsTheorem}, we have the inequalities
\begin{equation}
\label{quotientinequalities}
{|G^{(k_1,\ldots,k_m)}|\over |G^{(\ga_1,\ldots,\ga_m)}|}
\ge {k_1 k_2\cdots k_m\over\ga_1\ga_2\cdots\ga_m}
\ge \lcm({k_1\over\ga_1},\ldots,{k_m\over\ga_m}),
\end{equation}
where there is equality in the first if and only if $G/N$ is abelian,
and equality in the second if and only if $G/N$ is cyclic.

\begin{example}
\label{quotientineqexample}
Here we illustrate the cases of equality in (\ref{quotientinequalities}) for
$G^{(2,2,{m\over p})}=G(m,p,2)$, $p$ even (three orbits),
of Example \ref{G(m,p,2)quotExample}. From the group orders, we have
\begin{align*}
{|G^{(2,2,{m\over p})}|\over|G^{(1,1,{m\over q})}|}
={2m(m/p)\over(m/q)^2}=2q{q\over p}>{2\cdot2\cdot (m/p)\over1\cdot1\cdot (m/q)}
	=4{q\over p},\ q\ne 2
& \Implies \hbox{$G/N$ is not abelian}, \cr
{|G^{(2,2,{m\over p})}|\over|G^{(2,2,{m\over q})}|}
={2m(m/p)\over2m(m/q)}={q\over p}=\lcm\bigl({2\over2},{2\over2},{m/p\over m/q}\bigr)
	&\Implies \hbox{$G/N$ is cyclic}, \cr
{|G^{(2,2,{m\over p})}|\over|G^{(1,2,{m\over q})}|}
={2m(m/p)\over2(m/2){m/2\over q/2}}=2{q\over p}={2\cdot2\cdot (m/p)\over1\cdot2\cdot (m/q)}
&\Implies \hbox{$G/N$ is abelian}.
\end{align*}
For $q=2$ in the first case above, 
	there is equality in the strict inequality,
and we obtain an 
	abelian quotient, i.e.,
$$ {G^{(2,2,{m\over 2})}\over G^{(1,1,{m\over 2})}}
	={G(m,2,2)\over C_{m\over2}\times C_{m\over2}} = C_2\times C_2. $$
For the last case above, $G/N=C_2\times C_{q\over p}$, and we can check
	(\ref{quotcyclicCondition}) using 
$$ \lcm\bigl({2\over1},{2\over2},{m/p\over m/q}\bigr)=
\begin{cases}
  2{q\over p}, & \hbox{$q\over p$ odd}, \cr
  {q\over p}, & \hbox{$q\over p$ even}; 
\end{cases}
$$
to conclude that $G/N$ is cyclic when ${q\over p}$ is odd.
\end{example}

It is natural to consider to what extent (\ref{quotcyclicCondition}) determines
the collineation preserving normal subgroups. 

\begin{example} 
\label{G/NcyclicNotcollpreserving}
If $G/N$ is cyclic, i.e., (\ref{quotcyclicCondition}) holds,
then $N$ is collineation preserving, except for the cases
$$ {G_{10}\over G_5} \cong C_4, \quad 
{G_{10}\over G_7} \cong C_2, \quad 
{G_{14}\over G_5} \cong C_2, \quad 
{G_{15}\over G_7} \cong C_2, $$
	$$ {G(m,p,2)\over G({m\over2},{q\over2},2)}\cong C_{2{q\over p}} \quad
	\hbox{(${p\over2}$ even, ${q\over p}$ odd)},  \quad
 {G(m,p,d)\over G(m,q,d)}\cong C_{q\over p} , \quad
	\gcd(p,d)\ne\gcd(q,d). $$
This follows by direct calculation of the size of the collineation groups
of $G$ and $N$.
\end{example}

Heuristically, 
the normal reflection subgroups $N=G^{(\ga_1,\ldots,g_m)}\ne 1$ of $G=^{(k_1,\ldots,k_m)}$
have a hierarchy of ``size'' given by 
\begin{equation}
\label{heuristicsize}
\hbox{primitive,\quad imprimitive,\quad reducible\ (abelian) \qquad (large to small), }
\end{equation}
with the quotients $G/N$ being of opposite size. For example, we have
\begin{itemize}
\item If $N$ is primitive, 
then  $G/N\cong C_{k_1\over\ga_1}\times\cdots\times C_{k_m\over\ga_m}$ is reducible.
\item If $N$ is imprimitive, then $G/N$ is imprimitive or reducible, except for the one case
	$\displaystyle {G_{26}\over G(3,3,3)} \cong G_4$.
\item If $N$ is reducible, then $G$ and $G/N$ are imprimitive, i.e.,
$\displaystyle {G(m,p,d)\over(C_{m\over q})^d} \cong G(m,q,d)$. 
\end{itemize}
In the same vein, we also have
\begin{itemize}
\item If $N$ is a collineation preserving subgroup of $G$, then
$G/N$ is reducible (cyclic).
\end{itemize}
We now consider when a normal reflection subgroup appears multiple times.

For a complex reflection group $G=G^{(k_1,\ldots,k_m)}$ of reflection type
$n_1C_{k_1},\ldots,n_mC_{k_m}$, we consider the
subgroup of permutations in $S_m$ which preserve the reflection type,
i.e., the stabiliser of $(k_1,\ldots,k_m)$ and $(n_1,\ldots,n_m)$.
This group $\Stab_{S_m}(k,n)$ acts on the normal reflection 
subgroups $N=G^{(\ga_1,\ldots,\ga_m)}$, with the orbits being natural 
candidates for isomorphic normal subgroups.
We now consider the cases when this group is nontrivial.

\begin{example}
\label{multinormalreftype}
We consider those reflection groups $G=G^{(k_1,\ldots,k_m)}$ for which $\Stab_{S_m}(k,n)$ 
is nontrivial, and determine the corresponding orbits containing more than one normal 
reflection subgroup.  For the primitive complex reflection groups, 
there are three cases
\begin{align*}
G_5=G^{(3,3)}: \qquad & G^{(1,3)},G^{(3,1)}\cong G_4\lhd G_5, \cr
G_7=G^{(2,3,3)}: \qquad & G^{(2,1,3)},G^{(2,3,1)}\cong G_6\lhd G_7, \cr
 & G^{(1,1,3)},G^{(1,3,1)}\cong G_4\lhd G_7, \cr
	G_{28}=G^{(2,2)}: \qquad & G^{(1,2)},G^{(2,1)}\cong G(2,2,4)\lhd G_{28}.
\end{align*} 
As an example, for $G_7=G^{(2,3,3)}$ above, we have 
$\Stab_{S_3}((2,3,3),(6,4,4)) =\{(\,),(2\, 3)\}$.

For the imprimitive complex reflection groups $G(m,p,d)$,
by Example \ref{G(m,p,2)quotExample},
there are three reflection orbits of the same type if and only if 
${m\over2}=2$, ${m\over p}=2$, which gives
$$ G(4,2,2)=G^{(2,2,2)}: \qquad  G^{(1,2,2)},G^{(2,1,2)},G^{(2,2,1)}\cong 
G(4,4,2)\cong G(2,1,2)\lhd G(4,2,2), $$
$$ G(4,2,2)=G^{(2,2,2)}: \qquad  G^{(1,1,2)},G^{(1,2,1)},G^{(2,1,1)}\cong 
G(2,2,2)\cong C_2\times C_2\lhd G(4,2,2), $$
and there are two when
$$ G(m,p,2)=G^{(2,2,{m\over p})}, \ \hbox{$p$ even}:\qquad
G^{(1,2,{m\over q})},G^{(2,1,{m\over q})}
\cong G({m\over2},{q\over2},2)\lhd G(m,p,2), $$
where $p\divides q$ and  $q\divides m$. There are no other cases.
\end{example}




We now show that Example \ref{multinormalreftype} gives all the cases
where there is a repeated normal reflection subgroup.

\begin{theorem}
The normal reflection subgroups $G^{\ga}=G^{(\ga_1,\ldots,\ga_m)}$
and $G^{\gb}=G^{(\gb_1,\ldots,\gb_m)}$ of the complex reflection 
	group $G=G^{(k_1,\ldots,k_m)}$ are isomorphic if and only if
	$\gb=\gs \ga$ for some permutation $\gs\in S_m$ which fixes $(k_1,\ldots,k_m)$.

The 
	cases where a reflection group appears multiple
times as a normal subgroup are
\begin{enumerate}[\rm (i)]
\item \ $G_4\lhd G_5$, \quad $G_6\lhd G_7$, \quad $G_4\lhd G_7$, \quad
        $G(2,2,4)\lhd G_{28}$\quad (each appears twice).
\item \ $G(4,4,2)\cong G(2,1,2)\lhd G(4,2,2)$, \quad
$G(2,2,2)\cong C_2\times C_2\lhd G(4,2,2)$ \ (three times).
\item \ $G({m\over2},{q\over2},2)\lhd G(m,p,2)\ne G(4,2,2)$, \quad
$p$ even, $p\divides q$, $q\divides m$ \quad (each appears twice).
\end{enumerate}
\end{theorem}

\begin{proof} For normal reflection subgroups to be isomorphic,
	equivalently conjugate in the general linear group (having the
	same character), they 
must have the same reflection type, i.e.,  
	one must have $\gb=\gs \ga=(\ga_{\gs j})$, for some $\gs\in S_m$ where
$$ n_j C_{\ga_j}= n_{\gs j} C_{\ga_{\gs j}}
\Implies n_{\gs j}=n_j, \quad \ga_j=\ga_{\gs j}. $$

There are just two cases where some $k_j$ is composite: $G_{11}$ and 
its normal subgroups $G_8$, $G_9$, $G_{10}$, and $G(m,p,d)$ for ${m\over p}$ 
composite. In each of these cases, the orbit length $n_j$ appears only once, and
so $n_{\gs j}=n_j$ gives $\gs j=j$ and hence $k_{\gs j}=k_j$.

If $k_j$ is prime, then we can have either
$$ \ga_j=k_j \Implies k_j=\ga_j=\ga_{\gs j}=k_{\gs j}, $$
or $\ga_j=1$, in which case there is no restriction on $\gs j$ other than
$n_{\gs j}=n_j$, which implies that $k_j=k_{\gs j}$ (by inspection).
	Hence $\gs\in\Stab_{S_m}(k,n)$, and 
therefore all the repeated normal subgroups are given by Example \ref{multinormalreftype}.
\end{proof}

This gives the Theorem 6.2 of \cite{AW23}, with (1),(2) corresponding
to (ii), (3) to (iii) and (4),(5) to (i), except for the following cases
which are missing
$$ G_6\lhd G_7, \quad G_4\lhd G_7, \quad \hbox{(each appears twice)}. $$

One might hope that the lattice of reflection groups obtained from that of the normal reflection
subgroups of $G$ by replacing $N$ by $G/N$
(see Figure \ref{G7-lattice-quotients}) might give an ``upside down'' lattice of reflection subgroups of $G$,
as the orders of $|G/N|$ would allow.
However, this appears not to be the case, 
as  $G_7$ has no reflection subgroups (normal or not) of
 order $6$, and so the quotient $C_2\times C_3=G_7/G_4$ cannot be on such a sublattice.
There is an inversion symmetry of the lattice of normal subgroups given by the correspondence 
\begin{equation}
\label{latticesymmetry}
G^\ga = G^{(\ga_1,\ldots,\ga_m)}
\longleftrightarrow G^{k/\ga}=G^{({k_1\over\ga_1},\ldots,{k_m\over\ga_m})}.
\end{equation}
This sets up a correspondence (duality) between normal reflection subgroups, e.g.,
\begin{align*}
G_7: & \qquad G_5\longleftrightarrow G(4,2,2), \quad
G_4\longleftrightarrow G_6, \cr
G_{11}: & \qquad
G_{15}\longleftrightarrow G(4,2,2), \quad
G_{9}\longleftrightarrow G_5, \quad
G_{10}\longleftrightarrow G_{12}, \quad
G_{14}\longleftrightarrow G_{8}, \quad
G_{13}\longleftrightarrow G_{7}, \cr
G_{19}: & \qquad
G_{21}\longleftrightarrow G_{16}, \quad
G_{9}\longleftrightarrow G_{5}, \quad
G_{13}\longleftrightarrow G_{7}.
\end{align*}
This duality flips the group's location in the ordering (\ref{heuristicsize}), e.g.,
for the pairs above $|G^\ga|\cdot|G^{k/\ga}|$ is constant, except for the 
pair $G_{15}\longleftrightarrow G(4,2,2)$, which is the only case where one of the
groups is not the intersection of those above it in the lattice,
namely $G(4,2,2)$.
Similarly, by (\ref{abelianisationcdn}), $|(G^\ga)_{\rm ab}|\cdot|(G^{k/\ga})_{\rm ab}|$ is constant, equal to $|G_{\rm ab}|$,
when the groups are collineation preserving (have no split orbits).


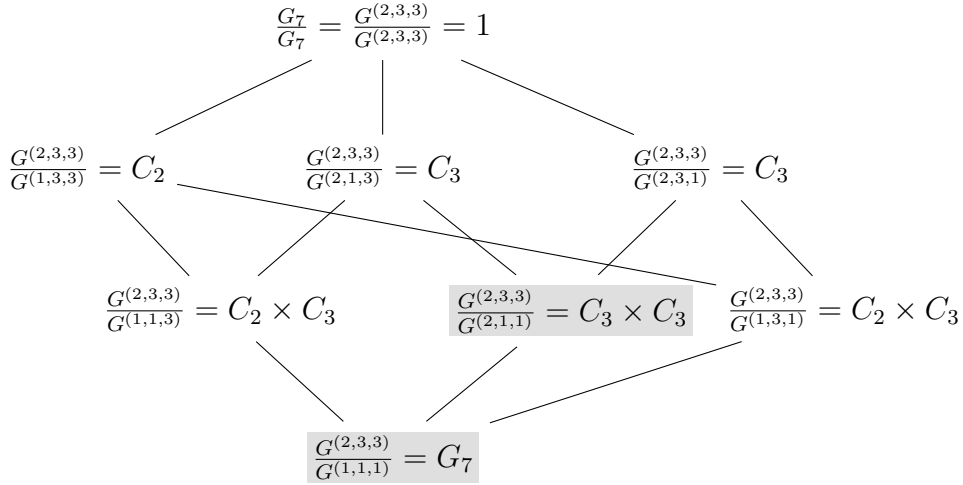
\begin{figure}[H]
\begin{center}
\begin{tikzpicture}
    \matrix (A) [matrix of nodes, row sep=1.0cm, column sep = -1.1 cm]
    {
 &&& ${G_7\over G_7}={G^{(2,3,3)}\over G^{(2,3,3)}}=1$ &&& \\
	& ${G^{(2,3,3)}\over G^{(1,3,3)}}= C_2$ && ${G^{(2,3,3)}\over G^{(2,1,3)}}=C_3$ && ${G^{(2,3,3)}\over G^{(2,3,1)}}=C_3$ & \\
	&& ${G^{(2,3,3)}\over G^{(1,1,3)}}=C_2\times C_3$ && \greybox{${G^{(2,3,3)}\over G^{(2,1,1)}}=C_3\times C_3$} && ${G^{(2,3,3)}\over G^{(1,3,1)}}=C_2\times C_3$ \\
	& && \greybox{${G^{(2,3,3)}\over G^{(1,1,1)}}=G_7$} &&& \\
    };
    \draw (A-1-4)--(A-2-2);
    \draw (A-1-4)--(A-2-4);
    \draw (A-1-4)--(A-2-6);

    \draw (A-2-2)--(A-3-3);
    \draw (A-2-4)--(A-3-3);
    \draw (A-2-4)--(A-3-5);
    \draw (A-2-6)--(A-3-5);
    \draw (A-2-6)--(A-3-7);
    \draw (A-2-2)--(A-3-7);

    \draw (A-3-3)--(A-4-4);
    \draw (A-3-5)--(A-4-4);
    \draw (A-3-7)--(A-4-4);

\end{tikzpicture}
\end{center}
\caption{\label{G7-lattice-quotients} 
The quotient reflection groups $G/N$ for the normal reflection subgroups 
$N=G^\ga$ of $G=G_{7}=G^{(2,3,3)}$, 
arranged as in the lattice of 
	Figure \ref{G7-lattice}. These all happen to be abelian, 
	and hence equal to the abelianisation of $G/N$, for $N\ne1$.}
\end{figure}

\section{Quaternionic reflection groups}

Here we consider some examples of the normal reflection subgroups of the quaternionic 
reflection groups. The main observations are
\begin{itemize}
\item The labels (\ref{Quaternionsubgrouplabels}) giving the normal reflection subgroups
may not be unique (see Examples \ref{Pgroups-example} and \ref{typeOexample}),
		 as they are for complex reflection groups  (Theorem \ref{Galatticethm}).
Indeed, the reflections on one orbit may generate reflections on a different orbit.
\item The stabiliser of a reflection subgroup $R_a$ may not pointwise 
	stabilise it, as is the case for complex reflection groups 
(Theorem \ref{RaStabilising}).
\item The abelianisation (which can be trivial) gives limited insight 
into the reflection subgroups for the roots, especially for primitive
reflection groups
(see Table \ref{PrimitiveQuatOrbits}).
\end{itemize}
We have not calculated all the possible multiple labels, and hence all the 
normal reflection subgroups, though this has been done for the
primitive quaternionic reflection groups.


\begin{example}
\label{Pgroups-example}
We consider the primitive quaternionic reflection groups of type $P$ in \cite{C80},
which are given by (see \cite{BW26}) 
\begin{align*}	
P_1 & := \inpro{\pmat{i&0\cr0&1}, {1\over2} \pmat{1+j&-1+j\cr-1+j&1+j}}, 
	& |P_1|=320,\  \\
P_2 & := \inpro{ \pmat{i&0\cr0&1}, \pmat{j&0\cr0&1}, {1\over2} \pmat{1+j&-1+j\cr-1+j&1+j}}, 
	& |P_2|=1920, \\
P_3 & := \inpro{ \pmat{i&0\cr0&1},\pmat{j&0\cr0&1},{1\over\sqrt{2}}\pmat{1&1\cr1&-1}}, 
	& |P_3|=3840.
\end{align*}
with $P_1\subset P_2\subset P_3$
The group $P_3$ has $50$ reflections of order two and $60$ of order four. 
The reflections of order four are all conjugate (and also conjugate in $P_2$),
and correspond to $10$ root lines, giving a $10Q_8$ orbit
(containing $10$ reflections of order two), where 
$$ Q_8=\{1,i,j,k,-1,-i,-j,-k\}, $$
The remaining $40$ reflections of order two form a single $40C_2$ orbit. 
The quaternion group $Q_8$ is Hamiltonian, 
i.e., is nonabelian with all of its six subgroups being normal. Hence, there are
$12$ labels of the form (\ref{Quaternionsubgrouplabels})
giving normal subgroups of $G=G^{(C_2,Q_8)}=P_3$. 
The subgroups corresponding to these labels are
\begin{align*}
P_3 &=G^{(C_2,Q_8)}=G^{(C_2,\inpro{i})}=G^{(C_2,\inpro{j})}=G^{(C_2,\inpro{k})}
=G^{(C_2,\inpro{-1})} =G^{(C_2,1)}\quad\hbox{($6$ labels)}, \cr
P_2 &=G^{(1,Q_8)}=G^{(1,\inpro{i})}=G^{(1,\inpro{j})}=G^{(1,\inpro{k})}
\quad\hbox{($4$ labels)}, \cr
K &=G^{(1,\inpro{-1})} \quad\hbox{($1$ label)}, \cr
1 &=G^{(1,1)} \quad\hbox{($1$ label)},
\end{align*}
where $K\subset P_1$ is the imprimitive quaternionic reflection group
$$ K:=\inpro{\pmat{-1&0\cr0&1}, \pmat{0&1\cr1&0},\pmat{0&i\cr-i&0}, \pmat{0&j\cr-j&0}
}, \qquad |K|=32. $$
Thus, $P_3$ has four normal reflection subgroups.
	
Since $P_2$ has a single $10Q_8$ orbit, it has, similarly, 
three normal reflection subgroups, i.e., $1$, $K$, $P_2$ ($4$ labels).
Since $G=G^{(C_4)}=P_1$ has a single $10C_4$ orbit, it has 
three normal reflection subgroups $P_1$, $G^{(C_2)}=K$, $G^{(1)}=1$,
with unique labels.

The reflection type of the imprimitive reflection group $K$ is $2C_2,2C_2,2C_2,2C_2,2C_2$, and 
so there are $32=2^5$ labels for $K=K^{(2,2,2,2,2)}$ 
(here we adopt the notation (\ref{Quaternionsubgrouplabels})
used in the complex case, since the reflection type
consists only of cyclic groups). Let $n_\ga$ be the number of $2$'s in the 
index $\ga$ for a normal reflection subgroup $K^\ga$ of $K$.
The group $K$ has six labels, for $n_\ga=5,4$, i.e.,
$$ K=K^{(2,2,2,2,2)}
=K^{(1,2,2,2,2)}
=K^{(2,1,2,2,2)}
=K^{(2,2,1,2,2)}
=K^{(2,2,2,1,2)}
=K^{(2,2,2,2,1)} \quad \hbox{($6$ labels)}, $$
and the remaining $26$ labels give distinct normal subgroups, namely
\begin{align*}
& K^\ga, \quad n_\ga=3, \quad\hbox{($10$ isomorphic groups)}, 
	& K^\ga, \quad n_\ga=2, \quad\hbox{($10$ isomorphic groups)},  \cr
& K^\ga, \quad n_\ga=1, \quad\hbox{($5$ isomorphic groups)}, 
& K^{(1,1,1,1,1)}=1 \quad  \hbox{($n_\ga=0$)}. \hskip2.7truecm
\end{align*}
Thus $K$ has $27$ 
normal reflection subgroups. We note that $K$ has $68$ normal subgroups
and five subgroups which are not normal (the reflection subgroups
$R_a$ for its roots). 
\end{example}
	
We observe that 
\begin{itemize}
\item There need not be a unique label (\ref{Quaternionsubgrouplabels})
for quaternionic reflection groups, with the number of labels being an
upper bound for the number of normal reflection subgroups.
Indeed, there is a lattice homomorphism from the lattice of labels onto the
		lattice of normal reflection subgroups.
\item In particular, some quaternionic reflection groups can be generated by a set 
of reflections which does not contain at least one reflection from each orbit, 
i.e., Lemma \ref{abelianisationlemma} does not hold for quaternionic reflection groups.
\end{itemize}

The analogue of the abelianisation of Lemma \ref{abelianisationlemma} 
for a quaternionic reflection group $G$ with reflection type
$n_1 H_{a_1},n_2 H_{a_2},\ldots, n_m H_{a_m}$ 
would be the existence of a surjective homomorphism
\begin{equation}
\label{quatabelianisationform}
G\to H_{a_1}\times H_{a_2}\times \ldots\times H_{a_m},
\end{equation}
which implies that a generating set of reflections for $G$ must 
contain a reflection from each reflection orbit.
By considering the normal subgroups of $G$, 
it is easy see that this is not possible for the groups of Example \ref{Pgroups-example}.

\begin{example} 
For the quaternionic reflection groups $G$ in Example \ref{Pgroups-example},
the groups $H_{a_1}\times \ldots\times H_{a_m}$ of 
(\ref{quatabelianisationform}), 
the possible orders of nontrivial homomorphic images of $G$,
and the abelianisations of $G$ are
$$ \begin{array}{llll}
C_2\times Q_8, & \quad Q_8, & \quad C_4, & \quad C_2\times C_2\times C_2 \times C_2\times C_2, \cr
2, 120, 1920, & \quad 60, 960, & \quad 2, 10, 160, & \quad 2, 4, 8, 16, \cr
(P_3)_{\rm ab} = C_2, & \quad (P_2)_{\rm ab} = 1, & \quad (P_1)_{\rm ab} = C_2, 
& \quad K_{\rm ab}=C_2\times C_2\times C_2 \times C_2,
	\end{array} $$
and so we conclude that there is no surjective homomorphism
of the form (\ref{quatabelianisationform}) for any of these groups.
We observe, however, that the abelianisation of $K$ maps each of its five
$2C_2$ reflection orbits to a different subgroup of order $2$ 
(of which there are $15$). Thus, we can conclude that a set of reflections
which generates $K$ must contain a reflection from four of its five
	reflection orbits. In \cite{BW26}, it is observed that all such 
	sets of four reflections do generate $K$, which leads to the six
	labels for $K$ in Example \ref{Pgroups-example}.
\end{example}

The reflection groups $P_1,P_2,P_3,K$ have centre $\inpro{-1}$, 
and hence have no hidden reflections and have different collineation groups. 
In particular, we observe that
\begin{itemize}
\item The normal reflection subgroup $P_2$ of $P_3$ has reflection type $10Q_8$ 
and so does not have a split orbit, despite it not being collineation preserving. \\
Therefore, Theorem \ref{splitorbitcollineation} does not hold for quaternionic reflection groups.
\end{itemize}

There are also multiple labels for the
quaternionic reflection groups of type $O$.

\begin{example} 
\label{typeOexample}
We consider the primitive quaternionic reflection groups of type $O$ given in \cite{C80}.
In \cite{W24}, these are given by
$$
H_{120}=O_1=\langle r_1,r_2,r_3 \rangle, \quad
H_{720}=O_2=\langle r_1,r_2,r_4 \rangle, \quad
H_{1440}=O_3=\langle r_1,r_2,r_5 \rangle,
$$ 
where the first index is the order of the group, and the 
	reflections $r_1,\ldots,r_5$ are
$$ r_1=\pmat{1&0\cr0&-{1\over2}+{\sqrt{3}\over2}i}, \quad
r_2=\pmat{{1\over\sqrt{3}}i&{\sqrt{2}\over\sqrt{3}}\cr
{1\over\sqrt{6}}-{1\over\sqrt{2}}i&{1\over2}+{1\over2\sqrt{3}}i}, \quad
r_3=\pmat{{1\over2}-{1\over2\sqrt{3}}j&{\sqrt{2}\over\sqrt{3}}k\cr
{1\over\sqrt{2}}i-{1\over\sqrt{6}}k&{1\over\sqrt{3}}j}, $$
$$ r_4=\pmat{-{1\over2}+{\sqrt{3}\over2}j&0\cr0&1}, \quad
r_5=\pmat{0&-{1\over\sqrt{2}}-{1\over\sqrt{2}}k\cr-{1\over\sqrt{2}}+{1\over\sqrt{2}}k&0}.
$$
Since $H_{1440}$ has reflection type $20C_3,30C_2$, it has 
normal reflection subgroups
$$ H_{1440}=G^{(C_2,C_3)}=G^{(C_2,1)} \quad \hbox{($2$ labels)}, \qquad
H_{720}=G^{(1,C_3)}, \qquad 1=G^{(1,1)}. $$
The reflection types of $H_{720}$ and $H_{120}$ are $20C_3$ and $10C_3$,
so they have no nontrivial normal reflection subgroups.
These groups have centre $\inpro{-1}$, which is their only other normal subgroup, 
and so, in particular, $H_{720}$ is
not a collineation preserving normal subgroup of $H_{1440}$
(which does not have a split orbit).
The abelianisations of these groups are
$$ (H_{120})_{\rm ab}=1, \quad (H_{720})_{\rm ab}=1, \quad (H_{1440})_{\rm ab}=C_2.  $$
These abelianisations map $r_1,r_2,r_3,r_4$ to the identity
and $r_5$ to the element of order $2$.
\end{example}

Unlike the complex case, for a quaternionic reflection group $G$, 
it is possible to have a proper subgroup $\hat R_{a_j}$ of the reflection subgroup 
$R_{a_j}$ for a root $a_j$, with
\begin{equation}
\label{multilabelRa}
\hat R_{a_j}^G = R_{a_j}^G, \qquad \hat R_{a_j} \ne R_{a_j}, 
\end{equation}
which leads to multiple labels.

We now explain this phenomenon in terms of the conjugation action of $G$ on $R_{a_j}$.


\begin{theorem}
\label{RaStabilising}
Let $G$ be a reflection group, $R_a$ be the reflection subgroup for a root $a$, 
and $\Stab_G(R_a)=\{g\in G:gR_ag^{-1}=R_a\}$ be the stabiliser of $R_a$. Then 
\begin{enumerate}[\rm (i)]
\item If $G$ is a complex reflection group or $|R_a|=2$, 
then the stabiliser of $R_a$ pointwise stabilises $R_a$, i.e.,
$gr_{a,\xi}g^{-1}=r_{a,\xi}$, $\forall r_{a,\xi}\in R_a$, $\forall g\in\Stab_G(R_a)$.
\item If $G$ is a quaternionic reflection group and $|R_a|\ne2$, 
then there may be reflections in $R_a$ of order greater than $2$ which are not
	fixed by the stabiliser of $R_a$.
\end{enumerate}
	As an example of {\rm(ii)}, all the reflections of order $4$ in $P_1,P_2,P_3$ are conjugate,
and all the reflections of order $3$ in $H_{120},H_{720},H_{1440}$ are conjugate.
\end{theorem}

\begin{proof} If $g\in G$ conjugates the reflection $r_{a,\xi}$ to one with the 
same root, then (\ref{raxi-conjugation}) gives $ga=a\gb$, 
for some unit scalar $\gb\in K_a$, 
where $K_a$ is defined in (\ref {HaKadefn}), and
$$ gr_{a,\xi}g^{-1}=r_{ga,\xi}=r_{a\gb,\xi}=r_{a,\gb\xi\gb^{-1}},$$
which equals $r_{a,\xi}$ when $G$ is complex, since in this case $\gb\xi\gb^{-1}=\xi$,
or when $\xi=-1$, i.e., $r_{a,\xi}$ has order $2$.
For $G$ a quaternionic reflection group and $\xi\ne-1$, 
it may be that $\gb\xi\gb^{-1}\ne\xi$, which 
gives a different reflection of the same order for the root $a$.
We observe that $r_{a,\xi}$ and $r_{a,\xi}^{-1}=r_{a,\xi^{-1}}$ are reflections in $R_a$ 
of the same order, 
and that they are different if and only if their order is not $2$.

The Examples \ref{Pgroups-example} and \ref{typeOexample} give the listed 
examples, 
where $r_{a,\xi}\ne r_{a,\xi}^{-1}$ (of orders $4$ and $3$, respectively)
are conjugate by some $g\in\Stab_G(R_a)$.
\end{proof}

The geometric description of the reflection type in terms of dominos
(see Section \ref{domino-sect}) 
can be complicated for quaternionic reflection groups.

\begin{example} Let $G$ be $P_2$ or $P_3$, which have a $10Q_8$ orbit.
Let $R_a=\{r_{a,\xi}\}_{\xi\in Q_8}$ 
be the reflection subgroup for a given root $a$. 
Since this is not pointwise stabilised by $\Stab_G(R_a)$, 
it is natural to think of it as a set, rather than as a column (ordered set). 
The corresponding orbit gives an $8\times 10$ matrix $[R_{a}^{g_1},\ldots,R_a^{g_{10}}]$,
which partitions the $70$ reflections by their roots.
Now let $\hat R_a$ be the subgroup $R_{a,\inpro{i}}=\{r_{a,\xi}\}_{\xi\in\inpro{i}}$.
The orbit of this ``column of size $4$'' gives a $4\times 30$ matrix, which 
includes three columns $R_{a,\inpro{i}},R_{a,\inpro{j}},R_{a,\inpro{k}}$ for 
the root $a$, and similarly for the other roots. 
The total number of reflections in this matrix is $90=3\cdot 30$, with each 
reflection $r_{a,-1}$ appearing three times. Thus $\hat R_a=R_{a,\inpro{i}}$ should
be thought of as a domino with some caution.
\end{example}

The reflection types and attributes of the remaining primitive quaternionic
reflection groups of types $Q,R,S,T,U$ are tabulated in Table \ref{PrimitiveQuatOrbits}.
We give a quick overview.

\begin{example}
The primitive quaternionic reflection groups 
$$ Q,\, R \ \hbox{(rank $3$)}, \qquad S_1,\, S_2,\, S_3\ \hbox{(rank $4$)},\qquad
T,\,U \ \hbox{(rank $5$) }, $$
each have a single
$n_1 C_2$ orbit, and hence have no nontrivial normal reflection subgroups.
In this way, they can be thought of as being ``close to simple''. They all have
	centre $\inpro{-1}$, which gives a third normal subgroup. In the 
cases where the collineation group $G/\inpro{-1}$ is simple, we list it in
the final column of Table \ref{PrimitiveQuatOrbits}.
Moreover, all the proper normal subgroups are reflection free.
In particular, $Q$ and $U$ have a collineation preserving normal
	subgroup of index $2$ (with all the reflections being hidden),
	which is isomorphic to the collineation group (which is 
	$PUS_3(3)$ and $PSU_5(2)$, respectively).
\end{example}

We have not yet calculated all the normal reflection subgroups 
of the imprimitive quaternionic reflection groups. 
However, we have included the imprimitive group $K$ 
in Table \ref{PrimitiveQuatOrbits}, which seems to be indicative,
as for the imprimitive quaternionic reflection groups $G(n,a,b,r)$ of
\cite{W25}, for which $K=G(2,1,1,2)$, we have
$$ G(n,a,b,r)_{\rm ab}=
\begin{cases}
C_2\times C_2, & \hbox{${n\over a},{n\over b},r$ are odd}; \cr
	C_2\times C_2\times C_2\times C_2, & \hbox{${n\over a},{n\over b},r$ are even, $r\ne {2n\over ab}$}; \cr
C_2\times C_2\times C_2, & \hbox{otherwise}. \cr
\end{cases}
$$
We observe that for the particular case
$$ G(n,1,n,1)\cong_\HH G(2n,n,2), $$
we have the abelianisation given by (\ref{G(n,p,d)ab}), i.e.,
$$ G(2n,n,d)_{\rm ab} = C_2\times C_{2} \times C_{\gcd(2,n)}. $$

In all the cases calculated so far, it appears that 
the abelianisation of a quaternionic reflection group is an elementary $2$-group.

\begin{table}[H]
\caption{
\label{PrimitiveQuatOrbits}
The primitive quaternionic reflection groups $G$
(with the imprimitive group $K$ for comparision), giving 
the reflection type, abelianisation, centre, with the last 
	column giving the number of normal reflection subgroups and
	the number of normal subgroups (respectively).
All the normal subgroups are reflection free except for $G$.
For the groups $Q$ and $U$ there is index $2$ collineation
	preserving reflection free simple normal subgroup.}
\smallskip
\begin{center}
\begin{tabular}{ |  >{$}l<{$} | >{$}l<{$} | >{$}l<{$} | >{$}l<{$} | >{$}l<{$} | >{$}l<{$} |  >{$}l<{$} |  }
\hline
&&&&&&\\[-0.3cm]
\hbox{Rank} & G & \hbox{Order} & \hbox{Reflection type} & G_{\rm ab} &
	 Z(G) & \#\lhd \\[0.1cm]
\hline
&&&&&&\\[-0.4cm]
2 & O_1 & 120 & 10C_3 & 1 & \inpro{-1} & 2,\, 3\ A_5 \\
2 & O_2 & 720 & 20C_3 & 1 & \inpro{-1} & 2,\, 3\ S_6 \\
2 & O_3 & 1440 & 20C_3,30C_2 & C_2 & \inpro{-1} & 3,\, 4 \\
2 & P_1 & 320 & 10C_4 & C_2 & \inpro{-1} & 3,\, 5 \\
2 & P_2 & 1920 & 10Q_8 & 1 & \inpro{-1} & 3,\, 4 \\
2 & P_3 & 3840 & 10Q_8, 40C_2 & C_2 & \inpro{-1} & 4,\, 5 \\
&&&&&&\\[-0.4cm]
3 & Q & 12096 & 63C_2 & C_2 & \inpro{-1} & 2,\, 4\ PSU_3(3) \\
3 & R & 1209600 & 315C_2 & 1 & \inpro{-1} & 2,\, 3\ J_2 \\
&&&&&&\\[-0.4cm]
4 & S_1 & 6912 & 36C_2 & C_2 & \inpro{-1} & 2,\, 10 \\
4 & S_2 & 82944 & 72C_2 & C_2 & \inpro{-1} & 2,\, 7 \\
4 & S_3 & 3317760 & 180C_2 & 1 & \inpro{-1} & 2,\, 4  \\
&&&&&&\\[-0.4cm]
5 & T & 2592000 & 180C_2 & C_2 & \inpro{-1} & 2,\, 5 \\
	5 & U & 27371520 & 165C_2 & C_2 & \inpro{-1} & 2,\, 4\ PSU_5(2) \\ [0.1 cm]
\hline
&&&&&&\\[-0.4cm]
2 & K & 32 & 2C_2,2C_2,2C_2,2C_2,2C_2 & (C_2)^4 
& \inpro{-1} & 27,\, 68 \\ [0.1 cm]
\hline
\end{tabular}
\end{center}
\end{table}

\section{Concluding remarks}

Reflection groups $G$ are typically generated by a small number of the
reflections that they contain.
Thus, to find reflection subgroups naively, one could take a small set of
generators from the set of all reflections.
The method given here involves taking a larger, but structured subset
of generators, i.e., the reflection orbits $\hat R_{a_j}^G$ and more 
generally ``dominos'', then selecting a smaller subset of generators
(when desired). This 
gives a simple 
method to 
find the normal reflection subgroups and their properties
(Lemma \ref{normalrefslatticelemma}).

A key observation is that every reflection group is a normal collineation
preserving subgroup of its maximal reflection group, and so the it 
suffices to determine the maximal reflection groups (Theorem \ref{complexrefgpclass}). 
By combining the Theorems \ref{primitivemaximal} and \ref{imprimitivemaximal},
we have:

\begin{theorem} 
The maximal 
complex reflection groups are the primitive groups
$$ G_7,\ G_{11},\ G_{19},\ G_{23},\ G_{24},\ G_{26},\ldots,G_{37}, $$
and the imprimitive groups
\begin{align*}
\label{maximalimprimitive}
G(2a,2,2), & \qquad d=2, \quad a=2,3,4,\ldots,\cr
G(ka,k,d), & \qquad d\ge3, \quad k\divides d, \quad a=1,2,3,\ldots,
\end{align*}
with each of these groups corresponding to a unique collineation group.
\end{theorem}

Given their prominence,
one might develop an indexing for the $17$ maximal
primitive complex reflection groups $G$ which reflects their properties, e.g.,
$M_{d,|G|}^{(k_1,\ldots,k_m)}$, 
where $d$ is the rank of $G$, $|G|$ its order, and $n_1C_{k_1},\ldots,n_m C_{k_m}$ is its reflection type, i.e.,
\begin{equation}
\label{Maximalprimitive-notation}
M^{(2,3,3)}_{2,144}=G_{7}, \quad M^{(2,4,3)}_{2,576}=G_{11}, \quad 
M^{(2,3,5)}_{2,3600}=G_{19}, \quad\ldots\quad  M_{8,696729600}^{(2)}=G_{37}.
\end{equation}
The rank and order uniquely defines a maximal primitive complex reflection group.

We gave two ways to calculate all the reflection subgroups (illustrated in 
the most complicated case of $G_{11}$), 
i.e., as normal subgroups of the reflection subgroups
which are maximal reflection groups (Example \ref{reflectsubgroupsG11}
), and using ``dominos'' (Section \ref{domino-sect}
), which has the
advantage that reflection subgroups and their conjugates can be visualised.
Other results for complex reflection subgroups include:
\begin{itemize}
\item The splitting of the reflection orbit of a normal reflection subgroup
corresponds to a change of the collineation group
(Theorem \ref{splitorbitcollineation}, Corollary \ref{abeliansplitorbitscor}). 
\item The abelianisation corresponds to the reflection type (Lemma \ref{abelianisationlemma}).
\item The quotients by the normal reflection subgroups 
(which are reflection groups) were calculated (Theorem \ref{QuotientsTheorem}).
\end{itemize}

The (normal) reflection subgroups of the quaternionic reflection groups were
also investigated up to the point where the differences in their structural 
properties from those of the complex reflection groups became apparent.
In particular:
\begin{itemize}
\item Every quaternionic reflection group appears to be maximal.
\item The abelianisation has a weaker 
correspondence with the reflection type.
\item The reflections from one reflection orbit can generate 
	reflections on another orbit, and so 
it is not necessary to pick a generating reflection from each orbit.
\item The stabiliser of a reflection subgroup $R_a$ may not pointwise
stabilise it.
\end{itemize}
Heuristically, quaternionic reflection groups are less abelian than
complex reflection groups (with $K$ being an exception), and so 
they have fewer normal subgroups. The Lemma \ref{normalrefslatticelemma} 
provides a systematic way to calculate all the normal reflection subgroups
of the quaternionic reflection groups and to quantify the various
subsets of reflections that generate them (see \cite{W26b}).

One could also take the view that the multiple labels
(generating sets of reflections) that would give different 
normal reflection subgroups in the complex case, but give just a single
normal reflection subgroup in the quaternionic case (because of the
high noncommutatively of matrices over the quaternions), results in
there being
fewer normal reflection subgroups for quaternionic reflection groups
(than might be expected).

\vfil\eject
\bibliographystyle{alpha}
\bibliography{references}
\nocite{*}

\end{document}